\documentclass[12pt,a4paper]{amsart}

\usepackage{amssymb,amscd,amsthm,amsxtra,mathtools,amsmath}
\usepackage[margin=1.5cm]{geometry}
\usepackage[alphabetic]{amsrefs}
\usepackage{hyperref}
\usepackage{mathtools}
\usepackage{xcolor}
\usepackage{graphicx}
\usepackage{wrapfig}
\usepackage{float}
\usepackage{enumitem}
\usepackage{mathrsfs}
\usepackage{comment}

\numberwithin{equation}{section}
\allowdisplaybreaks[4]

\newtheorem{theorem}{Theorem}[section]
\newtheorem{proposition}[theorem]{Proposition}
\newtheorem{corollary}[theorem]{Corollary}
\newtheorem{lemma}[theorem]{Lemma}
\theoremstyle{definition}
\newtheorem{definition}[theorem]{Definition}
\theoremstyle{remark}
\newtheorem{remark}[theorem]{Remark}

\newcommand{\dist}{\operatorname{dist}}
\newcommand{\R}{\mathbb{R}}
\newcommand{\N}{\mathbb{N}}
\newcommand{\eps}{\varepsilon}

\newcommand{\abs}[1]{\left\lvert#1\right\rvert}

\newcommand{\dotprod}[2]{\left\langle#1,#2\right\rangle}

\newcommand{\restr}[2]{{\left.\kern-\nulldelimiterspace#1\vphantom{\big|}\right|_{#2}}}
\newcommand{\determinant}{\operatorname{det}}

\newcommand{\smallo}{\mathchoice{{\scriptstyle\mathcal{O}}}{{\scriptstyle\mathcal{O}}}{{\scriptscriptstyle\mathcal{O}}}{\scalebox{.7}{$\scriptscriptstyle\mathcal{O}$}}}

\renewcommand{\leq}{\leqslant}
\renewcommand{\le}{\leqslant}
\renewcommand{\geq}{\geqslant}
\renewcommand{\ge}{\geqslant}
\renewcommand{\epsilon}{\varepsilon}

\makeatletter
\newcommand{\prooflabel}[2]{%
  \def\@currentlabelname{#2}%
  \label{#1}%
}
\makeatother

\newcounter{constctr}
\newcounter{appconstctr}

\makeatletter
\newcommand{\newconst}[1]{%
  \stepcounter{constctr}%
  \expandafter\xdef\csname const@#1\endcsname{C_{\theconstctr}}%
}

\newcommand{\const}[1]{%
  \csname const@#1\endcsname%
}

\newcommand{\newconstapp}[1]{%
  \stepcounter{appconstctr}%
  \expandafter\xdef\csname constapp@#1\endcsname{C_{A\theappconstctr}}%
}

\newcommand{\constapp}[1]{%
  \csname constapp@#1\endcsname%
}

\makeatother

\newconst{CONST::UniformRegularity}
\newconst{CONST::ExistenceRegularity}
\newconst{CONST::ForcingTermCalphaBound}
\newconst{CONST::LinearSchauder}
\newconst{CONST::rhoBoundaryDistance}
\newconst{CONST::PiecewiseRegularityToGlobal}
\newconst{CONST::PiecewiseLipschitzToGlobalLipschitz}
\newconst{CONST::PiecewiseRegularityToGlobalWithTime}
\newconst{CONST::RuHolderEstimate}
\newconst{CONST::TTopologicalIsomorphismUpper}
\newconst{CONST::TTopologicalIsomorphismLower}
\newconst{CONST::AijEllipticUpper}
\newconst{CONST::AijEllipticLower}
\newconst{CONST::AijHolder}
\newconst{CONST::AastijEllipticUpper}
\newconst{CONST::AastijEllipticLower}
\newconst{CONST::AastijHolder}
\newconst{CONST::BjTRdjuHolder}
\newconst{CONST::LaplaceBeltramiApproximation}
\newconstapp{CONST::HolderEmbeddingA}
\newconstapp{CONST::HolderEmbeddingB}
\newconstapp{CONST::HolderInterpolation}
\newconstapp{CONST::HolderProductRule}
\newconstapp{CONST::Quasiconvexity}

\mathtoolsset{showonlyrefs=true}
\begin{document}
\title[Wildfire in a narrow gully]{Wildfire in a narrow gully: a geometric reduction approach}
\author[L. De Gaspari]{Lorenzo De Gaspari}
\address{L. D. G., 
Department of Mathematics and Statistics,
University of Western Australia,
35~Stirling Highway, WA 6009 Crawley, Australia. }

\email{lorenzo.degaspari@research.uwa.edu.au}

\author[S. Dipierro]{Serena Dipierro}
\address{S. D., 
Department of Mathematics and Statistics,
University of Western Australia,
35~Stirling Highway, WA 6009 Crawley, Australia. }

\email{serena.dipierro@uwa.edu.au}

\author[E. Valdinoci]{Enrico Valdinoci}
\address{E. V., 
Department of Mathematics and Statistics,
University of Western Australia,
35~Stirling Highway, WA 6009 Crawley, Australia. }

\email{enrico.valdinoci@uwa.edu.au}

\thanks{SD and EV are members of the Australian Mathematical Society (AustMS). This work has been supported by Australian Research Council DP250101080, FT230100333, and FL190100081. LDG is supported by a Scholarship for International Research
Fees at the University of Western Australia.}

\subjclass[2020]{35K10, 45H05, 53B21.}

\keywords{evolution equations, bushfire models, regularity theory, geometric analysis.}

\begin{abstract}
    We consider a bushfire model in a gully. The biological scenario under consideration involves flammable fuel (trees, leaves, etc.) concentrated within the gully, surrounded by rocky hillslopes containing little or no burnable material. The mathematical formulation of the problem is a nonlocal evolution equation of parabolic type. The nonlocality arises from an ignition mechanism that becomes active when the temperature reaches the ignition threshold and is modeled via a kernel interaction with limitrophe areas.

    The rocky hillsides of the gully impose insulating boundary conditions of Neumann type, while the entrance and exit of the gully are modeled by (not necessarily homogeneous) Dirichlet boundary data, corresponding to prescribed environmental temperatures on the gully's terminals.

    Given the geometry of the domain, in the asymptotic regime of a narrow gully the model undergoes a dimensional reduction and can be analyzed through a geometric equation posed along the (not necessarily straight) axis of the gully. The reduced equation is supplemented with inner and outer Dirichlet boundary conditions (with no Neumann condition remaining in the limit).

    The analysis relies on the use of Fermi coordinates to capture the potentially curvilinear geometry of the gully, as well as on parabolic estimates tailored to the specific equation in order to properly account for the ignition interactions. These estimates are delicate, as the domain degenerates and the boundary conditions vary in the limit. To overcome these difficulties, we develop a bespoke reflection technique that provides uniform bounds and enables the passage to the limit.
\end{abstract}
\maketitle
\section{Introduction}

\subsection{Documented history.}

\begin{wrapfigure}{r}{0.45\textwidth} 
    \centering
    {
    \includegraphics[width=0.4\textwidth]{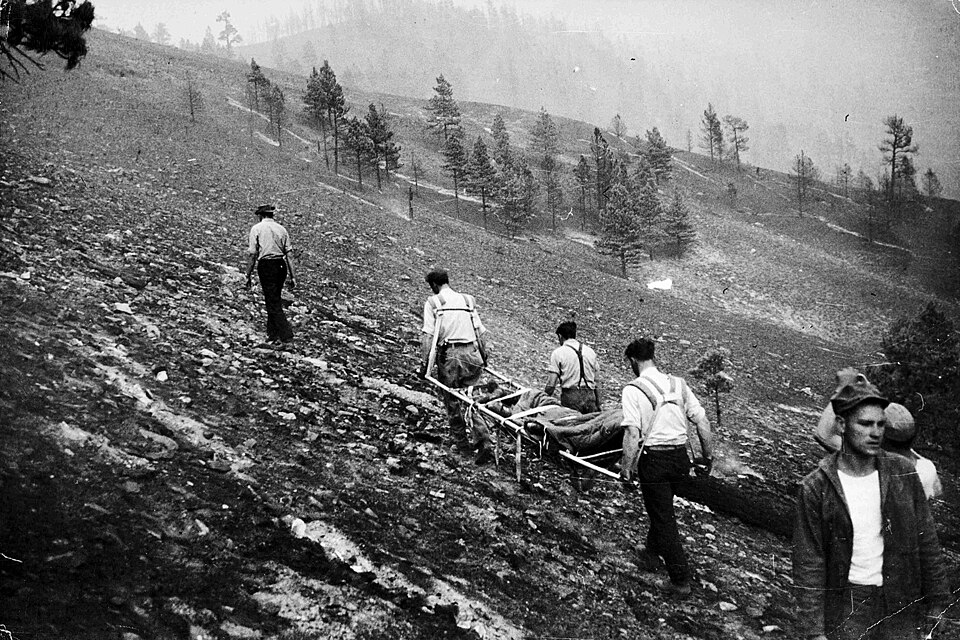}
    }
    \caption{\footnotesize\sl North slope of Mann Gulch, August 6, 1949. Retrieval of victim's bodies~\cite{USFS_MannGulch_1949}.}\label{BLs1-24rE}
\end{wrapfigure}

Wildfires have a documented history of occurring in gullies, ravines, and canyons. These fires are particularly dangerous because of how fire and wind interact with the terrain. The specific configuration of a location can enhance natural phenomena, intensifying the fire's behavior. For example, gullies can act like natural chimneys, funneling heat and flames upward and accelerating the fire's spread and speed. Additionally, the rising hot air in steep gullies can draw in oxygen from below, further fueling the fire (see e.g.~\cites{10.1071/WF08041, HOLSINGER201659, 10.1071/WF17147}). This combination of factors can make these fires exceptionally difficult to contain.

For instance, the Mann Gulch Fire, which took place in Montana in 1949, is considered one of the most tragic
wildfire disasters in history, and twelve smokejumpers and a ground-based firefighter were fatally burned,
see e.g.~\cite{Rothermel} and Figure~\ref{BLs1-24rE}. 
Other well-documented examples include the South Canyon Fire in Colorado in 1994 (see~\cite{Butler}), the Price Canyon Fire Entrapment in Utah in 2002 (see~\cite{USREP}), and more recent events depicted in Figures~\ref{BLs1-24rEzf} and~\ref{BLs1-24rEf}.

The objective of this paper is to describe a simple
mathematical equation for the fire front propagation in a gully, as a specific case
of a model introduced in~\cite{MR4772545}, and relate its solution
to a lower-dimensional problem.

\subsection{Topographical scenario}

\begin{wrapfigure}{l}{0.5\textwidth} 
    \centering
    {
    \includegraphics[width=0.45\textwidth]{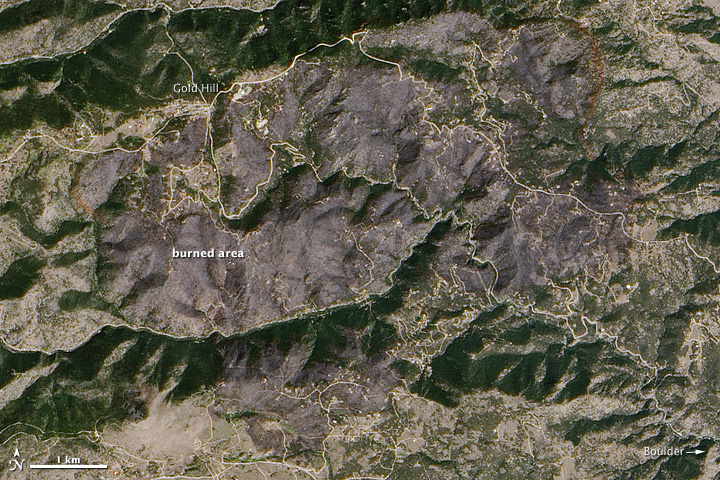}
    }
    \caption{\footnotesize\sl
        The Fourmile Canyon Fire had burned more than 6,000 acres in 2010.
        This image showcasing burn scars was taken by
        the Advanced Land Imager on NASA's Earth Observing-1~\cite{NASA_FourmileCanyonFire_2010}.}\label{BLs1-24rEzf}
\end{wrapfigure}

We stress that the topographical scenario considered in this paper differs from the one that has already been extensively studied in the literature. Indeed, most existing studies focus on bushfire spread within a valley forming the axis of a canyon, under the assumption that both the valley floor and the canyon walls consist of flammable material, such as grasses and trees. In that setting, it is well known that the slope of the walls, and possibly that of the valley floor, can significantly enhance fire-front propagation (see, e.g.,~\cite{ViegasPita2004}).

In contrast, the landscape considered in this paper is characterized by a {\em fuel discontinuity}. Specifically, flammable vegetation is concentrated within a gully or drainage line, while the surrounding rocky hillslopes contain little or no burnable material. This configuration occurs in several concrete and practically relevant settings. Examples include Mediterranean-type ecosystems such as the California chaparral, regions of the Mediterranean Basin, and rocky ranges in Australia, where dense shrubs, grasses, and small trees tend to accumulate in gullies that collect runoff water, while adjacent ridges and thin-soiled hills remain sparsely vegetated and often rocky.

\begin{wrapfigure}{r}{0.5\textwidth} 
    \centering
    {
    \includegraphics[width=0.45\textwidth]{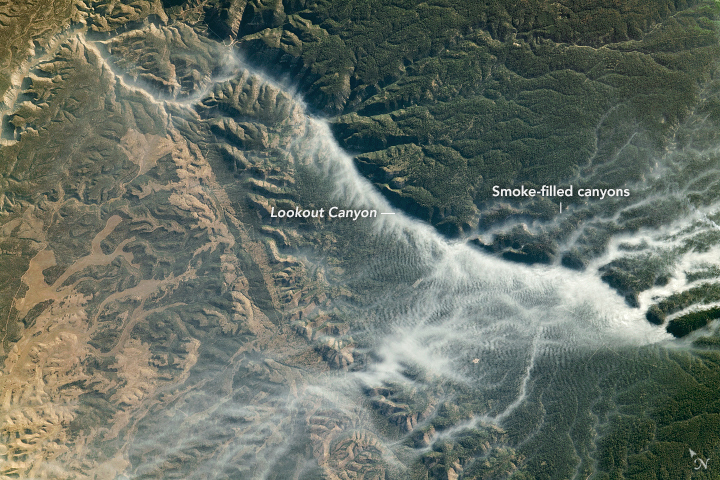}
    }
    \caption{\footnotesize\sl Smoke filled canyons, Arizona.
    The image represents the northern rim of the Grand Canyon in 2019.
    A wildfire burnt more than 19,000 acres. 
    The image was taken almost a month after the initial incident, as the land was still burning, by
    an astronaut onboard the International Space Station~\cite{NASA_SmokeFilledCanyons_2019}.}\label{BLs1-24rEf}
\end{wrapfigure}

Similar patterns are observed in desert mountain environments such as the Sonoran Desert and the Flinders Ranges. In these regions, grasses may accumulate seasonally in drainage channels following rainfall events, whereas the surrounding rocky hillslopes support only minimal vegetation. Analogous configurations can also be found in certain volcanic landscapes, where vegetation colonizes depressions in which soil accumulates, while surrounding lava fields or rocky uplands remain largely barren.

See, for instance,~\cites{Austock000215558, Austock000221021, Austock000215551, Austock000225918} for aerial pictures of this kind of topographical scenario.

\subsection{Mathematical description of a gully.}

Throughout this paper, we will denote by~$\mathcal{S}$ a compact, connected, and orientable hypersurface with boundary embedded into~$\R^n$. We assume that~$\mathcal{S}$ is of class~$C^{3,\alpha}$ for some~$\alpha\in(0,1)$, and the symbol~$\alpha$ will denote this regularity exponent unless otherwise specified. A unit normal field to~$\mathcal{S}$ induced by its orientation is denoted by~$\nu$.

We model the gully geometry as a tubular neighborhood of~$\mathcal{S}$ with radius~$L$, given by
\begin{equation}\label{EQ::OmegaLDefinition}
    \Omega_L \coloneq \bigcup_{s\in(-L,L)}\mathcal{S}(s)=\Big\{ x+s\nu(x),\; x\in\mathcal{S} {\mbox{ and }}s\in(-L,L) \Big\},
\end{equation}
with
\begin{equation}
    \mathcal{S}(s):=\Big\{ x+s\nu(x),\; x\in\mathcal{S}\Big\}.
\end{equation}

We denote with~${\{\kappa_j(x)\}}_{j}$ the principal curvatures of~$\mathcal{S}$ at each point~$x \in \mathcal{S}$ and we define
\begin{equation}\label{EQ::L0Definition}
    L_0 \coloneq \inf_{\substack{x \in \mathcal{S}\\1\leq j\leq n-1}}\frac{1}{\abs{\kappa_j(x)}},
\end{equation}
that is,~$L_0$ is the smallest curvature radius of~$\mathcal{S}$, which may also be equal to~$+\infty$, in which case~$\mathcal{S}$ is contained in a hyperplane.

\begin{figure}
    \centering{
        \includegraphics[width=0.8\textwidth]{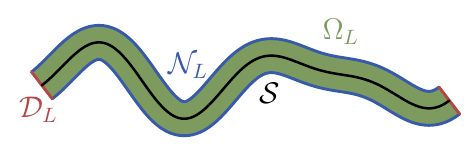}
        \caption{Sketch of the geometry of a gully in dimension~$n=2$.}\label{FIG::Gully}
    }
\end{figure}

\subsection{Mathematical description of bushfire propagation in a gully.}

We now formulate an adaptation to this geometry of the general equation proposed in~\cite{MR4772545} (see also~\cites{MR4861891, MR4968074} for an existence theory for this type of equations). For~$T \in (0,+\infty)$ and~$X \subset \R^n$ we will denote with~$X^T$ the time cylinder~$X \times (0,T]$, and with~$X^\infty$ the set~$X \times (0,+\infty)$. We consider the equation posed in the time cylinder~$\Omega_L^{\infty}$ in the form 
\begin{equation}\label{EQ::MainEquation}
    \partial_{t} u(x,t) = \Delta u(x,t) + \int_{\Omega_L} K_L(x,y) {(u(y, t)-\vartheta)}^+\, dy + \psi\big(x,t,u(x,t),\nabla u(x,t)\big),
\end{equation}
with the ``positive part'' notation~$r^+ \coloneq \max\{r,0\}$.

We assume that~$K_L$ is non-negative and that there exists~$C_L > 0$ such that, for every~$x,x^\prime \in \Omega_L$,
\begin{equation}\label{EQ::KernelAssumptionL1}
    \int_{\Omega_L} K_L(x,y)\,dy \leq C_L,
\end{equation}
and
\begin{equation}\label{EQ::KernelAssumptionHolder}
    \int_{\Omega_L} \abs{K_L(x,y) - K_L(x^\prime, y)}\, dy \leq C_L \abs{x-x^\prime}^\alpha.
\end{equation}

Also, we suppose that~$\psi = \psi(x,t,s,p)$ is a non-negative function and that there exists~$C_\psi >0$ such that, for every~$(x,t),(x^\prime,t^\prime) \in \Omega_L^\infty$,~$s,s^\prime \in \R$, and~$p,p^\prime \in \R^n$,
\begin{equation}\label{EQ::psiAssumptionZero}
    \psi(x,t,s,0) = 0
\end{equation}
and
\begin{equation}\label{EQ::psiAssumptionLipschitz}
    \left|\psi(x,t,s,p) - \psi(x^\prime, t^\prime, s^\prime, p^\prime)\right| \leq C_\psi \left(\abs{x-x^\prime}^{\alpha} + \abs{t-t^\prime}^{\frac{\alpha}{2}} + \abs{s-s^\prime} + \abs{p-p^\prime}\right).
\end{equation}

As detailed in~\cite{MR4772545}, equation~\eqref{EQ::MainEquation} models the environmental temperature evolution under diffusion and a combustion mechanism driven by the ignition threshold~$\vartheta\in\R$ and an interaction kernel~$K_L$. Replacing~$u$ with~$u+\vartheta$, we may assume without loss of generality that~$\vartheta=0$.

Equation~\eqref{EQ::MainEquation} is complemented with boundary conditions. The hillsides are modeled as perfectly insulating, and thus satisfy homogeneous Neumann (zero-flux) conditions. We take~$\mathcal{N}_L$ to be the relative interior of~$\mathcal{S}(L)\cup\mathcal{S}(-L)$, corresponding to the hillsides of the gully, and impose
\begin{equation}\label{EQ::MainNeumann}
    \partial_\nu u(x,t) = 0\text{ for all }(x,t) \in \mathcal{N}_L^\infty.
\end{equation}
The rest of the boundary of~$\Omega_L$ is provided with Dirichlet data. As customary in parabolic problems, we treat these conditions unitedly with the initial conditions, on subsets of the parabolic boundary of~$\Omega_L^T$. The value of the solution is prescribed on regions of the type
\begin{equation}
    \mathcal{P}_L^{(T)} \coloneq \mathcal{D}_L^T \cup \left(\overline{\Omega}_L \times \{0\}\right), 
\end{equation}
with~$\mathcal{D}_L \coloneq \partial\Omega_L \smallsetminus \mathcal{N}_L$. Namely, for a given initial/boundary datum~$g_L$, we ask that
\begin{equation}\label{EQ::MainDirichlet}
    u(x,t) = g_L(x,t) \text{ for all } (x,t) \in \mathcal{P}_L^{(\infty)}.
\end{equation}

\subsection{Dimensional reduction.}
It is natural to seek a reduction of the model to a lower dimensional problem, because the gully is a small neighborhood of a codimension one object. Let~$u_\eps$ denote the solution in~$\Omega_\eps$ for a small parameter~$\eps$. We define for~$x\in\mathcal{S}$ and~$t \in [0,+\infty)$ the transverse average 
\begin{equation}\label{EQ::UaverageDefinition}
    U_\eps(x,t) \coloneq \frac{1}{2\eps}\int_{-\eps}^\eps u_\eps(x+s\nu(x),t)\,ds.
\end{equation}

Our goal is to compare~$U_\eps$ with the solution of a geometric equation posed on~$\mathcal{S}$. This reduction provides technical simplifications by removing one spatial dimension and replacing the mixed boundary conditions with Dirichlet data only. Accordingly, we consider on~$\mathcal{S}$ the equation
\begin{equation}\label{EQ::MainEquationS}
    \partial_t U(x,t) = \Delta_{\mathcal{S}}U(x,t) + \int_{\mathcal{S}} K^\ast(x,y)U^+(y,t)\,d\mathcal{H}^{n-1}(y) + \psi\big(x,t,U(x,t),\nabla_T U(x,t)\big),
\end{equation}
where~$\nabla_T \coloneq \nabla - (\nu \cdot \nabla)\nu$ is the tangential gradient along~$\mathcal{S}$ and~$\Delta_{\mathcal{S}}$ denotes the Laplace-Beltrami operator\footnote{For example,
if~$\mathcal{S}$ lies on the hyperplane~$\R^{n-1}\times\{0\}$, we have
\[\Delta_\mathcal{S}=\frac{\partial^2}{\partial x_1^2}+\cdots+\frac{\partial^2}{\partial x_{n-1}^2}.\]
See e.g.~\cite{MR4784613} for the basics of the Laplace-Beltrami operator.} along the hypersurface~$\mathcal{S}$.

The interaction kernel~$K^\ast$ in~\eqref{EQ::MainEquationS} is defined by
\begin{equation}\label{EQ::KastDefinition}
    K^\ast(x,y) \coloneq \lim_{\eps \to 0} \left(\frac{1}{2\eps} \int_{-\eps}^\eps\int_{-\eps}^\eps K_\eps(x+s\nu(x), y+\sigma\nu(y))\,ds\,d\sigma\right),
\end{equation}
where the limit is taken in the~$L^1$ sense.

This problem is complemented by the Dirichlet boundary condition
\begin{equation}\label{EQ::MainDirichletS}
    U(x,t) = g^\ast(x,t)\coloneq \lim_{\eps \to 0}g_\eps(x,t) \text{ for all } (x,t) \in {(\partial\mathcal{S})}^\infty \cup \left(\mathcal{S} \times \{0\}\right). 
\end{equation}

The dimensional reduction procedure requires precise regularity estimates for solutions of~\eqref{EQ::MainEquation}. To this end, we use weighted Hölder spaces~$\mathcal{H}_a$ and~$\mathcal{H}_a^{(b)}$ defined through parabolic Hölder norms weighted by the distance from the Dirichlet boundary, respectively endowed with the norms~$\abs{\cdot}_a$ and~$\abs{\cdot}_a^{(b)}$. We postpone to Appendix~\ref{APP::WeightedHolder} the definition and some comments about basic properties of these spaces.

\subsection{Description of the results.}\label{SSEC::Results}

We now state the main result of this paper.

\begin{theorem}\label{THM::thm1}
    Let~$\mathcal{S}$ be a compact, connected, and orientable~$C^{3,\alpha}$ hypersurface with boundary embedded into~$\R^n$. Let~$L_0$ be as in~\eqref{EQ::L0Definition} and~$L \in (0,L_0)$.
    
    Suppose that the family~${\{K_\eps\}}_{\eps \in (0,L)}$ satisfies~\eqref{EQ::KernelAssumptionL1} and~\eqref{EQ::KernelAssumptionHolder} for a constant~$C_L$ independent of~$\eps$, and that~$\psi$ satisfies~\eqref{EQ::psiAssumptionZero} and~\eqref{EQ::psiAssumptionLipschitz} for a constant~$C_\psi$.
    
    Furthermore, let~$\lambda \in (0,1)$ and assume that for every~$T>0$ there exists a constant~$G_T > 0$ such that, for every~$\eps \in (0,L)$, we have that~$\abs{g_\eps}_{\lambda; \mathcal{P}_\eps^T} \leq G_T$.
    
    Then, there exists~$U \in C\left(\overline{\mathcal{S}^\infty}\right)$ with
    \begin{equation}\label{EQ::UisCstar}
        \nabla U,D^2 U, \partial_t U \in C\left(\operatorname{int}(\mathcal{S}^\infty)\right),
    \end{equation} 
    which satisfies~\eqref{EQ::MainEquationS} and~\eqref{EQ::MainDirichletS}.

    Moreover, let~$u_\eps$ denote the solution of~\eqref{EQ::MainEquation},~\eqref{EQ::MainNeumann} and~\eqref{EQ::MainDirichlet}, and~$U_\eps$ be as in~\eqref{EQ::UaverageDefinition}. Then, for every infinitesiaml sequence~${\{\eps_k\}}_k \subset (0,L)$ and every~$T > 0$, there exists a subsequence, still denoted~$\eps_k$, such that
    \begin{equation}\label{EQ::UkConvergesToU}
        \begin{aligned}
            &U_{\eps_k} \to U,&&\text{uniformly in }\overline{\mathcal{S}^T},\\
            &\nabla_T U_{\eps_k} \to \nabla_T U,&&\text{locally uniformly in }{\left(\operatorname{int}(\mathcal{S})\right)}^T,\\
            &\nabla_T^2 U_{\eps_k} \to \nabla_T^2 U,&&\text{locally uniformly in }{\left(\operatorname{int}(\mathcal{S})\right)}^T,\\
            &\partial_t U_{\eps_k} \to \partial_t U,&&\text{locally uniformly in }{\left(\operatorname{int}(\mathcal{S})\right)}^T,
        \end{aligned}
    \end{equation}
    as~$k \to +\infty$.
\end{theorem}

Our proof of Theorem~\ref{THM::thm1} is based on asymptotic estimates for the terms in equation~\eqref{EQ::MainEquation}. As previously mentioned, this requires developing some strong regularity results for solutions of the equation. Namely, we require local~$C^{2,a}$ regularity and global~$C^{0,a}$ regularity of~$u_\eps$ in order to be able to pass to the limit and obtain~\eqref{EQ::UkConvergesToU}. This is the object of our next result.

\begin{theorem}\label{THM::epsIndependentRegularity}
    Let~$\mathcal{S}$ be a compact, connected, and orientable~$C^{3,\alpha}$ hypersurface with boundary embedded into~$\R^n$. Let~$L_0$ be as in~\eqref{EQ::L0Definition} and~$L \in (0,L_0)$.
    
    Suppose that the family~${\{K_\eps\}}_{\eps \in (0,L)}$ satisfies~\eqref{EQ::KernelAssumptionL1} and~\eqref{EQ::KernelAssumptionHolder} for a constant~$C_L$ independent of~$\eps$, and that~$\psi$ satisfies~\eqref{EQ::psiAssumptionZero} and~\eqref{EQ::psiAssumptionLipschitz} for a constant~$C_\psi$.
    
    Then, there exists a constant~$\lambda_2 \in (0,1)$, which depends only on~$n$ and~$\mathcal{S}$, such that, if~$\lambda \in (0,\lambda_2)$ and, for every~$\eps \in (0,L)$ and~$T>0$, there holds~$g_\eps \in \mathcal{H}_\lambda(\mathcal{P}_\eps^{(T)})$, then problem~\eqref{EQ::MainEquation},~\eqref{EQ::MainNeumann} and~\eqref{EQ::MainDirichlet} admits a unique classical solution~$u$.
    
    Moreover, for every~$\eps \in (0,L)$ and~$T>0$, there exists a constant~$\const{CONST::UniformRegularity} > 0$, which depends only on~$n$,~$\mathcal{S}$,~$\alpha$,~$L$,~$T$,~$\lambda$,~$C_L$ and~$C_\psi$, such that
    \begin{equation}\label{EQ::NonlinearSchauderEpsilon}
        \abs{u_\eps}_{2+\alpha; \Omega_\eps^T}^{(-\lambda)} \leq \const{CONST::UniformRegularity}\left(1+\abs{g_\eps}_{\lambda; \mathcal{P}_\eps^{(T)}}\right).
    \end{equation}
\end{theorem}

The proof that we provide for Theorem~\ref{THM::epsIndependentRegularity} is based upon an extension by even reflection and dilation of the solution in~$\Omega_\eps$. The reflection argument is a generalization of the periodic extension of a function defined on an interval. To make this generalization, we use Fermi coordinates (see Section~\ref{APP::FermiCoordinates}) to locally flatten the differential structure of~$\mathcal{S}$ and~$\Omega_L$.

\subsection{Some further remarks and notation.}

We gather here some further remarks and introduce some notation that we use throughout the paper.

\begin{remark}
    In our model, the domain~$\Omega_L$ embodies a long and narrow gully (in the previously described real world scenario, the spatial dimension is~$n = 2$). The hypersurface~$\mathcal{S}$ (which is a curve when~$n=2$) describes the axis of a deep valley (e.g., originally formed by running water that has now disappeared). The set~$\mathcal{N}_L$ constitutes the hillsides of the gully. The remaining part of the boundary of~$\Omega_L$, namely~$\mathcal{D}_L$, can be considered as the ``entrance'' or ``exit'' of the gully. See also Figure~\ref{FIG::Gully} for a representation of this kind of domain in dimension~$n=2$.
\end{remark}

\begin{remark}
    Roughly speaking, assumptions~\eqref{EQ::KernelAssumptionL1} and~\eqref{EQ::KernelAssumptionHolder} serve as an integral analogue of Hölder regularity conditions posed on the coefficients of equation~\eqref{EQ::MainEquation}. In fact,~\eqref{EQ::KernelAssumptionHolder} should be compared to~\eqref{EQ::psiAssumptionLipschitz}. Such assumptions will play a pivotal role in the study of the regularity of solutions of~\eqref{EQ::MainEquation} in view of proving Theorem~\ref{THM::epsIndependentRegularity}.
\end{remark}

\begin{remark}
    The notion of solution that we use is that of classical solution. Namely, for~$u_\eps$ to be considered a solution of~\eqref{EQ::MainEquation},~\eqref{EQ::MainNeumann} and~\eqref{EQ::MainDirichlet}, we require it to belong to the space
    \begin{equation}\label{EQ::CstarDefinition}
        C^\ast\left(\Omega_L^T\right) \coloneq \left\{u \in C\left(\overline{\Omega}_L^T\right) \colon \nabla u, D^2 u, \partial_t u \in C\left(\Omega_L^T \cup \mathcal{N}_L^T\right)\right\},
    \end{equation}
    and to satisfy the equation and the boundary conditions pointwisely.
\end{remark}

For brevity, throughout the paper we adopt the following notation
\begin{equation}\label{EQ::ForcingTerm}
    f_u^{(\eps)}(x,t) \coloneq \int_{\Omega_L} K_\eps(x,y) u^+(y, t)\, dy + \psi\big(x,t,u(x,t),\nabla u(x,t)\big),
\end{equation}
and, whenever no ambiguity may arise, we may omit the dependence from~$\eps$ of~$K_\eps$,~$g_\eps$, and~$f_u^{(\eps)}$, denoting them respectively as~$K$,~$g$, and~$f_u$.

We shall also refer to the system~\eqref{EQ::MainEquation},~\eqref{EQ::MainNeumann} and~\eqref{EQ::MainDirichlet} more compactly as
\begin{equation}\label{EQ::MainProblemEpsilon}
    \begin{cases}
        \partial_{t} u(x,t) = \Delta u(x,t) + f_u^{(\eps)}(x,t)&\text{for all }(x,t) \in \Omega_\eps^\infty,\\
        \partial_{\nu} u(x,t) = 0&\text{for all }(x,t) \in \mathcal{N}_\eps^\infty,\\
        u(x,t) = g_\eps(x,t)&\text{for all }(x,t) \in \mathcal{P}_\eps^{(\infty)}.
    \end{cases}
\end{equation}

Similarly, we consider the dimensionally reduced problem
\begin{equation}\label{EQ::MainProblemS}
    \begin{cases}
        \partial_{t} U(x,t) = \Delta_\mathcal{S} U(x,t) + f_{U}^{\ast}(x,t)&\text{for all }(x,t) \in {\left(\operatorname{int}(\mathcal{S})\right)}^T,\\
        U(x,t) = g^\ast(x,t)&\text{for all }(x,t) \in {(\partial\mathcal{S})}^\infty \cup \left(\mathcal{S} \times \{0\}\right),
    \end{cases}
\end{equation}
with
\[f_U^{\ast}(x,t) \coloneq \int_{\mathcal{S}} K^\ast(x,y)U^+(y,t)\,dy + \psi\big(x,t,U(x,t),\nabla_T U(x,t)\big).\]

Throughout this paper we adopt the convention~$\N = \{0,1,2,\dots\}$, and we denote with~$\lfloor \cdot \rfloor$ the floor function, that is, for any~$x \in \R$, we have
\[\lfloor x \rfloor \coloneq \max\left\{m \in \mathbb{Z} \colon m \leq x\right\}.\]
We use the symbol~$\delta$ with two subscripted indices to denote the Kronecker delta symbol, that is,
\begin{equation}
    \delta_{ij} \coloneq \left\{\begin{aligned}
        &1&\text{if }i=j,\\
        &0&\text{otherwise.} 
    \end{aligned}\right.
\end{equation}
\subsection{Organization of the paper.}

The rest of this paper is organized as follows. Sections~\ref{SEC::Existence} and~\ref{SEC::UniformRegularity} contain the proof of Theorem~\ref{THM::epsIndependentRegularity}, respectively covering the existence of solutions and the uniform regularity estimates, Section~\ref{SEC::DimensionalReduction} is devoted to the proof of Theorem~\ref{THM::thm1}. Some final comments are given in Section~\ref{SEC::Conclusion}. The paper ends with three appendices, in which we recall some well known facts about weighted Hölder spaces (Appendix~\ref{APP::WeightedHolder}) and about Fermi coordinates in differential geometry (Appendix~\ref{APP::FermiCoordinates}). In Appendix~\ref{APP::TechnicalProofs} we collect some longer proofs of technical results contained in this paper.

\section{Existence of classical solutions to~\eqref{EQ::MainProblemEpsilon}}\label{SEC::Existence}
In this section we discuss the existence of classical solutions to problem~\eqref{EQ::MainProblemEpsilon} in a given domain~$\Omega_L$. For a lighter notation, throughout the section we omit the dependence on~$L$ of $f_u^{(L)}$,~$K_L$, and~$g_L$, and we rewrite the problem as
\begin{equation}\label{EQ::MainProblem}
    \begin{cases}
        \partial_{t} u(x,t) = \Delta u(x,t) + f_u(x,t)&\text{for all }(x,t) \in \Omega_L^\infty,\\
        \partial_{\nu} u(x,t) = 0&\text{for all }(x,t) \in \mathcal{N}_L^\infty,\\
        u(x,t) = g(x,t)&\text{for all }(x,t) \in \mathcal{P}_L^{(\infty)}.
    \end{cases}
\end{equation}

Our method of proof involves finding the solution~$u$ as a fixed point of an appropriate nonlinear operator, following a somewhat similar procedure to~\cite{MR4968074}, although in different functional spaces. Some of the auxiliary results that we present here will also be crucial to the proof of Theorem~\ref{THM::epsIndependentRegularity} in Section~\ref{SEC::UniformRegularity}. Our existence result can be stated as follows.

\begin{proposition}\label{PROP::ExistenceRegularity}
    Let~$\mathcal{S}$ be a compact, connected, and orientable~$C^{3,\alpha}$ hypersurface with boundary embedded into~$\R^n$.  Let~$L_0$ be as in~\eqref{EQ::L0Definition} and~$L \in (0,L_0)$.
    
    Furthermore, assume that~$K$ satisfies~\eqref{EQ::KernelAssumptionL1} and~\eqref{EQ::KernelAssumptionHolder} for a constant~$C_L$, and that~$\psi$ satisfies~\eqref{EQ::psiAssumptionZero} and~\eqref{EQ::psiAssumptionLipschitz} for a constant~$C_\psi$.
    
    Then, there exists a constant~$\lambda_2 \in (0,1)$, which depends only on~$n$ and~$\mathcal{S}$, such that, if~$\lambda \in (0,\lambda_2)$ and, for every~$T>0$ there holds~$g \in \mathcal{H}_\lambda\left(\mathcal{P}_L^{(T)}\right)$, then problem~\eqref{EQ::MainProblem} admits a unique classical solution~$u \in C^\ast(\Omega_L^\infty)$.
    
    Moreover, for every~$T>0$, there exists a constant~$\const{CONST::ExistenceRegularity} > 0$, which depends only on~$n$,~$\mathcal{S}$,~$\alpha$,~$L$,~$T$,~$\lambda$, $C_L$ and~$C_\psi$, such that
    \begin{equation}\label{EQ::NonlinearSchauder}
        \abs{u}_{2+\alpha; \Omega_L^T}^{(-\lambda)} \leq \const{CONST::ExistenceRegularity}\left(1+\abs{g}_{\lambda; \mathcal{P}_L^{(T)}}\right).
    \end{equation}
\end{proposition}
 
\begin{remark}
    Proposition~\ref{PROP::ExistenceRegularity} is strictly weaker than Theorem~\ref{THM::epsIndependentRegularity}, as the uniform estimate~\eqref{EQ::NonlinearSchauderEpsilon} immediately implies the non-uniform one~\eqref{EQ::NonlinearSchauder}. However, we first establish Proposition~\ref{PROP::ExistenceRegularity} and only then use it as a stepping stone to prove Theorem~\ref{THM::epsIndependentRegularity}.
    
    The reason is that the reflection and rescaling argument of Section~\ref{SEC::UniformRegularity}, which yields the uniform estimate~\eqref{EQ::NonlinearSchauderEpsilon}, requires the solution~$u_\eps$ to already possess the regularity given by~\eqref{EQ::NonlinearSchauder} in order to be carried out. In other words, one must first find the solution in the appropriate weighted Hölder space and establish its regularity via Proposition~\ref{PROP::ExistenceRegularity}, and only afterwards can one show that the resulting estimate is in fact uniform in~$\eps$, which is the content of Theorem~\ref{THM::epsIndependentRegularity}.
\end{remark}

The rest of this section is devoted to the proof of Proposition~\ref{PROP::ExistenceRegularity}.

\subsection{Solutions of the linearized problem}
We first tackle the linearized version of problem~\eqref{EQ::MainProblem}, recalling classical existence results for problems of this type and reducing such problem to previously known cases. We start by proving a preliminary Hölder estimate.

\begin{lemma}\label{LEM::ForcingTermCalphaBound}
    Let~$a \in (1,2)$,~$b \in [-a,+\infty)$,~$T>0$, and~$v \in \mathcal{H}_{a}^{(b)}(\Omega_L^T) \cap C(\overline{\Omega}_L^T)$. Define~$f_v \colon \Omega_L^T \to \R$ as 
    \begin{equation}
        f_v(x,t) = \int_{\Omega_L} K(x,y)v^+(y,t)\,dy + \psi\left(x,t,v(x,t),\nabla v(x,t)\right),
    \end{equation}
    with~$K$ satisfying~\eqref{EQ::KernelAssumptionL1} and~\eqref{EQ::KernelAssumptionHolder} for a constant~$C_L$, and~$\psi$ satisfying~\eqref{EQ::psiAssumptionZero} and~\eqref{EQ::psiAssumptionLipschitz} for a constant~$C_\psi$. 
    
    Then, there exists a constant~$\const{CONST::ForcingTermCalphaBound} > 0$, which depends only on~$a$,~$b$,~$C_L$,~$C_\psi$, and~$\operatorname{diam}(\Omega_L^T)$ such that
    \begin{equation}\label{EQ::ForcingTermCalphaBound}
        \abs{f_v}_{a-1;\Omega^T_L}^{(b+1)} \leq \const{CONST::ForcingTermCalphaBound} \left(1+\abs{v}_{0;\Omega_L^T}+\abs{v}_{{a; \Omega_L^T}}^{(b)}\right).
    \end{equation}
\end{lemma}

\begin{proof}
    For brevity we let
    \begin{equation}\label{EQ::etazetaDefinition}
        \eta_v(x,t) \coloneq \int_{\Omega_L} K(x,y)v^+(y,t)\,dy,\quad \text{and}\quad\zeta_v(x,t) \coloneq \psi\big(x,t,v(x,t),\nabla v(x,t)\big),
    \end{equation}
    so that~$f_v = \eta_v + \zeta_v$.

    We claim that it suffices to prove that, for every~$\delta > 0$, there exists~$C > 0$ independent of~$\delta$ such that
    \begin{equation}\label{EQ::etavDeltaNormInequality}
        \abs{\eta_v}_{a-1; I_\delta(\Omega_L^T)} \leq C \left(\abs{v}_{0; \Omega_L^T} + \abs{v}_{a; I_\delta(\Omega_L^T)}\right)
    \end{equation}
    and
    \begin{equation}\label{EQ::zetavDeltaNormInequality}
        \abs{\zeta_v}_{a-1; I_\delta(\Omega_L^T)} \leq C \left(1 + \abs{v}_{a; I_\delta(\Omega_L^T)}\right).
    \end{equation}

    Indeed, if~\eqref{EQ::etavDeltaNormInequality} and~\eqref{EQ::zetavDeltaNormInequality} are true then
    \begin{align}
        \abs{f_v}_{a-1;\Omega_L^T}^{(b+1)} &= \sup_{\delta > 0} \delta^{a+b}\abs{f_v}_{a-1; I_\delta(\Omega_L^T)} \leq \sup_{\delta >0} \delta^{a+b}\left(\abs{\eta_v}_{a-1; I_\delta(\Omega_L^T)} + \abs{\zeta_v}_{a-1; I_\delta(\Omega_L^T)}\right) \\
            &= \sup_{\delta \in (0,\operatorname{diam}(\Omega^T))} \delta^{a+b}\left(\abs{\eta_v}_{\alpha; I_\delta(\Omega_L^T)} + \abs{\zeta_v}_{\alpha; I_\delta(\Omega_L^T)}\right)\\
            &\leq 2C \sup_{\delta \in (0,\operatorname{diam}(\Omega^T))}\delta^{a+b}\left(1+\abs{v}_{0;\Omega_L^T}+\abs{v}_{a; I_\delta(\Omega_L^T)}\right)\\
            &\leq 2C\max\left\{1,{\left(\operatorname{diam}(\Omega_L^T)\right)}^{a+b}\right\} \left(1+\abs{v}_{0;\Omega_L^T}+\abs{v}_{{a; \Omega_L^T}}^{(b)}\right),
    \end{align}
    so that~\eqref{EQ::ForcingTermCalphaBound} holds with~$\const{CONST::ForcingTermCalphaBound} \coloneq 2C\max\left\{1,{\left(\operatorname{diam}(\Omega_L^T)\right)}^{a+b}\right\}$.

    Therefore we first focus on proving~\eqref{EQ::etavDeltaNormInequality}. We let~$\delta > 0$, and we then pick~$(x,t)$ and~$(x^\prime, t^\prime) \in I_\delta(\Omega_L^T)$, with~$(x,t) \neq (x^\prime, t^\prime)$. Using~\eqref{EQ::KernelAssumptionL1} and~\eqref{EQ::KernelAssumptionHolder}, we have
    \begin{align}
            |\eta_v(&x,t)-\eta_v(x^\prime, t^\prime)|\\
                &=\abs{\int_{\Omega_L} K(x,y)v^+(y,t)\, dy - \int_{\Omega_L}K(x^\prime,y)v^+(y,t^\prime)\, dy}\\
                &\leq \int_{\Omega_L}\abs{K(x,y)v^+(y,t) - K(x^\prime, y)v^+(y,t^\prime)}\, dy\\
                &\leq \int_{\Omega_L}\abs{K(x,y)v^+(y,t) - K(x^\prime, y)v^+(y,t)}\, dy+\int_{\Omega_L}\abs{K(x^\prime,y)v^+(y,t) - K(x^\prime, y)v^+(y,t^\prime)}\, dy\\
                &=\int_{\Omega_L}\abs{K(x,y) - K(x^\prime, y)}v^+(y,t)\, dy + \int_{\Omega_L} K(x^\prime,y)\abs{v^+(y,t) - v^+(y,t^\prime)}\, dy\\
                &\leq C_L \abs{v}_{0; \Omega^T} \abs{x-x^\prime}^{a-1} + C_L {[v]}_{a-1; I_\delta(\Omega_L^T)} \abs{t-t^\prime}^{\frac{a-1}{2}}\\
                &\leq C_L \left(\abs{v}_{0; \Omega^T} + {[v]}_{a-1; I_\delta(\Omega_L^T)}\right) \abs{(x-x^\prime, t-t^\prime)}_{P}^{a-1},
    \end{align}
    where we also used the fact that the positive part map is Lipschitz continuous with Lipschitz constant equal to~$1$, and we used the notation for the parabolic norm in~\eqref{EQ::ParabolicDistanceDefinition}. Dividing the extremal terms in the above chain of inequalities by~$\abs{(x-x^\prime, t-t^\prime)}_{P}^{a-1}$ yields
    \begin{equation}\label{EQ::etavCalphaBound}
        {[\eta_v]}_{a-1; I_\delta(\Omega_L^T)} \leq C_L \left(\abs{v}_{0; \Omega_L^T} + {[v]}_{a-1; I_\delta(\Omega_L^T)}\right) \leq C_L \left(\abs{v}_{0;\Omega_L^T}+\abs{v}_{{a; I_\delta(\Omega_L^T)}}\right)
    \end{equation}
    due to the arbitrarity of~$(x,t)$ and~$(x^\prime, t^\prime)$.

    Moreover, thanks to~\eqref{EQ::KernelAssumptionL1},
    \begin{equation}
        \abs{\eta_v(x,t)} = \abs{\int_{\Omega_L} K(x,y)v^+(y,t)\,dy} \leq C_L \abs{v}_{0;\Omega_L^T}.
    \end{equation}
    From this and~\eqref{EQ::etavCalphaBound}, it follows that~\eqref{EQ::etavDeltaNormInequality} holds true with~$C \geq 2C_L$.

    To prove~\eqref{EQ::zetavDeltaNormInequality} we use~\eqref{EQ::psiAssumptionLipschitz} and obtain
    \begin{align}
        \abs{\zeta_v(x,t) - \zeta_v(x^\prime, t^\prime)}
            &= \abs{\psi(x,t,v(x,t),\nabla v(x,t)) - \psi(x^\prime,t^\prime,v(x^\prime,t^\prime),\nabla v(x^\prime,t^\prime))}\\
            &\leq 2C_\psi \left( \abs{(x-x^\prime, t-t^\prime)}_{P}^{a-1} + \abs{v(x,t) - v(x^\prime,t^\prime)} + \abs{\nabla v(x,t) - \nabla v(x^\prime,t^\prime)} \right)\\
            &\leq 2C_\psi \left( 1 + {[v]}_{a-1;I_\delta(\Omega_L^T)} + {[\nabla v]}_{a-1;I_\delta(\Omega_L^T)}\right)\abs{(x-x^\prime, t-t^\prime)}_{P}^{a-1}.
    \end{align}
    We divide by the parabolic distance and find that
    \begin{equation}\label{EQ::zetavCalphaBound}
        {[\zeta_v]}_{a-1; I_\delta(\Omega_L^T)} \leq 2C_\psi \left( 1 + {[v]}_{a-1;I_\delta(\Omega_L^T)} + {[\nabla v]}_{a-1;I_\delta(\Omega_L^T)}\right) \leq 2C_\psi \left( 1 + \abs{v}_{a;I_\delta(\Omega_L^T)}\right).
    \end{equation}

    Finally, combining~\eqref{EQ::psiAssumptionZero} and~\eqref{EQ::psiAssumptionLipschitz} we deduce that
    \begin{align}
        \abs{\zeta_v(x,t)}
            &= \abs{\psi(x,t,v(x,t),\nabla v(x,t))}\\
            &= \abs{\psi(x,t,v(x,t),\nabla v(x,t)) - \psi(x,t,v(x,t),0)}\\
            &\leq C_\psi \abs{\nabla v(x,t)} \leq C_\psi \abs{\nabla v}_{0;I_\delta(\Omega_L^T)} \leq C_\psi \abs{v}_{a; I_\delta(\Omega_L^T)},
    \end{align}
    which, together with~\eqref{EQ::zetavCalphaBound}, proves~\eqref{EQ::zetavDeltaNormInequality} with~$C \geq 3C_\psi$.
    
    The proof is concluded by choosing~$C \coloneq \max\left\{2C_L, 3C_\psi\right\}$ in~\eqref{EQ::etavDeltaNormInequality} and~\eqref{EQ::zetavDeltaNormInequality} so that~\eqref{EQ::ForcingTermCalphaBound} holds with~$\const{CONST::ForcingTermCalphaBound} \coloneq 2\max\left\{2C_L, 3C_\psi\right\}\max\left\{1,{\left(\operatorname{diam}(\Omega_L^T)\right)}^{a+b}\right\}$.
\end{proof}

\begin{remark}\label{REM::ForcingTermRegularity}
Lemma~\ref{LEM::ForcingTermCalphaBound} plays a dual role in our analysis. On one hand, in the linear setting, estimate~\eqref{EQ::ForcingTermCalphaBound} provides the control on the norm of the forcing term~$f_v$ that is needed to apply the linear existence theory and solve problem~\eqref{EQ::LinearProblem}, as carried out in Lemma~\ref{LEM::LinearSolution} later in this section.

On the other hand, and more crucially, estimate~\eqref{EQ::ForcingTermCalphaBound} is the key ingredient that makes the nonlinear arguments work, both in the existence proof of Proposition~\ref{PROP::ExistenceRegularity} and in the uniform regularity estimates of Theorem~\ref{THM::epsIndependentRegularity}. The reason is the following. The natural regularity space for solutions of parabolic equations of the type we consider is~$\mathcal{H}_{2+\alpha}^{(b)}$, yet estimate~\eqref{EQ::ForcingTermCalphaBound} controls the norm of~$f_v$ in terms of~$\abs{v}_{1+\alpha}^{(b)}$, which is a strictly weaker norm. This gap between the regularity used to control~$f_v$ and the full regularity of the solution is precisely what allows compactness arguments to close.
\end{remark}

We present our existence result for the linearized equation.

\begin{lemma}\label{LEM::LinearSolution}
    There exists a constant~$\lambda_2 \in (0,1)$, which depends only on~$n$ and~$\mathcal{S}$, such that, if~$\lambda \in (0,\lambda_2)$, for every~$T>0$,~$v \in \mathcal{H}_{1+\alpha}^{-\lambda}(\Omega_L^T)$ and~$g \in \mathcal{H}_\lambda\left(\mathcal{P}_L^{(T)}\right)$, the problem
    \begin{equation}\label{EQ::LinearProblem}
        \begin{cases}
            \partial_{t} w_v(x,t) = \Delta w_v(x,t) + f_v(x,t)&\text{for all }(x,t) \in \Omega_L^T,\\
            \partial_{\nu} w_v(x,t) = 0&\text{for all }(x,t) \in \mathcal{N}_L^T,\\
            w_v(x,t) = g(x,t)&\text{for all }(x,t) \in \mathcal{P}_L^{(T)}
        \end{cases}
    \end{equation}
    admits a unique classical solution~$w_v \in \mathcal{H}_{2+\alpha}^{(-\lambda)}(\Omega_L^T)$.
    
    Moreover, there exists a constant~$\const{CONST::LinearSchauder} > 0$, which depends only on~$n$,~$\mathcal{S}$,~$L$,~$\alpha$,~$\lambda$,~$T$,~$C_L$ and~$C_\psi$, such that~$w_v$ satisfies the estimate
    \begin{equation}\label{EQ::LinearSchauder}
        \abs{w_v}_{2+\alpha;\Omega_L^T}^{(-\lambda)} \leq \const{CONST::LinearSchauder} \left(1+\abs{v}_{0; \Omega_L^T} + \abs{v}_{1+\alpha; \Omega_L^T}^{(1-\lambda)} + \abs{g}_{\lambda; \mathcal{P}_L^{(T)}}\right).
    \end{equation}
\end{lemma}

\begin{proof}
    We want to apply~\cite{MR826642}*{Theorem 4}. Thus, we now discuss the satisfaction of its assumptions.
    
    From the results of Appendix~\ref{APP::FermiCoordinates} we have that~$\mathcal{N}_L$ has~$C^{2,\alpha}$ regularity since~$\mathcal{S}$ is~$C^{3,\alpha}$ and therefore~$\Phi(\cdot, \pm s)$ (as defined in~\eqref{EQ::PhiDefinition}) is a~$C^{2,\alpha}$ diffeomorphism\footnote{Although~$\Phi$ is defined on~$\mathcal{S} \times (-L,L)$, it is possible to choose~$L^\ast \in (L,L_0)$ and define~$\Phi$ on~$\mathcal{S} \times (-L^\ast,L^\ast)$ and then restrict it to~$\mathcal{S} \times [-L,L]$. This procedure is equivalent to a continuous extension up to the closure of the domain.}. Due to the same logic, we have that~$\mathcal{D}_L = \Phi(\partial \mathcal{S}, [-L, L])$ and its relative boundary are also~$C^{2,\alpha}$.
    
    Therefore,~$\Omega_L$ satisfies a uniform~$\Sigma$-wedge condition as defined in~\cite{MR826642} (see also the comments at the end of~\cite{MR826642}*{page~426}). Moreover, the regularity of~$\mathcal{D}_L$ entails uniform internal and external cone conditions on it. We conclude that~$\Omega_L$ satisfies the geometric assumptions of~\cite{MR826642}*{Theorem~4}.

    In the notation of~\cite{MR826642}, our operators can be written as
    \[Pu \coloneq \Delta u - \partial_t u, \qquad Mu \coloneq \partial_\nu u,\]
    and the corresponding coefficients are then
    \[a^{ij} \coloneq \delta_{ij}, \qquad b^i \coloneq c \coloneq \gamma \coloneq 0, \qquad \beta^i \coloneq \nu_i.\]
    The problem data are
    \[f_1 \coloneq f_v,\qquad f_2 \coloneq 0,\qquad f_3 \coloneq g.\]
    Clearly, for~$\xi \in \R^n$,
    \[a^{ij} \xi_i \xi_j = \delta_{ij} \xi_i \xi_j = \abs{\xi}^2\]
    and
    \[\beta^i \nu_i = \nu_i\nu_i= 1.\]
    Also,~$a^{ij} \in \mathcal{H}_\alpha^{(0)}(\Omega_L)$,~$b^i \in \mathcal{H}_\alpha^{(1)}(\Omega_L)$,~$c \in \mathcal{H}_\alpha^{(2)}(\Omega_L)$, and~$\gamma\in \mathcal{H}_{1+\alpha}^{(1)}(\mathcal{N}_L)$. To see that~$\beta^i\in \mathcal{H}_{1+\alpha}^{(0)}(\mathcal{N}_L)$, it suffices to notice that~$\mathcal{N}$ being~$C^{2,\alpha}$ entails (also due to Proposition~\ref{PROP::HolderEmbeddingB}) that
    \[\nu \in C^{1,\alpha}(\overline{\mathcal{N}}) = \mathcal{H}_{1+\alpha}^{(-1-\alpha)}(\mathcal{N}) \hookrightarrow \mathcal{H}_{1+\alpha}^{(0)}(\mathcal{N}).\]

    The conditions
    \[\lim_{\delta\to 0}\delta^{1+\alpha} \abs{b^i}_{\alpha; I_\delta(\Omega_L)} =0\]
    and
    \[\sup_\Omega c+ \sup_{\Gamma_N} \gamma < +\infty\]
    are trivially satisfied.

    We apply~\cite{MR826642}*{Lemma 3} to find~$\lambda_2 \in (0,1)$, and thereby assume~$\lambda \in (0,\lambda_2)$. All in all,~\cite{MR826642}*{Theorem 4, point (b)} yields a unique classical solution~$w_v$ of~\eqref{EQ::LinearProblem}, which satisfies the Schauder-type estimate
    \begin{equation}\label{EQ::SchauderEstimateWv}
        \abs{w_v}_{2+\alpha;\Omega_L^T}^{(-\lambda)} \leq C \left( \abs{f_v}_{\alpha; \Omega_L^T}^{(2-\lambda)} + \abs{g}_{\lambda; \mathcal{P}_L^{(T)}} \right),
    \end{equation}
    where~$C > 0$ depends only on~$n$,~$\mathcal{S}$,~$L$,~$\alpha$,~$\lambda$ and~$T$.

    Plugging~\eqref{EQ::ForcingTermCalphaBound} into~\eqref{EQ::SchauderEstimateWv} with~$a \coloneq 1+\alpha$ and~$b \coloneq 1-\lambda$, we find~\eqref{EQ::LinearSchauder} with~$\const{CONST::LinearSchauder} \coloneq \max\{1,\const{CONST::ForcingTermCalphaBound}\}C$.
\end{proof}

\subsection{A Gronwall-type parabolic maximum estimate}
This section contains a pointwise estimate on the evolution of a solution. Later, we will discuss two important consequences, both of which are instrumental in the proof of Proposition~\ref{PROP::ExistenceRegularity}.

\begin{lemma}\label{LEM::ParabolicGronwall}
    Let~$L \in (0,L_0)$ and~$T > 0$. Let~$u \in C^\ast(\Omega_L^T)$, where the space~$C^\ast(\Omega_L^T)$ is defined by~\eqref{EQ::CstarDefinition}, and suppose that there exists a constant~$C_\# > 0$ such that, for every~$(x,t) \in \Omega_L^T$,
    \begin{equation}\label{EQ::ParabolicGronwallAssumption}
        \abs{\partial_t u(x,t) - \Delta u(x,t)} \leq C_\# \abs{u}_{0; \Omega(t)} + h(x,t),
    \end{equation}
    where~$h \in C\left(\Omega_L^T \cup \mathcal{N}_L^T\right)$ is such that
    \begin{equation}\label{EQ::ParabolicGronwallH}
        h(x,t) = 0\text{ for every }(x,t)\text{ such that }\nabla u(x,t) = 0.
    \end{equation}

    Assume also that
    \begin{equation}\label{EQ::ParabolicGronwallNeumann}
        \partial_\nu u(x,t)= 0\text{ for every }(x,t)\in\mathcal{N}_L^T.
    \end{equation}

    Then, for every~$t \in [0,T]$, 
    \begin{equation}\label{EQ::ParabolicGronwallThesis}
        \abs{u(\cdot, t)}_{0; \Omega_L} \leq e^{C_\# t} \abs{u}_{0; \mathcal{P}_L^{(T)}}.
    \end{equation}
\end{lemma}

\begin{proof}
    We will actually show that
    \begin{equation}\label{EQ::ParabolicGronwallMu}
        \abs{u(\cdot, t)}_{0; \Omega_L} < e^{\left(C_\# + 2\mu \right) t} \left(\mu + \abs{u}_{0; \mathcal{P}_L^{(T)}}\right)
    \end{equation}
    holds for every~$\mu \in (0,1)$, from which~\eqref{EQ::ParabolicGronwallThesis} will plainly follow upon choosing arbitrarily small values of~$\mu$.
    
    Now, we reason by contradiction and suppose that~\eqref{EQ::ParabolicGronwallMu} is false. Obviously,~\eqref{EQ::ParabolicGronwallMu} holds true at~$t=0$. Due to the continuity of~$u$ and the compactness of~$\overline{\Omega}_L$, we have that~\eqref{EQ::ParabolicGronwallMu} also holds for every~$t \in (0,\bar{t})$, if~$\bar{t}$ is chosen small enough. Therefore, there exists a smallest time~$t_\ast > 0$ at which~\eqref{EQ::ParabolicGronwallMu} is violated.

    Thanks to the Weierstrass Theorem, we find at least one point~$x_\ast \in \overline{\Omega}_L$ at which the maximum absolute value is attained. That is,~$x_\ast$ is such that
    \begin{equation}\label{EQ::xAstAbsoluteMaximality}
        \abs{u(x_\ast, t_\ast)} = \abs{u(\cdot, t_\ast)}_{0; \Omega_L}.
    \end{equation}

    We now use the symbol~``$\diamondsuit$'' to denote~``$\le$'' if~$u(x_\ast,t_\ast)<0$ and~``$\ge$'' if~$u(x_\ast,t_\ast)>0$: in this way, we know from the maximality of~$|u(x_\ast,t_\ast)|$ that~$x_\ast$ is minimizing when~$u(x_\ast,t_\ast)<0$ and maximizing when~$u(x_\ast,t_\ast)>0$ (note that the violation of~\eqref{EQ::ParabolicGronwallMu} excludes~$u(x_\ast, t_\ast) = 0$). That is
    \begin{equation}\label{EQ::xAstMaximality}
        u(x_\ast, t_\ast) \,\diamondsuit\, u(x,t_\ast) \text{ for every } x \in \overline{\Omega}_L.
    \end{equation}

    Clearly, 
    \begin{equation}\label{EQ::xastNotInD}
        x_\ast \notin \mathcal{D}_L,
    \end{equation}
    otherwise the violation of~\eqref{EQ::ParabolicGronwallMu} would not occur.
    
    We claim that
    \begin{equation}\label{EQ::GronwallNullGradient}
        \nabla u(x_\ast,t_\ast) = 0.
    \end{equation}
    Indeed, if~$x_\ast$ is an interior point of~$\Omega_L$, the claim in~\eqref{EQ::GronwallNullGradient} immediately follows from the criticality property of~$x_\ast$. If, instead,~$x_\ast \in \partial\Omega_L$, we know from~\eqref{EQ::xastNotInD} that~$x_\ast\in\mathcal{N}_L$. In this case,~$x_\ast$ is an interior point of~$\mathcal{N}_L$ due to~$\mathcal{N}_L$ being open. Hence, denoting with~$\nabla_T$ the tangential gradient on~$\mathcal{N}_L$, we infer that~$\nabla_T u(x_\ast, t_\ast) = 0$, which, together with the Neumann condition in~\eqref{EQ::ParabolicGronwallNeumann}, yields~\eqref{EQ::GronwallNullGradient}.

    We also claim that
    \begin{equation}\label{EQ::xAstLaplacianSign}
        -\Delta u(x_\ast, t_\ast) \,\diamondsuit\, 0.
    \end{equation}
    Again, this holds trivially if~$x_\ast$ is an interior maximizer/minimizer. In the case~$x_\ast \in \mathcal{N}_L$, we know that
    \begin{equation}\label{EQ::xAstLaplaceBeltramiSign}
        -\Delta_{\mathcal{N}_L} u(x_\ast,t_\ast) \,\diamondsuit\, 0,
    \end{equation}
    with~$\Delta_{\mathcal{N}_L}$ being the Laplace-Beltrami operator on~$\mathcal{N}_L$.
    
    Moreover, if~$s>0$ is small enough, there holds
    \begin{equation}\label{EQ::DiamondSuitInequalityAsymptotic}
        0 \,\diamondsuit\, u(x_\ast - s\nu(x_\ast), t_\ast) - u(x_\ast, t_\ast) = -s\partial_\nu u(x_\ast, t_\ast) + \frac{s^2}{2}\partial_\nu^2u(x_\ast, t_\ast) + \smallo(s^2),
    \end{equation}
    where we have used~\eqref{EQ::xAstAbsoluteMaximality}. Plugging~\eqref{EQ::ParabolicGronwallNeumann} into~\eqref{EQ::DiamondSuitInequalityAsymptotic} and choosing~$s$ as small as we want, we deduce
    \begin{equation}
        -\partial_\nu^2 u(x_\ast,t_\ast)\,\diamondsuit\,0,
    \end{equation}
    which, in conjunction with~\eqref{EQ::xAstLaplaceBeltramiSign} and Proposition~\ref{PROP::PacardLaplacianDecomposition}, establishes~\eqref{EQ::xAstLaplacianSign}.

    Now we claim that
    \begin{equation}\label{EQ::xAstLogDerivativeBound}
        \frac{\partial_t u(x_\ast,t_\ast)}{u(x_\ast,t_\ast)} \leq C_\#.
    \end{equation}
    If~$u(x_\ast, t_\ast) > 0$, ``$\diamondsuit$'' is to be read as ``$\ge$''. Then,~\eqref{EQ::xAstLaplacianSign} tells us that
    \begin{equation}
        \partial_t u(x_\ast,t_\ast) \leq \partial_t u(x_\ast,t_\ast) -\Delta u(x_\ast, t_\ast) \leq \abs{\partial_t u(x_\ast,t_\ast) -\Delta u(x_\ast, t_\ast)},
    \end{equation}
    hence, using~\eqref{EQ::ParabolicGronwallAssumption} and~\eqref{EQ::ParabolicGronwallH},
    \begin{equation}
        \partial_t u(x_\ast,t_\ast) \leq C_\#\abs{u(\cdot, t_\ast)}_{0; \Omega_L} + h(x_\ast,t_\ast) = C_\#\abs{u(\cdot, t_\ast)}_{0; \Omega_L}.
    \end{equation}
    Thanks to this,~\eqref{EQ::xAstAbsoluteMaximality} and~$u(x_\ast, t_\ast)>0$, we find
    \begin{equation}
        \partial_t u(x_\ast,t_\ast) \leq C_\# u(x_\ast, t_\ast),
    \end{equation}
    from which~\eqref{EQ::xAstLogDerivativeBound} follows since~$u(x_\ast, t_\ast) > 0$.

    The case~$u(x_\ast,t_\ast) < 0$ is analogous. Namely, using in this order~\eqref{EQ::xAstLaplacianSign},~\eqref{EQ::ParabolicGronwallAssumption},~\eqref{EQ::ParabolicGronwallH},~\eqref{EQ::xAstAbsoluteMaximality}, and~$u(x_\ast,t_\ast) < 0$, we compute
    \begin{align}
        \partial_t u(x_\ast,t_\ast) &\geq \partial_t u(x_\ast,t_\ast) - \Delta u(x_\ast,t_\ast)\\
            &\geq -\abs{ \partial_t u(x_\ast,t_\ast) - \Delta u(x_\ast,t_\ast)}\\
            &\geq - C_\# \abs{u(\cdot, t_\ast)}_{0; \Omega_L} - h(x_\ast,t_\ast)\\
            &= - C_\#\abs{u(\cdot, t_\ast)}_{0; \Omega_L} = C_\# u(x_\ast,t_\ast).
    \end{align}
    Since here we supposed~$u(x_\ast,t_\ast)<0$, the above computation gives~\eqref{EQ::xAstLogDerivativeBound}.

    We now define, for~$\delta>0$ small enough,
    \begin{equation}
        [t_\ast-\delta, t_\ast+\delta] \ni \sigma \longmapsto V(\sigma) \coloneq \log \frac{u(x_\ast, \sigma)}{u(x_\ast, t_\ast)},
    \end{equation}
    and conclude that
    \begin{equation}
        \dot{V}(\sigma) = \frac{\partial_t u(x_\ast, \sigma)}{u(x_\ast, \sigma)},
    \end{equation}
    which, together with~\eqref{EQ::xAstLogDerivativeBound}, gives~$\dot{V}(t_\ast) \leq C_\#$.

    Then, there exists~$\delta_\mu \in (0,\delta)$ such that~$\dot{V}(\sigma) \leq C_\# + \mu$ for every~$\sigma \in [t_\ast-\delta_\mu, t_\ast+\delta_\mu]$, so that
    \begin{equation}
        \log \frac{u(x_\ast, t_\ast)}{u(x_\ast, t_\ast-\delta_\mu)} = V(t_\ast) - V(t_\ast-\delta_\mu) \leq \left(C_\# + \mu\right)\delta_\mu,
    \end{equation}
    and, as a result,
    \begin{equation}\label{EQ::xAstExponentialBound}
        \abs{u(x_\ast, t_\ast)} \leq e^{(C_\# + \mu)\delta_\mu}\abs{u(x_\ast, t_\ast-\delta_\mu)}.
    \end{equation}

    Since~$t_\ast$ is the first occurrence at which~\eqref{EQ::ParabolicGronwallMu} is violated,
    \begin{equation}
        \abs{u(x_\ast, t_\ast-\delta_\mu)} \leq e^{(C_\# + 2\mu)(t_\ast-\delta_\mu)}\left(\mu + \abs{u}_{0; \mathcal{P}_L^{(T)}}\right).
    \end{equation}
    This, together with~\eqref{EQ::xAstExponentialBound}, yields that
    \begin{align}
        \abs{u(x_\ast, t_\ast)} &\leq e^{(C_\# + \mu)\delta_\mu+(C_\# + 2\mu)(t_\ast-\delta_\mu)}\left(\mu + \abs{u}_{0; \mathcal{P}_L^{(T)}}\right)\\
            &= e^{-\mu\delta_\mu}e^{(C_\# + 2\mu)t_\ast}\left(\mu + \abs{u}_{0; \mathcal{P}_L^{(T)}}\right)\\
            &<e^{(C_\#+2\mu)t_\ast}\left(\mu + \abs{u}_{0; \mathcal{P}_L^{(T)}}\right),
    \end{align}
    in contradiction with the definition of~$t_\ast$.
\end{proof}

As anticipated, we examine two consequences of Lemma~\ref{LEM::ParabolicGronwall}. The first is a uniqueness result.

\begin{corollary}\label{COR::NonlinearUniqueness}
    There exists up to one solution of~\eqref{EQ::MainProblem} in the class~$C^\ast\left(\Omega_L^T\right)$.
\end{corollary}

\begin{proof}
    Let~$u$ and~$v$ be two arbitrary, possibly different,~$C^\ast\left(\Omega_L^T\right)$ solutions of~\eqref{EQ::MainProblem}. Then,~$w\coloneq u-v$ belongs to the same class and is a solution of the problem
    \begin{equation}\label{EQ::wUniquenessProblem}
        \begin{cases}
            \partial_{t} w(x,t) = \Delta w(x,t) + f_u(x,t) - f_v(x,t)&\text{for all }(x,t) \in \Omega_L^T,\\
            \partial_{\nu} w(x,t) = 0&\text{for all }(x,t) \in \mathcal{N}_L^T,\\
            w(x,t) = 0&\text{for all }(x,t) \in \mathcal{P}_L^{(T)}.
        \end{cases}
    \end{equation}

    Using the notation introduced in~\eqref{EQ::etazetaDefinition}, we use the Lipschitz property of the positive part map to deduce that, for every~$(x,t) \in \Omega^T$,
    \begin{equation}\label{EQ::wUniquenessBoundG}
        \begin{split}    
            \abs{\eta_u(x,t)-\eta_v(x,t)} &= \abs{\int_\Omega K(x,y)u^+(y,t)\,dy - \int_\Omega K(x,y)v^+(y,t)\,dy}\\
                &\leq \int_\Omega K(x,y) \abs{u^+(y,t)-v^+(y,t)}\,dy\\
                &\leq \int_\Omega K(x,y) \abs{w(y,t)}\,dy\\
                &\leq C_L \abs{w(\cdot, t)}_{0; \Omega_L},
        \end{split}
    \end{equation}
    where we also used~\eqref{EQ::KernelAssumptionL1}.

    Similarly, we use~\eqref{EQ::psiAssumptionLipschitz} to obtain that, for every~$(x,t) \in \Omega_L^T$,
    \begin{align}\label{EQ::wUniquenessBoundH}
        \abs{\zeta_u(x,t)-\zeta_v(x,t)} &= \abs{\psi(x,t,u(x,t),\nabla u(x,t))-\psi(x,t,v(x,t),\nabla v(x,t))}\\
            &\leq C_\psi \left(\abs{u(x,t)-v(x,t)} + \abs{\nabla u(x,t)-\nabla v(x,t)}\right)\\
            &= C_\psi \left(\abs{w(x,t)} + \abs{\nabla w(x,t)}\right).
    \end{align}

    Putting together~\eqref{EQ::wUniquenessProblem},~\eqref{EQ::wUniquenessBoundG}, and~\eqref{EQ::wUniquenessBoundH}, we find out, for every~$(x,t) \in \Omega_L^T$,
    \begin{equation}\label{EQ::dtwLapwControl}
        \abs{\partial_{t} w(x,t) - \Delta w(x,t)} = \abs{f_u(x,t)-f_v(x,t)}\leq \left(C_L + C_\psi\right)\abs{w(\cdot,t)}_{0; \Omega_L} + C_\psi \abs{\nabla w(x,t)}.
    \end{equation}
    
    From the regularity of~$u$,~$v$, and~$w$, we deduce that~\eqref{EQ::dtwLapwControl} is also valid for~$(x,t)\in\mathcal{N}_L^T$. We thereby notice that~$w$ satisfies~\eqref{EQ::ParabolicGronwallAssumption} with~$C_\# \coloneq \left(C_L + C_\psi\right)$ and~$h(x,t) \coloneq C_\psi \abs{\nabla w(x,t)}$. Clearly, this choice of~$h$ also satisfies~\eqref{EQ::ParabolicGronwallH}.
    
    We then apply Lemma~\ref{LEM::ParabolicGronwall} to~$w$ (with~$C_\# \coloneq C_L + C_\psi$ and~$h(x,t) \coloneq C_\psi \abs{\nabla w(x,t)}$) and see that~\eqref{EQ::ParabolicGronwallThesis} and~\eqref{EQ::wUniquenessProblem} entail that~$w\equiv 0$, proving the desired uniqueness property.
\end{proof}

We also get a uniform bound on the~$\mathcal{H}_0$ norm (i.e., the maximum absolute value) of solutions in a finite time cylinder.

\begin{corollary}\label{COR::NonlinearSolutionSupBound}
    Let~$u \in C^\ast\left(\Omega_L^T\right)$ be a solution of~\eqref{EQ::MainProblem}.
    
    Then,
    \begin{equation}\label{EQ::NonlinearSolutionSupBound}
        \abs{u}_{0; \Omega_L^T} \leq e^{C_L T}\abs{g}_{0; \mathcal{P}_L^{(T)}}.
    \end{equation}
\end{corollary}
\begin{proof}
    From~\eqref{EQ::KernelAssumptionL1} we know that
    \begin{equation}
        \int_{\Omega_L} K(x,y)u^+(y,t)\,dy \leq C_L \abs{u(\cdot,t)}_{0; \Omega_L}.
    \end{equation}

    Moreover, since~$K$ and~$\psi$ are non-negative,~$f_u \geq 0$. Hence,~$\partial_t u - \Delta u \geq 0$ and
    \begin{equation}
        \abs{\partial_t u - \Delta u} \leq C_L \abs{u(\cdot,t)}_{0; \Omega} + \psi(x,t,u(x,t),\nabla u(x,t)).
    \end{equation}
    Also thanks to~\eqref{EQ::psiAssumptionZero}, we can apply Lemma~\ref{LEM::ParabolicGronwall} with~$C_\# \coloneq C_L$ and~$h \coloneq \psi$, from which~\eqref{EQ::NonlinearSolutionSupBound} follows.
\end{proof}

\subsection{The fixed point argument}
We now deal with the proof of Proposition~\ref{PROP::ExistenceRegularity} by characterizing the solution as a fixed point of an operator acting on weighted Hölder spaces. To this end, we let~$T > 0$,~$\beta \in (0,\alpha)$,~$\lambda \in(0,\lambda_2)$, and~$l \in (0,\lambda)$ and consider the nonlinear operator~$A$, defined as
\begin{equation}
    A \colon \mathcal{H}_{2+\beta}^{(-l)}\left(\Omega_L^T\right) \ni v \longmapsto A(v) \coloneq w_v \in \mathcal{H}_{2+\beta}^{(-l)}\left(\Omega_L^T\right),
\end{equation}
where~$w_v$ is the solution of~\eqref{EQ::LinearProblem} with the forcing term~$f_v$ obtained from~$v$. We look for a fixed point of~$A$, because its existence would provide a classical solution of~\eqref{EQ::MainProblem}.

In the next result, we discuss some properties of~$A$ that will guarantee the existence of a fixed point.

\begin{lemma}\label{LEM::SchaeferAssumptions}
    Assume~$g \in \mathcal{H}_{\lambda}\left(\mathcal{P}_L^{(T)}\right)$.

    Then, the operator~$A \colon \mathcal{H}_{2+\beta}^{(-l)}\left(\Omega_L^T\right) \to \mathcal{H}_{2+\beta}^{(-l)}\left(\Omega_L^T\right)$ is well defined, continuous and compact. Moreover, the Schaefer set
    \begin{equation}
        S_A \coloneq \left\{ v \in \mathcal{H}_{2+\beta}^{(-l)}\left(\Omega_L^T\right) \colon v = \mu A(v),\, \mu\in[0,1]\right\}
    \end{equation}
    is bounded in~$\mathcal{H}_{2+\beta}^{(-l)}\left(\Omega_L^T\right)$, depending only on~$n$,~$\alpha$,~$\beta$,~$\mathcal{S}$,~$L$,~$T$,~$C_L$.
\end{lemma}
\begin{proof}
    Let~$v \in \mathcal{H}_{2+\beta}^{(-l)}\left(\Omega_L^T\right)$. By the continuous embedding~$\mathcal{H}_{2+\beta}^{(-l)}\left(\Omega_L^T\right) \hookrightarrow \mathcal{H}_{1+\alpha}^{(1-\lambda)}\left(\Omega_L^T\right)$ (which follows by applying Proposition~\ref{PROP::HolderEmbeddingA} with~$a \coloneq 2+\beta$,~$a^\prime \coloneq 1+\alpha$, and~$b \coloneq -l$, and then Proposition~\ref{PROP::HolderEmbeddingB} with~$a \coloneq 1+\alpha$,~$b \coloneq -l$, and~$b^\prime \coloneq 1-\lambda$), we can use Lemma~\ref{LEM::LinearSolution} to uniquely determine~$w_v = A(v) \in \mathcal{H}_{2+\alpha}^{(-\lambda)}\left(\Omega_L^T\right)$. This proves that~$A$ is well defined thanks to the continuous embedding~$\mathcal{H}_{2+\alpha}^{(-\lambda)}\left(\Omega_L^T\right) \hookrightarrow \mathcal{H}_{2+\beta}^{(-l)}\left(\Omega_L^T\right)$ (as a consequence of Proposition~\ref{PROP::HolderEmbeddingA} applied with~$a \coloneq 2+\alpha$,~$a^\prime \coloneq 2+\beta$, and~$b \coloneq -\lambda$, and then Proposition~\ref{PROP::HolderEmbeddingB} with~$a \coloneq 2+\beta$,~$b \coloneq -\lambda$, and~$b^\prime \coloneq -l$).

    Now let~${\{v_k\}}_k \subset \mathcal{H}_{2+\beta}^{(-l)}\left(\Omega_L^T\right)$ be a bounded sequence. From the estimate~\eqref{EQ::LinearSchauder} we conclude that~${\{A(v_k)\}}_k$ is bounded in~$\mathcal{H}_{2+\alpha}^{(-\lambda)}\left(\Omega_L^T\right)$. Then, Proposition~\ref{PROP::HolderCompactEmbedding} (applied with~$a \coloneq 2+\alpha$,~$b \coloneq -\lambda$,~$a^\prime \coloneq 2+\beta$, and~$b^\prime \coloneq -l$) entails that the embedding~$\mathcal{H}_{2+\alpha}^{(-\lambda)}\left(\Omega_L^T\right) \hookrightarrow \mathcal{H}_{2+\beta}^{(-l)}\left(\Omega_L^T\right)$ is compact. It follows that there exists a subsequence~${\{v_{k_m}\}}_{m}$ that is strongly convergent in~$\mathcal{H}_{2+\beta}^{(-l)}\left(\Omega_L^T\right)$, bringing us to the conclusion that~$A$ is continuous and compact.

    For every~$v \in S_A$, we apply Lemma~\ref{LEM::LinearSolution}, and in particular~\eqref{EQ::LinearSchauder}, to obtain that
    \begin{equation}\label{EQ::SchaeferBound}
        \begin{split}
            \abs{v}_{2+\beta; \Omega_L^T}^{(-l)} &= \mu\abs{A(v)}_{2+\beta;\Omega_L^T}^{(-l)} \leq \abs{A(v)}_{2+\beta; \Omega_L^T}^{(-l)} \leq \abs{A(v)}_{2+\alpha; \Omega_L^T}^{(-\lambda)}\\
                &\leq \const{CONST::LinearSchauder}\left(1+\abs{v}_{0; \Omega_L^T} +\abs{v}_{1+\alpha; \Omega_L^T}^{(1-\lambda)} + \abs{g}_{\lambda; \mathcal{P}_L^{(T)}}\right).
        \end{split}
    \end{equation}

    Since~$v = \mu A(v)$, we have~$\partial_t v -\Delta v = \mu f_v$ in~$\Omega^T$. Since~$\mu \in [0,1]$ and~$f_v \geq 0$, we have
    \begin{equation}
        \abs{\partial_t v -\Delta v} \leq \mu f_v \leq f_v,
    \end{equation}
    so that
    \begin{equation}
        \abs{\partial_t v - \Delta v} \leq C_L \abs{v(\cdot,t)}_{0; \Omega} + \psi(x,t,v(x,t),\nabla v(x,t)).
    \end{equation}
    Thus, thanks to Lemma~\ref{LEM::ParabolicGronwall}, used here with~$C_\# \coloneq C_L$ and~$h \coloneq \psi$, we see that 
    \begin{equation}\label{EQ::SchaeferSupNorm}
        \abs{v}_{0; \Omega_L^T} \leq e^{C_L T} \abs{g}_{0; \mathcal{P}_L^{(T)}}\leq e^{C_L T} \abs{g}_{\lambda; \mathcal{P}_L^{(T)}}.
    \end{equation}

    We apply Proposition~\ref{PROP::HolderInterpolation} with~$a \coloneq 2+\beta$,~$a^\prime \coloneq 0$,~$b \coloneq \frac{(1-\lambda)(2+\beta)}{1+\alpha}$,~$b^\prime \coloneq 0$,~$\vartheta \coloneq \frac{1+\alpha}{2+\beta} \in (0,1)$,~$a^\ast \coloneq 1+\alpha$, and~$b^\ast \coloneq 1-\lambda$, and Proposition~\ref{PROP::HolderEmbeddingB} with~$a \coloneq 2+\beta$,~$b \coloneq -l$, and~$b^\prime \coloneq \frac{(1-\lambda)(2+\beta)}{1+\alpha}$, obtaining
    \begin{equation}
        \abs{v}_{1+\alpha; \Omega_L^T}^{(1-\lambda)} \leq \constapp{CONST::HolderInterpolation} {\left(\abs{v}_{2+\beta; \Omega_L^T}^{\left(\frac{(1-\lambda)(2+\beta)}{1+\alpha}\right)}\right)}^\vartheta{\left(\abs{v}_{0; \Omega_L^T}\right)}^{1-\vartheta}\leq \constapp{CONST::HolderInterpolation}\constapp{CONST::HolderEmbeddingB}^\vartheta{\left(\abs{v}_{2+\beta; \Omega_L^T}^{(-l)}\right)}^\vartheta{\left(\abs{v}_{0; \Omega_L^T}\right)}^{1-\vartheta}.
    \end{equation}
    Hence, for every~$\delta > 0$, the generalized Young inequality entails the existence of~$C_\delta > 0$ such that
    \begin{equation}\label{EQ::vEstimateYoung}
        \abs{v}_{1+\alpha; \Omega_L^T}^{(1-\lambda)} \leq \delta \abs{v}_{2+\beta; \Omega_L^T}^{(-l)} + C_\delta \abs{v}_{0; \Omega_L^T}.
    \end{equation}

    We choose~$\delta \coloneq \frac{1}{2\const{CONST::LinearSchauder}}$ and plug~\eqref{EQ::vEstimateYoung} and~\eqref{EQ::SchaeferSupNorm} into~\eqref{EQ::SchaeferBound}, obtaining
    \begin{equation}\label{EQ::NonlinearSchauderSchaeferSet}
        \abs{v}_{2+\beta; \Omega_L^T}^{(-l)} \leq \const{CONST::LinearSchauder}\left(1 + (1+(1+C_\delta)e^{C_L T})\abs{g}_{\lambda; \mathcal{P}_L^{(T)}}\right),
    \end{equation}
    which concludes the proof.
\end{proof}

\begin{remark}\label{REM::NonlocalitySemilinearForRegularity}
    We stress that several results contained in this section depend on the fact that the nonlocal term is independent of the derivatives of the solution. Indeed, this fact is used in the proof of Lemma~\ref{LEM::ForcingTermCalphaBound} to find~\eqref{EQ::etavCalphaBound}, and in the proofs Corollaries~\ref{COR::NonlinearUniqueness} and~\ref{COR::NonlinearSolutionSupBound} to show~\eqref{EQ::ParabolicGronwallAssumption} in order to apply Lemma~\ref{LEM::ParabolicGronwall}.
\end{remark}

We now conclude the section proving Proposition~\ref{PROP::ExistenceRegularity}.

\begin{proof}[Proof of Proposition~\ref{PROP::ExistenceRegularity}]
    Let~$T>0$. Thanks to Lemma~\ref{LEM::SchaeferAssumptions}, the assumptions of the Schaefer Fixed Point Theorem (see e.g.~\cite{MR3967045}*{Theorem~1.20}) are satisfied and we deduce the existence of~$u_T \in H_{2+\beta}^{(-l)}(\Omega_L^T)$ such that~$A(u_T) = u_T$. Corollary~\ref{COR::NonlinearUniqueness} ensures that the solution~$u_T$ that we find is independent of our choice of~$\beta$ and~$l$.
    
    Lemmata~\ref{LEM::ForcingTermCalphaBound} and~\ref{LEM::LinearSolution} are sufficient to guarantee that~$u_T \in H_{2+\alpha}^{(-\lambda)}(\Omega_L^T)$ after finding its existence in the less regular space~$H_{2+\beta}^{(-l)}(\Omega_L^T)$ (see also Remark~\ref{REM::ForcingTermRegularity}). Because~$A(u_T) = u_T$, we have that~$u_T \in S_A$, so~$u_T$ must satisfy~\eqref{EQ::NonlinearSchauderSchaeferSet}, which gives~\eqref{EQ::NonlinearSchauder} with~$\const{CONST::ExistenceRegularity} \coloneq \const{CONST::LinearSchauder} (1+(1+C_\delta)e^{C_L T})$.

    Since~$T$ is arbitrary, the previous considerations allow us to build a sequence of solutions~${\{u_{T_k}\}}_k$, for a suitable sequence~$T_k \nearrow +\infty$ as~$k \to +\infty$, such that every pair~$u_{T_i}$ and~$u_{T_j}$ agree in~$\Omega_L^{\min\{T_i, T_j\}}$ (thanks to uniqueness). Then, the desired solution is the function
    \[u(x,t) = \lim_{k\to +\infty} u_{T_k}(x,t),\] 
    where each~$u_{T_k}$ is extended by setting it identically equal to zero for~$t > T_k$, and the limit is taken in the pointwise sense.
\end{proof}

\section{Uniform regularity estimates}\label{SEC::UniformRegularity}
This section presents the reflection approach that we adopt to prove Theorem~\ref{THM::epsIndependentRegularity}. Our procedure extends functions defined on a domain~$\Omega_\eps^T$ to a larger domain~$\Omega_L^T$. Such an extension, as we will see, acts as a topological isomorphism between some appropriate subsets of~$\mathcal{H}_a^{(b)}\left(\Omega_\eps^T\right)$ and~$\mathcal{H}_a^{(b)}\left(\Omega_L^T\right)$, so that the weighted Hölder norm of certain functions (namely, all solutions of~\eqref{EQ::MainProblemEpsilon}) is comparable to the norm of their extension.

We shall also establish several properties describing how the extension operators interact with differential operators up to second order. At the end of this section, we employ these partial results to prove Theorem~\ref{THM::epsIndependentRegularity}, by finding an equation that is satisfied by the extension of solutions to~\eqref{EQ::MainProblemEpsilon}, and applying classical Schauder estimates to it.

\subsection{The reflection map}
We introduce the periodic reflection map on a tubular neighborhood and investigate some of its properties, in view of exploiting them in Sections~\ref{SSEC::ReflectionOperator} and~\ref{SSEC::ReflectedEquation}.

Let~$L \in (0,L_0)$ and~$\eps \in (0,L)$. We define
\begin{equation}\label{EQ::kappaDefinition}
    \mathfrak{K} \coloneq \min \left\{k \in \N \colon (1+2k)\eps \geq \frac{L}{3}\right\}
\end{equation}
and
\begin{equation}\label{EQ::tauDefinition}
    \tau \coloneq (1+2\mathfrak{K})\eps.
\end{equation}

We introduce the map~$\omega \colon [-\tau, \tau] \to [-\eps,\eps]$, defined by
\begin{equation}\label{EQ::rhoTildeDefinition}
    \omega(s) \coloneq   \left\{\begin{aligned}
                            &s,&\text{if }s \in [-\eps,\eps),\\
                            &2\eps-s,&\text{if }s \in [\eps,3\eps),\\
                            &4\eps\text{-periodic,}&
                        \end{aligned}\right.
\end{equation}
and we let the reflection map~$\rho \colon \overline{\Omega}_\tau \to \overline{\Omega}_\eps$ be such that, for every~$x \in \mathcal{S}$ and~$s \in [-\tau, \tau]$,
\begin{equation}\label{EQ::rhoDefinition}
    \rho(x+s\nu(x)) = x+\omega(s)\nu(x).
\end{equation}

We present a preliminary geometric estimate regarding~$\rho$.
\begin{lemma}
    There exists a constant~$\const{CONST::rhoBoundaryDistance}>0$, which depends only on~$\mathcal{S}$ and~$L$, such that, for every~${\delta > 0}$,
    \begin{equation}\label{EQ::rhoBoundaryDistance}
        \rho\left(I_\delta(\Omega_\tau)\right) \subset I_{\const{CONST::rhoBoundaryDistance}\delta}(\Omega_\eps^T).
    \end{equation}
\end{lemma}
\begin{proof}
    We define~$\Phi \colon \mathcal{S} \times [-\tau, \tau] \to \overline{\Omega}_\tau$ as in~\eqref{EQ::PhiDefinition}. Then, from~\eqref{EQ::rhoDefinition} it follows that
    \begin{equation}\label{EQ::rhoAsComposition}
        \rho = \Phi \circ (\operatorname{Id}_\mathcal{S},\omega) \circ \Phi^{-1}.
    \end{equation}
    Thanks to the fact that~$L<L_0$, both~$\Phi$ and~$\Phi^{-1}$ are Lipschitz continuous with constant~$C_\Phi$, which depends only on~$\mathcal{S}$ and~$L$ because~$\Phi$ can be defined as the restriction of~$\Phi \colon \mathcal{S} \times [-L, L] \to \overline{\Omega}_L$.
    
    From~\eqref{EQ::rhoTildeDefinition} it is clear that~$(\operatorname{Id}_\mathcal{S},\omega)$ is also Lipschitz with constant~$1$. Therefore,~\eqref{EQ::rhoAsComposition} entails that~$\rho$ is Lipschitz continuous, and its Lipschitz constant is~$C_\Phi^2$.

    We let~$P_\mathcal{S}$ be the projection map onto~$\mathcal{S}$. That is
    \begin{equation}
        P_\mathcal{S} \coloneq \Phi \circ (Id_\mathcal{S}, 0) \circ \Phi^{-1}.
    \end{equation}
    From this and the Lipschitz property of~$\Phi$,~$(Id_\mathcal{S}, 0)$, and~$\Phi^{-1}$, it follows that~$P_\mathcal{S}$ is~$C_\Phi^2$-Lipschitz.

    Suppose by contradiction that~\eqref{EQ::rhoBoundaryDistance} is false. Then, for every~$\mu > 0$ there exists~$X \in I_\delta(\Omega_\tau)$ such that
    \begin{equation}
        X^\prime \coloneq \rho(X) \in \overline{\Omega}_\eps \smallsetminus {\left(\Omega_\eps\right)}_{\mu \delta}.
    \end{equation}
    Therefore, there exists~$Y \in \mathcal{D}_\eps$ such that
    \begin{equation}\label{EQ::XprimeYnear}
        \abs{X^\prime - Y} \leq \mu \delta.
    \end{equation}

    We define~$(y,\sigma) \coloneq \Phi^{-1}(Y)$ and~$(x,s) \coloneq \Phi^{-1}(X)$. As a consequence,~$(x,\omega(s)) = \Phi^{-1}(X^\prime)$. Hence, by~\eqref{EQ::XprimeYnear},
    \begin{equation}\label{EQ::xyNear}
        \abs{x-y} \leq C_\Phi^2 \mu \delta.
    \end{equation}

    Also,
    \begin{equation}\label{EQ::xPhiYsNear}
        \abs{X-\Phi(y,s)} = \abs{x+s\nu(x)-y-s\nu(y)} \leq \abs{x-y} + \abs{s}\abs{\nu(x)-\nu(y)}\leq C \abs{x-y},
    \end{equation}
    where~$C > 0$ depends only on~$L$ (due to~$\abs{s} \leq L$) and on the Lipschitz constant of~$\nu$ on~$\mathcal{S}$.

    Since~$Y \in \mathcal{D}_\eps$, we have that~$y \in \partial\mathcal{S}$ and~$\Phi(y,s) \in \mathcal{D}_\tau$. We combine this fact with~$X \in I_\delta(\Omega_\tau)$ and~\eqref{EQ::xPhiYsNear}, finding
    \begin{equation}
        \delta < \abs{X-\Phi(y,s)} \leq C C_\Phi^2 \abs{x-y},
    \end{equation}
    which, together with~\eqref{EQ::xyNear}, yields
    \begin{equation}
        \delta < C \mu \delta,
    \end{equation}
    giving us the desired contradiction by choosing~$\mu < \frac{1}{C C_\Phi^2}$.
\end{proof}

\begin{remark}
    The reflection map~$\rho$ inherits a periodicity property from~$\omega$. If~$s$ and~$s+4\eps$ both lie in~$(-\tau,\tau)$, one can deduce from~\eqref{EQ::rhoTildeDefinition} and~\eqref{EQ::rhoDefinition} that
    \[\rho(x+s\nu(x)) = \rho(x+(s+4\eps)\nu(x)).\]

    In fact, the representation of~$\rho$ in Fermi coordinates is periodic with respect to the~$n$-th coordinate (i.e., the normal one). Although periodicity prevents the map~$\rho$ from being injective, it is worthwhile to notice that~\eqref{EQ::rhoAsComposition} entails that the restriction of~$\rho$ to an appropriate subdomain of~$\Omega_\tau$ acts as a~$C^{2,\alpha}$ diffeomorphism. Namely, for any~$m \in \{-\mathfrak{K},\dots,\mathfrak{K}\}$, we define
    \begin{equation}\label{EQ::VkDefinition}
        V_m \coloneq \bigcup_{s \in [(-1+2m)\eps, (1+2m)\eps]} \mathcal{S}(s).
    \end{equation}

    Then, it is immediate to check that~$V_0 = \overline{\Omega}_\eps$ and
    \begin{equation}\label{EQ::rhoC2aDiffeomorphism}
        \restr{\rho}{V_m} \text{ is a } C^{2,\alpha} \text{ diffeomorphism between } V_m \text{ and } \overline{\Omega}_\eps
    \end{equation}
    for every~$m \in \{-\mathfrak{K},\dots,\mathfrak{K}\}$, with regularity constant~$C_\rho > 0$ depending only on~$n$,~$\mathcal{S}$ and~$L$.
\end{remark}

Thanks to the quasiconvexity of~$\Omega_L$ given by Lemma~\ref{LEM::Quasiconvexity}, we have the following result regarding Hölder continuity across reflection boundaries.

\begin{lemma}\label{LEM::PiecewiseRegularityToGlobal}
    Let~$a \in (0,1)$,~$b\in [-a,+\infty)$, and~$u \in C(\Omega_\tau \cup \mathcal{N}_\tau)$ be such that\footnote{We consider the space~$\mathcal{H}_a^{(b)}(V_m)$ to be defined by the usual norm 
    \[\abs{\cdot}_{a; V_m}^{(b)} \coloneq \sup_{\delta > 0} \delta^{a+b} \abs{\cdot}_{a; I_\delta(V_m)}\]
    under the convention that~$I_\delta(V_m) = V_m \cap I_\delta(\Omega_\tau)$. This definition is coherent with that of~$\mathcal{H}_a^{(b)}(\Omega_\eps)$ because it entails that~$\mathcal{H}_a^{(b)}(V_0) = \mathcal{H}_a^{(b)}(\Omega_\eps)$, in accordance with~$V_0 = \Omega_\eps$.}
    \begin{equation}
        u \in \mathcal{H}_a^{(b)}(V_m)\qquad \text{for every }m \in \{-\mathfrak{K},\dots,\mathfrak{K}\}.
    \end{equation}
    
    Then, letting
    \begin{equation}\label{EQ::ThetaExponentDefinition}
        \vartheta \coloneq \left\{\begin{aligned}
            &1-a\qquad\text{if }b\geq 0,\\
            &1+b\qquad\text{otherwise,}
        \end{aligned}\right.
    \end{equation}
    there exists a constant~$\const{CONST::PiecewiseRegularityToGlobal}>0$, which depends only on~$n$,~$\mathcal{S}$,~$L$,~$a$, and~$b$, such that
    \begin{equation}\label{EQ::PiecewiseRegularityToGlobal}
        \abs{u}_{a; \Omega_\tau}^{(b)} \leq \const{CONST::PiecewiseRegularityToGlobal} \frac{1}{\eps^{\vartheta}} \max_{m \in \{-\mathfrak{K},\dots,\mathfrak{K}\}} \abs{u}_{a; V_m}^{(b)}.
    \end{equation}
\end{lemma}

For the facility of the reader, the proof of Lemma~\ref{LEM::PiecewiseRegularityToGlobal} is postponed to Appendix~\ref{APP::TechnicalProofs} (see~\nameref{PRF::PiecewiseRegularityToGlobal}).

\begin{remark}\label{REM::Counterexample}
    The dependence on~$\eps$ of the estimate~\eqref{LEM::PiecewiseRegularityToGlobal} is inevitable, unless stronger assumptions are taken. To see that the value of the exponent~$\vartheta$ given by~\eqref{EQ::ThetaExponentDefinition} is optimal, consider the following example. Let
    \[\mathcal{S} \coloneq [0,1] \times \bigg\{\frac{1}{2}\bigg\} \subset \R^2.\]
    Clearly, for this choice we have that~$L_0 = +\infty$. We choose~$L \coloneq \frac{1}{2}$ and~$\eps_k \coloneq \frac{1}{2k}$ for~$k \in \{2,3,\dots\}$, so that~$\Omega_{\eps_k} = (0,1) \times (\frac{1}{2}-\frac{1}{2k},\frac{1}{2}+\frac{1}{2k})$. Also, let~$a \in (0,1)$ and~$u_k \colon \Omega_L \to \R$ be such that, for all~$(x,s) \in \Omega_L$,
    \[u_k(x,s) = k^{1-a}s.\]
    
    For any~$k \in \{2,3,\dots\}$, we have that~$V_m^{(k)} = [0,1] \times [\frac{1}{2}+\frac{2m-1}{2k},\frac{1}{2}+\frac{2m+1}{2k}]$, so that
    \begin{equation}\label{EQ::CounterexampleSmallDomain}
        \begin{split}
            {[u_k]}_{a; V_m^{(k)}} &= \sup_{\substack{(x,s),(x^\prime,s^\prime) \in V_m^{(k)}\\(x,s)\neq(x^\prime,s^\prime)}} \frac{\abs{u_k(x,s)-u_k(x,s^\prime)}}{\abs{(x-x^\prime, s-s^\prime)}^a} \\
                &\leq \sup_{\substack{(x,s),(x^\prime,s^\prime) \in V_m^{(k)}\\(x,s)\neq(x^\prime,s^\prime)}} \frac{k^{1-a}\abs{s-s^\prime}}{\abs{s-s^\prime}^a}\\
                & = \sup_{\substack{(x,s),(x^\prime,s^\prime) \in V_m^{(k)}\\(x,s)\neq(x^\prime,s^\prime)}} k^{1-a}\abs{s-s^\prime}^{1-a}\\
                &\leq 1.
        \end{split}
    \end{equation}

    However,
    \begin{equation}\label{EQ::CounterexampleBigDomain}
        \begin{split}
            {[u_k]}_{a; \Omega_L} &= \sup_{\substack{(x,s),(x^\prime,s^\prime) \in \Omega_L\\(x,s)\neq(x^\prime,s^\prime)}} \frac{\abs{u_k(x,s)-u_k(x,s^\prime)}}{\abs{(x-x^\prime, s-s^\prime)}^a}\\
                &\geq \frac{\abs{u_k(\frac{1}{2}, \frac{2}{3})- u_k(\frac{1}{2}, \frac{1}{3})}}{\abs{\frac{1}{3}}^a}\\
                &= 3^{a-1}k^{1-a}\\
                &= 3^{a-1}\eps_k^{a-1}. 
        \end{split}
    \end{equation}

    Namely, from~\eqref{EQ::CounterexampleSmallDomain} and~$\abs{u_k} \leq 1$, we have that
    \[\max_{m \in \{-\mathfrak{K},\dots,\mathfrak{K}\}} \abs{u_k}_{a; V_m^{(k)}}^{(-a)} \leq 2,\]
    and clearly~$u_k \in C(\Omega_L \cup \mathcal{N}_L)$. In light of this, Lemma~\ref{LEM::PiecewiseRegularityToGlobal} applies to~$u_k$, hence
    \[\abs{u_k}_{a; \Omega_L}^{(-a)} \leq C_6 \frac{1}{\eps_k^{1-a}}\max_{m \in \{-\mathfrak{K},\dots,\mathfrak{K}\}} \abs{u_k}_{a; V_m^{(k)}}^{(-a)}.\]
    On the other hand,~\eqref{EQ::CounterexampleBigDomain} gives that
    \[\abs{u_k}_{a; \Omega_L}^{(-a)} \geq 2C \frac{1}{\eps_k^{\vartheta}} \geq C \frac{1}{\eps_k^{\vartheta}} \max_{m \in \{-\mathfrak{K},\dots,\mathfrak{K}\}} \abs{u_k}_{a; V_m^{(k)}}^{(-a)},\]
    where~$\vartheta \coloneq 1-a$ is the one given by~\eqref{EQ::ThetaExponentDefinition}. 
\end{remark}

Although the estimate~\eqref{EQ::PiecewiseRegularityToGlobal} is not uniform with respect to~$\eps$, it will be useful to prove that certain functions belong to the correct space. Nevertheless, the example given by Remark~\ref{REM::Counterexample} does not exclude the possibility of uniform estimates being achieved under stronger assumptions.

A natural question regards the behavior of estimate~\eqref{EQ::PiecewiseRegularityToGlobal} in the limit as~$(a,b) \to (1,b)$ with~$b \in [0,+\infty)$, and as~$(a,b) = (a,-a) \to (1,-1)$, both of which entail that~$\vartheta \to 0$. As it is often the case, Hölder spaces (even classical ones) require special care to deal with integer exponents. For this reason, we present the uniform version of Lemma~\ref{LEM::PiecewiseRegularityToGlobal} separately in the next result. 
\begin{lemma}\label{LEM::PiecewiseLipschitzToGlobalLipschitz}
    Let~$b \in \{-1\} \cup [0,+\infty)$ and~$u \in C(\Omega_\tau \cup \mathcal{N}_\tau)$ be such that
    \begin{equation}
        u \in \mathcal{H}_1^{(b)}(V_m)\qquad \text{for every }m \in \{-\mathfrak{K},\dots,\mathfrak{K}\}.
    \end{equation}

    Then, there exists a constant~$\const{CONST::PiecewiseLipschitzToGlobalLipschitz} > 0$, which depends only on~$n$,~$\mathcal{S}$,~$L$, and~$b$, such that, for every $a \in (0,1)$,
    \begin{equation}\label{EQ::PiecewiseLipschitzToGlobalHolder}
        \abs{u}_{a; \Omega_\tau}^{(\max\{-a,b\})} \leq \const{CONST::PiecewiseLipschitzToGlobalLipschitz} \max_{m \in \{-\mathfrak{K},\dots,\mathfrak{K}\}} \abs{u}_{1; V_m}^{(b)}.
    \end{equation}

    Moreover, if~$\nabla u \in C(\Omega_\tau \cup \mathcal{N}_\tau)$, then
    \begin{equation}\label{EQ::PiecewiseLipschitzToGlobalLipschitz}
        \abs{u}_{1; \Omega_\tau}^{(b)} \leq \const{CONST::PiecewiseLipschitzToGlobalLipschitz} \max_{m \in \{-\mathfrak{K},\dots,\mathfrak{K}\}} \abs{u}_{1; V_m}^{(b)}.
    \end{equation}
\end{lemma}

For the facility of the reader, the proof of Lemma~\ref{LEM::PiecewiseLipschitzToGlobalLipschitz} is postponed to Appendix~\ref{APP::TechnicalProofs} (see~\nameref{PRF::PiecewiseLipschitzToGlobalLipschitz}).

We also present a counterpart of Lemmata~\ref{LEM::PiecewiseRegularityToGlobal} and~\ref{LEM::PiecewiseLipschitzToGlobalLipschitz} for functions that also depend on the time variable. As one may expect, the spatial estimates carry over, with the necessary modifications, to the spatiotemporal case. The details are as follows.

\begin{corollary}\label{COR::PiecewiseRegularityToGlobalWithTime}
    Let either~$a \in (0,1)$,~$b \in [-a,+\infty)$, and 
    \begin{equation}
        \vartheta \coloneq \left\{\begin{aligned}
            &1-a\qquad\text{if }b\geq 0,\\
            &1+b\qquad\text{otherwise,}
        \end{aligned}\right.
    \end{equation}
    or~$a \coloneq 1$,~$b \in \{-1\}\cup [0,+\infty)$, and~$\vartheta \coloneq 0$. 
    
    Also, let $T>0$, and~$u \in C(\Omega_\tau^T \cup \mathcal{N}_\tau^T)$ be such that 
    \begin{equation}
        u \in \mathcal{H}_a^{(b)}(V_m^T)\qquad \text{for every }m \in \{-\mathfrak{K},\dots,\mathfrak{K}\}.
    \end{equation}
    If~$a \in (0,1)$, let~$a^\prime \coloneq a$ and $b^\prime \coloneq b$. If~$a = 1$ and~$\nabla u \in C(\Omega_\tau^T \cup \mathcal{N}_\tau^T)$, let~$a^\prime \coloneq 1$ and~$b^\prime \coloneq b$, otherwise, let~$a^\prime \in (0,1)$,~$b^\prime \coloneq \max\{-a^\prime, b\}$.
    
    Then, there exists a constant~$\const{CONST::PiecewiseRegularityToGlobalWithTime}>0$, which depends only on~$n$,~$\mathcal{S}$,~$L$,~$a$, and~$b$, such that
    \begin{equation}\label{EQ::PiecewiseRegularityToGlobalWithTime}
        \abs{u}_{a^\prime; \Omega_\tau^T}^{(b^\prime)} \leq \const{CONST::PiecewiseRegularityToGlobalWithTime} \frac{1}{\eps^{\vartheta}}\max_{m \in \{-\mathfrak{K},\dots,\mathfrak{K}\}} \abs{u}_{a; V_m^T}^{(b)}.
    \end{equation}
\end{corollary}

\begin{proof}
    For any~$t \in (0,T]$ let us denote~$h_t(\cdot) \coloneq u(\cdot, t)$. Suppose~$(x,t),(x^\prime, t) \in I_\delta(V_m^T)$ for some~$\delta > 0$. Then, 
    \begin{equation}\label{EQ::PiecewiseRegularityWithTimeCaBound}
        \abs{h_t(x)-h_t(x^\prime)} = \abs{u(x,t)-u(x^\prime,t)} \leq \abs{u}_{a;I_\delta(V_m^T)} \abs{x-x^\prime}^a.
    \end{equation}
    
    Furthermore, we note that
    \begin{equation}\label{EQ::PiecewiseRegularityWithTimeC0Bound}
        \abs{h_t(x)} = \abs{u(x,t)} \leq \abs{u}_{0; I_\delta(V_m^T)}
    \end{equation}
    and, if~$a = 1$, we also have that
    \begin{equation}\label{EQ::PiecewiseRegularityWithTimeC1Bound}
        \abs{\nabla h_t(x)} = \abs{\nabla u(x,t)} \leq \abs{\nabla u}_{0; I_\delta(V_m^T)}.
    \end{equation}
    
    From~\eqref{EQ::PiecewiseRegularityWithTimeCaBound},~\eqref{EQ::PiecewiseRegularityWithTimeC0Bound}, and~\eqref{EQ::PiecewiseRegularityWithTimeC1Bound} we conclude
    \begin{equation}
        \abs{h_t}_{a; V_m}^{(b)} \leq \abs{u}_{a;V_m^T}^{(b)},
    \end{equation}
    thus, applying either Lemma~\ref{LEM::PiecewiseRegularityToGlobal} or Lemma~\ref{LEM::PiecewiseLipschitzToGlobalLipschitz},
    \begin{equation}\label{EQ::gPiecewiseRegularityToGlobal}
        \abs{h_t}_{a^\prime; \Omega_\tau}^{(b^\prime)} \leq \max\{\const{CONST::PiecewiseRegularityToGlobal}, \const{CONST::PiecewiseLipschitzToGlobalLipschitz}\} \frac{1}{\eps^\vartheta} \max_{m \in \{-\mathfrak{K},\dots,\mathfrak{K}\}}\abs{u}_{a;V_m^T}^{(b)}.
    \end{equation}

    We compute, for~$(x,t),(x^\prime, t^\prime) \in I_\delta(\Omega_\tau^T)$, and letting~$m \in \{-\mathfrak{K},\dots,\mathfrak{K}\}$ be such that~$x^\prime \in V_m$,
    \begin{equation}\label{EQ::uLipschitzHolderBoundWithhT}
        \begin{split}
            \abs{u(x,t)-u(x^\prime, t^\prime)} &\leq \abs{u(x,t)-u(x^\prime, t)} + \abs{u(x^\prime,t)-u(x^\prime, t^\prime)}\\
                &= \abs{h_t(x)-h_t(x^\prime)} + \abs{u(x^\prime,t)-u(x^\prime, t^\prime)}\\
                &\leq \abs{h_t}_{a^\prime; I_\delta(\Omega_\tau)} \abs{x-x^\prime}^{a^\prime} + \abs{u}_{a;I_\delta(V_m^T)} \abs{t-t^\prime}^{\frac{a}{2}}\\
                &\leq \abs{h_t}_{a^\prime; I_\delta(\Omega_\tau)} \abs{x-x^\prime}^{a^\prime} + \abs{u}_{a;I_\delta(V_m^T)} \abs{t-t^\prime}^{\frac{a^\prime}{2}} T^{\frac{a-a^\prime}{2}}\\
                &\leq \left(\abs{h_t}_{a^\prime; I_\delta(\Omega_\tau)} + T\abs{u}_{a;I_\delta(V_m^T)}\right) \abs{(x-x^\prime,t-t^\prime)}_P^{a^\prime}.
        \end{split}
    \end{equation}
    
    If~$a^\prime \in (0,1)$, also using~\eqref{EQ::gPiecewiseRegularityToGlobal}, we find
    \begin{equation}
        \begin{split}
            \abs{u}_{a^\prime; \Omega_\tau^T}^{(b^\prime)} &= \sup_{\delta > 0} \delta^{a^\prime + b^\prime} \abs{u}_{a^\prime; I_\delta(\Omega_\tau^T)} = \sup_{\delta > 0} \delta^{a^\prime + b^\prime}\abs{u}_{0; I_\delta(\Omega_\tau^T)} + \sup_{\delta > 0} \delta^{a^\prime + b^\prime}{[u]}_{a^\prime; I_\delta(\Omega_\tau^T)}\\
                &\leq \sup_{\delta > 0} \delta^{a^\prime + b^\prime}\max_{m \in \{-\mathfrak{K},\dots,\mathfrak{K}\}}\abs{u}_{0; I_\delta(V_m^T)} + \sup_{\delta > 0}\delta^{a^\prime+b^\prime} \sup_{\substack{(x,t),(x^\prime,t^\prime) \in I_\delta(\Omega_\tau^T)\\(x,t)\neq(x^\prime,t^\prime)}}\frac{\abs{u(x,t)-u(x^\prime,t^\prime)}}{\abs{(x-x^\prime,t-t^\prime)}_P^a}\\
                &\leq \max_{m \in \{-\mathfrak{K},\dots,\mathfrak{K}\}} \abs{u}_{0; V_m^T}^{(b^\prime)} + \sup_{\delta>0}\delta^{a^\prime+b^\prime} \left(\abs{h_t}_{a; I_\delta(\Omega_\tau)} + T\abs{u}_{a;I_\delta(V_m^T)}\right)\\
                &\leq \left(1+T+\max\{\const{CONST::PiecewiseRegularityToGlobal},\const{CONST::PiecewiseLipschitzToGlobalLipschitz}\}\frac{1}{\eps^\vartheta}\right) \max_{m \in \{-\mathfrak{K},\dots,\mathfrak{K}\}} \abs{u}_{a; V_m^T}^{(b)}\\
                &\leq (L^\vartheta(1+T)+\max\{\const{CONST::PiecewiseRegularityToGlobal},\const{CONST::PiecewiseLipschitzToGlobalLipschitz}\})\frac{1}{\eps^\vartheta}\max_{m \in \{-\mathfrak{K},\dots,\mathfrak{K}\}} \abs{u}_{a; V_m^T}^{(b)},
        \end{split}
    \end{equation}
    yielding~\eqref{EQ::PiecewiseRegularityToGlobalWithTime} with~$\const{CONST::PiecewiseRegularityToGlobalWithTime} \coloneq L^\vartheta(1+T)+\max\{\const{CONST::PiecewiseRegularityToGlobal},\const{CONST::PiecewiseLipschitzToGlobalLipschitz}\}$ as desired.

    If~$a^\prime = 1$ (therefore~$\vartheta = 0$), the estimate in~\eqref{EQ::uLipschitzHolderBoundWithhT} becomes a Lipschitz bound. This and the continuity of~$\nabla u$ entail that
    \begin{equation}
        \abs{\nabla u}_{0; I_\delta(\Omega_\tau^T)} \leq \left(\abs{h_t}_{a^\prime; I_\delta(\Omega_\tau)} + T\abs{u}_{a;I_\delta(V_m^T)}\right).
    \end{equation}
    Thus, also thanks to~\eqref{EQ::gPiecewiseRegularityToGlobal}, we have that
    \begin{equation}
        \begin{split}
            \abs{u}_{a^\prime; \Omega_\tau^T}^{(b^\prime)} &= \sup_{\delta > 0} \delta^{1 + b^\prime} \left(\abs{u}_{0; I_\delta(\Omega_\tau^T)} + \abs{\nabla u}_{0; I_\delta(\Omega_\tau^T)}\right)\\
                &\leq \sup_{\delta > 0} \delta^{1 + b^\prime} \left(\abs{u}_{0; I_\delta(\Omega_\tau^T)} + \abs{h_t}_{1; I_\delta(\Omega_\tau)} + T\abs{u}_{1;I_\delta(V_m^T)}\right)\\
                &\leq (1+T+\max\{\const{CONST::PiecewiseRegularityToGlobal},\const{CONST::PiecewiseLipschitzToGlobalLipschitz}\}) \max_{m \in \{-\mathfrak{K},\dots,\mathfrak{K}\}} \abs{u}_{a; V_m^T}^{(b)},
        \end{split}
    \end{equation}
    so that the choice~$\const{CONST::PiecewiseRegularityToGlobalWithTime} \coloneq 1+T+\max\{\const{CONST::PiecewiseRegularityToGlobal},\const{CONST::PiecewiseLipschitzToGlobalLipschitz}\}$ gives~\eqref{EQ::PiecewiseRegularityToGlobalWithTime}, concluding the proof.
\end{proof}

\subsection{The reflection operator}\label{SSEC::ReflectionOperator}
We introduce the reflection operator, which plays a central role in the regularity argument of Section~\ref{SSEC::ReflectedEquation}. 

Let~$\mathcal{R} \colon C(\Omega_\eps^T \cup \mathcal{N}_\eps^T) \to C(\Omega_\tau^T \cup \mathcal{N}_\tau^T)$ be defined by
\begin{equation}\label{EQ::RuAsComposition}
    \mathcal{R}u (x, t) \coloneq u(\rho(x), t),
\end{equation}
where~$\rho$ is defined in~\eqref{EQ::rhoDefinition}. It is clear from~\eqref{EQ::RuAsComposition} that the operator~$\mathcal{R}$ is linear and that it maps continuous functions to continuous functions.

Roughly speaking, this operator acts on functions with spatial domain~$\Omega_\eps$, and extends them by reflection to the domain~$\Omega_\tau$. A similar domain extension could be achieved by rescaling, however, such procedure would present an ill behaviour with respect to (weighted) Hölder norms due to the degeneracy of~$\Omega_\eps$ when~$\eps$ is chosen to be arbitrarily small.

In the next result we discuss some properties of~$\mathcal{R}$ that clarify its ``good behavior'' quantitatively.

\begin{proposition}\label{PROP::RTopologicalIsomorphism}
    Let~$T > 0$ and~$u \in C\left(\Omega_\eps^T \cup \mathcal{N}_\eps^T\right)$. Then,
    \begin{enumerate}[label= (\roman*)]
        \item\label{ITM::RuHolderImpliesuHolder} for every~$a \in [0,2+\alpha]$ and~$b \in [-a,+\infty)$, if~$\mathcal{R}u \in \mathcal{H}_a^{(b)}(\Omega_\tau^T)$, then~$u \in \mathcal{H}_a^{(b)}(\Omega_\eps^T)$ and 
            \begin{equation}
                \abs{u}_{a; \Omega_\eps^T}^{(b)} \leq \abs{\mathcal{R}u}_{a; \Omega_\tau^T}^{(b)}.
            \end{equation}
        \item\label{ITM::uHolderImpliesRuHolder} for every~$a \in (0,1)$ and~$b \in [-a,+\infty)$, if~$u \in \mathcal{H}_a^{(b)}(\Omega_\eps^T)$, then~$\mathcal{R}u \in \mathcal{H}_a^{(b)}(\Omega_\tau^T)$. 
            
            Moreover, there exists a constant~$\const{CONST::RuHolderEstimate} > 0$, which depends only on~$n$,~$\mathcal{S}$,~$L$,~$a$, and~$b$, such that
            \begin{equation}\label{EQ::RuHolderEstimate}
                \abs{\mathcal{R}u}_{a; \Omega_\tau^T}^{(b)} \leq \const{CONST::RuHolderEstimate} \abs{u}_{a; \Omega_\eps^T}^{(b)}.
            \end{equation}
        \item\label{ITM::uCastImpliesRuCast} if~$u \in C^\ast(\Omega_\eps^T)$, where the space~$C^\ast(\Omega_\eps^T)$ is defined by~\eqref{EQ::CstarDefinition}, and 
            \begin{equation}\label{EQ::NeumannAssumptionRuProperties}
                \partial_\nu u = 0 \qquad \text{on }\mathcal{N}_\eps^T,
            \end{equation}
            then~$\mathcal{R}u \in C^\ast(\Omega_\tau^T)$.
    \end{enumerate}
\end{proposition}
\begin{proof}
    We observe that~$u = \restr{\mathcal{R}u}{\Omega_\eps}$, from which we deduce~\ref{ITM::RuHolderImpliesuHolder}.

    To prove~\ref{ITM::uHolderImpliesRuHolder}, we let~$T > 0$ and~$\delta > 0$. Recalling~\eqref{EQ::RuAsComposition} for every~$(x,t),(x^\prime,t^\prime) \in I_\delta(\Omega_\tau^T)$ we have that
    \begin{equation}\label{EQ::rhoLipschitzTouHolderFirstStep}
        \abs{\mathcal{R}u(x,t)-\mathcal{R}u(x^\prime, t^\prime)} = \abs{u(\rho(x),t)-u(\rho(x^\prime), t^\prime)}
    \end{equation}

    From~\eqref{EQ::rhoBoundaryDistance} it follows that~$(\rho(x),t)$ and~$(\rho(x^\prime),t^\prime)$ belong to~$I_{\const{CONST::rhoBoundaryDistance}\delta}(\Omega_\eps^T)$. Plugging this information into~\eqref{EQ::rhoLipschitzTouHolderFirstStep} and using the Hölder regularity of~$u$ in~$\Omega_\eps^T$, we find
    \begin{equation}\label{EQ::rhoLipschitzTouHolderSecondStep}
        \abs{\mathcal{R}u(x,t)-\mathcal{R}u(x^\prime, t^\prime)} \leq \abs{(\rho(x)-\rho(x^\prime), t-t^\prime)}_P^a \abs{u}_{a;I_{\const{CONST::rhoBoundaryDistance}\delta}(\Omega_\eps^T)},
    \end{equation}
    where the parabolic norm~$\abs{\cdot}_P$ is defined as in~\eqref{EQ::ParabolicDistanceDefinition}.

    Since the map~$\rho$ is~$C_\rho$-Lipschitz continuous in its domain, with~$C_\rho \geq 1$ (see also~\eqref{EQ::rhoC2aDiffeomorphism}), we deduce that
    \begin{equation}
        \begin{split}
            \abs{(\rho(x)-\rho(x^\prime), t-t^\prime)}_P^a &= {\left(\abs{\rho(x)-\rho(x^\prime)}^2 + \abs{t-t^\prime}\right)}^{\frac{a}{2}}\\
                &\leq {\left(C_\rho^2 \abs{x-x^\prime}^2 + \abs{t-t^\prime}\right)}^{\frac{a}{2}}\\
                &\leq C_\rho^a \abs{(x-x^\prime, t-t^\prime)}_P^a.
        \end{split}
    \end{equation}
    From this and~\eqref{EQ::rhoLipschitzTouHolderSecondStep} we infer
    \begin{equation}
        \abs{\mathcal{R}u(x,t)-\mathcal{R}u(x^\prime, t^\prime)} \leq  C_\rho^a \abs{(x-x^\prime, t-t^\prime)}_P^a \abs{u}_{a;I_{\const{CONST::rhoBoundaryDistance}\delta}(\Omega_\eps^T)},
    \end{equation}
    hence,
    \begin{equation}\label{EQ::RuPiecewiseRegular}
        \begin{split}
            \abs{\mathcal{R}u}_{a; \Omega_\tau^T}^{(b)} &= \sup_{\delta > 0} \delta^{a+b} \abs{\mathcal{R}u}_{a; I_\delta(\Omega_\tau^T)} \\
                &\leq \sup_{\delta > 0} C_\rho^a \delta^{a+b} \abs{u}_{a; I_{\const{CONST::rhoBoundaryDistance}\delta}(\Omega_\eps^T)}\\
                &= \sup_{\delta > 0} C_\rho^a \const{CONST::rhoBoundaryDistance}^{-a-b} \delta^{a+b} \abs{u}_{a; I_\delta(\Omega_\eps^T)}\\
                &\leq C_\rho^a \const{CONST::rhoBoundaryDistance}^{-a-b} \abs{u}_{a; \Omega_\eps^T}^{(b)},
        \end{split}
    \end{equation}
    therefore,~\eqref{EQ::RuHolderEstimate} holds true with
    \begin{equation}\label{EQ::C9Definition}
        \const{CONST::RuHolderEstimate} \coloneq C_\rho^a \const{CONST::rhoBoundaryDistance}^{-a-b}.
    \end{equation}

    From now on, we assume that~$u \in C^\ast(\Omega_\eps^T)$ in order to prove~\ref{ITM::uCastImpliesRuCast}. We want to prove that the partial derivatives of~$\mathcal{R}u$ up to order~$2$ are continuous in~$\Omega_\tau^T \cup \mathcal{N}_\tau^T$. On this regard, from~\eqref{EQ::RuPiecewiseRegular} it follows that all the aforementioned partial derivatives are at least piecewise continuous.

    Let~$m \in \{-\mathfrak{K},\dots,\mathfrak{K}\}$ and~$x \in V_m \cap V_{m+1}$. We find a neighborhood~$U$ of~$x$ inside of which the Fermi coordinates~$(p_1, \dots, p_n)$ are defined. Without loss of generality we furthermore assume that~$U$ is chosen such that
    \begin{equation}
        U \subset V_m \cup V_{m+1}.
    \end{equation}
    We denote with~$\tilde{U}$ the image of~$U$ under the Fermi coordinate change, and we remark that the coordinate change is a~$C^{2,\alpha}$ diffeomorphism between~$U$ and~$\tilde{U}$, with its~$C^{2,\alpha}$ norm controlled by a constant~$C>0$ which depends only on~$n$,~$\mathcal{S}$ and~$L$. We define~$v \colon \tilde{U} \to \R$ as
    \begin{equation}\label{EQ::vIsRuInFermiCoordinates}
        v(p) \coloneq \mathcal{R}u(y)
    \end{equation}
    for every~$p = (p_1(y),\dots,p_n(y))$ and~$y \in U$. Hence,
    \begin{equation}\label{EQ::vRegularIffRuRegular}
        v \in C^2(\tilde{U}) \iff \mathcal{R}u \in C^2 (U).
    \end{equation}

    Therefore, we shall prove that~$v \in C^2(\tilde{U})$. For~$y \in U$, we write~$(y_0, s) = \Phi^{-1}(y)$ and~$p = (p_1(y),\dots,p_n(y))$. Then, thanks to~\eqref{EQ::FermiCoordinatesDefinition},
    \begin{equation}\label{EQ::FermiCoordinatesDependenceRecall}
        s = p_n(y)\qquad \text{and}\qquad p_1(y),\dots,p_{n-1}(y) \text{ do not depend on }s.
    \end{equation}
    The dependence of~$\mathcal{R}u$ on~$s$ is understood from~\eqref{EQ::rhoDefinition} and~\eqref{EQ::RuAsComposition}. In light of this and~\eqref{EQ::FermiCoordinatesDependenceRecall}, we conclude that
    \begin{equation}\label{EQ::vIsEvenInNthFermiCoordinate}
        \begin{split}
            v(p_1, \dots, p_{n-1}, (1+2m)\eps + \sigma) &= \mathcal{R}u(y_0 + ((1+2m)\eps + \sigma)\nu(y_0))\\
            &= \mathcal{R}u(y_0 + ((1+2m)\eps - \sigma)\nu(y_0))\\
            &= v(p_1, \dots, p_{n-1}, (1+2m)\eps - \sigma),
        \end{split}
    \end{equation}
    or, in other words,~$v$ is even with respect to the hyperplane~$\{p_n = (1+2m)\eps\}$. Moreover, by~\eqref{EQ::RuPiecewiseRegular},~$v$ is~$C^2$ in~$\{p_n \geq (1+2m)\eps\}$ and~$\{p_n \leq (1+2m)\eps\}$. We just need to check that the limits from above and from below agree.

    For~$j \in \{1,\dots,n-1\}$, we exploit~\eqref{EQ::vIsEvenInNthFermiCoordinate} to find
    \begin{equation}\label{EQ::vTangentialDerivativeLimits}
        \begin{split}
            \lim_{\sigma \nearrow 0} \partial_j v(p_1, \dots, p_{n-1}, (1+2m)\eps + \sigma) &= \lim_{\sigma \nearrow 0} \partial_j v(p_1, \dots, p_{n-1}, (1+2m)\eps - \sigma)\\
                &= \lim_{\sigma \searrow 0} \partial_j v(p_1, \dots, p_{n-1}, (1+2m)\eps + \sigma).
        \end{split}
    \end{equation}

    For the limit of the~$n$-th partial derivative of~$v$ we still employ~\eqref{EQ::vIsEvenInNthFermiCoordinate}, but we also make use of the assumption~\eqref{EQ::NeumannAssumptionRuProperties} to find
    \begin{equation}\label{EQ::vNormalDerivativeLimits}
        \lim_{\sigma \nearrow 0} \partial_n v(p_1, \dots, p_{n-1}, (1+2m)\eps + \sigma) = 0 = \lim_{\sigma \searrow 0} \partial_n v(p_1, \dots, p_{n-1}, (1+2m)\eps + \sigma).
    \end{equation}

    The argument for the second derivatives is analogous. In case~$i,j \in \{1,\dots,n-1\}$,
    \begin{equation}\label{EQ::vSecondTangentialDerivativeLimits}
        \begin{split}
            \lim_{\sigma \nearrow 0} \partial_{ij} v(p_1, \dots, p_{n-1}, (1+2m)\eps + \sigma) &= \lim_{\sigma \nearrow 0} \partial_{ij} v(p_1, \dots, p_{n-1}, (1+2m)\eps - \sigma)\\
                &= \lim_{\sigma \searrow 0} \partial_{ij} v(p_1, \dots, p_{n-1}, (1+2m)\eps + \sigma).
        \end{split}
    \end{equation}

    Interestingly, we do not need any additional assumption to prove the continuity of the second normal derivative, because differentiating twice changes the sign twice (in perfect accordance with the ``even'' property of the reflection). Precisely,
    \begin{equation}\label{EQ::vSecondNormalDerivativeLimits}
        \begin{split}
            \lim_{\sigma \nearrow 0} \partial_{nn} v(p_1, \dots, p_{n-1}, (1+2)\eps + \sigma) &= \lim_{\sigma \nearrow 0} \partial_{nn} v(p_1, \dots, p_{n-1}, (1+2m)\eps - \sigma)\\
                &= \lim_{\sigma \searrow 0} \partial_{nn} v(p_1, \dots, p_{n-1}, (1+2m)\eps + \sigma).
        \end{split}
    \end{equation}

    The case of mixed tangential and normal derivatives is slightly more involved. We know that
    \begin{equation}
        \partial_n v(p_1, \dots, p_{n-1}, (1+2m)\eps) = 0
    \end{equation}
    holds in the pointwise sense for every~$(p_1,\dots,p_{n-1})$, owing to~\eqref{EQ::vNormalDerivativeLimits}. Hence,
    \begin{equation}
        \partial_{jn} v(p_1, \dots, p_{n-1}, (1+2m)\eps) = 0
    \end{equation}
    holds as a limits from both sides. Thus, for~$j \in \{1,\dots,n-1\}$,
    \begin{align}\label{EQ::vSecondMixedDerivativeLimits}
        \lim_{\sigma \nearrow 0} \partial_{jn} v(p_1, \dots, p_{n-1}, (1+2m)\eps + \sigma) = 0 = \lim_{\sigma \searrow 0} \partial_{jn} v(p_1, \dots, p_{n-1}, (1+2m)\eps + \sigma),
    \end{align}
    and the same holds for~$\partial_{nj} v(p_1, \dots, p_{n-1}, (1+2m)\eps)$.

    Summing up, we see that the combination of~\eqref{EQ::vTangentialDerivativeLimits},~\eqref{EQ::vNormalDerivativeLimits},~\eqref{EQ::vSecondTangentialDerivativeLimits},~\eqref{EQ::vSecondNormalDerivativeLimits} and~\eqref{EQ::vSecondMixedDerivativeLimits}, proves that
    \[v \in C^2(\tilde{U}).\]

    Clearly,~\eqref{EQ::RuAsComposition} entails that the operator~$\mathcal{R}$ commutes with time differentiation, from which the continuity of~$\partial_t \mathcal{R} u$ plainly follows from the continuity of~$\partial_t u$, concluding the proof.
\end{proof}

Our interest is now to study how the operator~$\mathcal{R}$ interacts with differential operators. It is natural to expect geometric quantities to be involved in the relationship between~$\mathcal{R}(\Delta u)$ and~$\Delta(\mathcal{R}u)$. This is exactly the case, as shown in the next results.

\begin{lemma}\label{LEM::RuPartialDerivatives}
    Let~$u \in C^1(\Omega_\eps^T \cup \mathcal{N}_\eps^T)$ be such that
    \begin{equation}
        \partial_\nu u = 0 \qquad \text{on }\mathcal{N}_\eps^T.
    \end{equation}

    Then, for each~$m \in \{-\mathfrak{K},\dots,\mathfrak{K}\}$, in the interior of~$V_m$ there holds
    \begin{equation}\label{EQ::RuPartialDerivatives}
        \partial_j \mathcal{R}u(x,t) = \partial_j \rho_l (x) \partial_l u(\rho(x),t).
    \end{equation}

    Moreover, if~$u \in C^2(\Omega_\eps^T \cup \mathcal{N}_\eps^T)$, then, for each~$m \in \{-\mathfrak{K},\dots,\mathfrak{K}\}$, in the interior of~$V_m \times (0,T)$ there holds
    \begin{equation}\label{EQ::RuPartialSecondDerivatives}
        \partial_{ij} \mathcal{R}u(x,t) = \partial_i \rho_k (x) \partial_j \rho_l (x) \partial_{kl} u(\rho(x),t) + \partial_{ij} \rho_k (x) \partial_k u(\rho(x),t).
    \end{equation}
\end{lemma}
\begin{proof}
    First of all, from our assumptions we deduce that~$\mathcal{R}u$ is continuously differentiable (twice when~$u$ is~$C^2$) thanks to point~\ref{ITM::uCastImpliesRuCast} of Proposition~\ref{PROP::RTopologicalIsomorphism}. Since we are computing the derivatives in the interior of~$V_m$,~$\rho$ is also~$C^2$, and thus the coefficients are well defined.

    The rest of the proof is a simple computation involving repeated application of the chain rule of differentiation, which we omit.
\end{proof}

\subsection{The rescaling operator}
To complement the extension given by the reflection operator, we apply a further rescaling procedure. We define the map~$\mathfrak{z} \colon \overline{\Omega_L} \to \overline{\Omega_\tau}$ as
\begin{equation}\label{EQ::zetaDefinition}
    \mathfrak{z}(x + s \nu (x)) = x + \frac{\tau}{L} s \nu (x)
\end{equation}
for every~$x \in \mathcal{S}$ and~$s \in (-L,L)$. It is clear from~\eqref{EQ::kappaDefinition} and~\eqref{EQ::tauDefinition} that~$\frac{\tau}{L} \in (\frac{1}{3},1]$. Recalling the map~$\Phi$ as defined in~\eqref{EQ::PhiDefinition}, we have that 
\begin{equation}
    \mathfrak{z} = \Phi \circ \left(\operatorname{Id}_\mathcal{S}, s \mapsto \frac{\tau}{L}s\right) \circ \Phi^{-1},
\end{equation}
hence,
\begin{equation}\label{EQ::zetaIsC2aDiffeo}
    \mathfrak{z} \text{ is a }C^{2,\alpha}\text{ diffeomorphism,}
\end{equation}
and its regularity constant~$C_\mathfrak{z}$ is bounded by~$3C_\Phi^2$, which depends only on~$n$,~$\mathcal{S}$, and~$L$.

Let~$\mathcal{T} \colon C(\Omega_\tau^T \cup \mathcal{N}_\tau^T) \to C(\Omega_L^T \cup \mathcal{N}_L^T)$ be defined as
\begin{equation}\label{EQ::TuAsComposition}
    \mathcal{T}u(x,t) \coloneq u(\mathfrak{z}(x), t).
\end{equation}
Unlike~$\rho$, the map~$\mathfrak{z}$ is smooth, so~$\mathcal{T}$ does not need further assumptions on its argument to preserve regularity, while~$\mathcal{R}$ requires the Neumann condition~\eqref{EQ::NeumannAssumptionRuProperties}. The definition of~$\mathcal{T}$ entails that it is linear and that it preserves continuity. Some further properties of~$\mathcal{T}$ are examined in the rest of this section.

We present the counterpart of Proposition~\ref{PROP::RTopologicalIsomorphism}.
\begin{proposition}\label{PROP::TTopologicalIsomorphism}
    For every~$a \in [0,2+\alpha] \smallsetminus \N$,~$b\in [-a,+\infty)$, and~$T>0$, there exist two constants~$\const{CONST::TTopologicalIsomorphismUpper} \geq \const{CONST::TTopologicalIsomorphismLower} > 0$, both of which depend only on~$n$,~$\mathcal{S}$,~$L$, and~$a$, such that, for every~$u \in C(\Omega_\tau^T \cup \mathcal{N}_\tau^T)$, 
    \begin{align}
            \label{EQ::uHolderIffTuHolder}&u \in \mathcal{H}_a^{(b)}(\Omega_\tau^T) \iff \mathcal{T}u \in \mathcal{H}_a^{(b)}(\Omega_L^T),\\
            \label{EQ::uCastIffTuCast}&u \in C^\ast(\Omega_\tau^T) \iff \mathcal{T}u \in C^\ast(\Omega_L^T),
    \end{align}
    and, if~$u \in \mathcal{H}_a^{(b)}(\Omega^T)$, then
    \begin{equation}\label{EQ::SuHolderEstimates}
        \const{CONST::TTopologicalIsomorphismLower}\abs{u}_{a; \Omega_\tau^T}^{(b)} \leq \abs{\mathcal{T}u}_{a; \Omega_L^T}^{(b)} \leq \const{CONST::TTopologicalIsomorphismUpper}\abs{u}_{a; \Omega_\tau^T}^{(b)}.
    \end{equation}
\end{proposition}

\begin{proof}
    Since~\eqref{EQ::zetaIsC2aDiffeo} is true, we express~$\mathcal{T}u$ as a composition (see~\eqref{EQ::TuAsComposition}), which proves~\eqref{EQ::uHolderIffTuHolder}.

    Moreover, from~\eqref{EQ::TuAsComposition} it follows that
    \begin{equation}
        \frac{1}{C_\mathfrak{z}} \abs{u(x,t)-u(x^\prime,t^\prime)} \leq \abs{\mathcal{T}u(x,t)-\mathcal{T}u(x^\prime, t^\prime)} \leq C_\mathfrak{z} \abs{u(x,t)-u(x^\prime,t^\prime)},
    \end{equation}
    and the same holds for all the spatial derivatives of~$\mathcal{T}u$ up to order~$\lfloor a \rfloor$ and the time derivative of~$\mathcal{T}u$ if~$a > 2$. From this,~\eqref{EQ::SuHolderEstimates} plainly follows with~$\const{CONST::TTopologicalIsomorphismLower} \coloneq \frac{1}{C_\mathfrak{z}^a}$, and~$\const{CONST::TTopologicalIsomorphismUpper} \coloneq C_\mathfrak{z}^a$. To see that~$\const{CONST::TTopologicalIsomorphismUpper}\geq \const{CONST::TTopologicalIsomorphismLower}$ as desired, we recall that~$C_\mathfrak{z}$ is defined as an upper bound for the~$C^{2,\alpha}$ norm of both~$\mathfrak{z}$ and its inverse. Hence, we have that~$C_\mathfrak{z} \geq 1 \geq \frac{1}{C_\mathfrak{z}}$.
\end{proof}

We also have an analogous of Lemma~\ref{LEM::RuPartialDerivatives}.

\begin{lemma}
    Let~$u \in C^1(\Omega_\tau^T \cup \mathcal{N}_\tau^T)$. Then, in~$\Omega_L^T$ there holds
    \begin{equation}
        \partial_j \mathcal{T}u(x) = \partial_j \mathfrak{z}_l(x) \partial_l u(\mathfrak{z}(x)),
    \end{equation}

    Moreover, if~$u \in C^2(\Omega_\tau^T \cup \mathcal{N}_\tau^T)$, then, in~$\Omega_L^T$ there holds
    \begin{equation}\label{EQ::TuPartialSecondDerivatives}
        \partial_{ij} \mathcal{T}u(x,t) = \partial_i \mathfrak{z}_k(x) \partial_j \mathfrak{z}_l(x)\partial_{kl} u(\mathfrak{z}(x),t) + \partial_{ij} \mathfrak{z}_k (x) \partial_k u(\mathfrak{z}(x),t).
    \end{equation}
\end{lemma}
\begin{proof}
    Owing to the regularity of~$\mathfrak{z}$ and~\eqref{EQ::TuAsComposition}, the result follows from chain rule computations, which we omit.
\end{proof}

\subsection{Pullback coefficients}
We seek an elliptic operator~$\mathcal{L}$ such that, whenever~$u$ solves~\eqref{EQ::MainProblemEpsilon}, the function~$v \coloneq \mathcal{T}\mathcal{R}u$ satisfies an equation in the form~$\mathcal{L}v = \mathcal{T}\mathcal{R}f_u$. This reformulation will allow us to apply regularity theory to~$v$. Then, leveraging Corollary~\ref{COR::PiecewiseRegularityToGlobalWithTime}, we will prove Theorem~\ref{THM::epsIndependentRegularity}.

To find the operator~$\mathcal{L}$ as we described it, we construct the coefficients in two steps. The first step is given by the next result.

\begin{lemma}\label{LEM::ReflectionPullback}
    There exists a matrix field~$\mathcal{K} \colon \displaystyle{\bigcup_{m \in \{-\mathfrak{K}, \dots, \mathfrak{K}\}}} \operatorname{int}(V_m) \to \R^{n\times n}$ such that
    \begin{equation}\label{EQ::KTensorDefinition}
        \mathcal{K}_{ji}(x) \partial_j \rho_k(x) \equiv \partial_k \rho_j(x) \mathcal{K}_{ij}(x) \equiv \delta_{ik}.
    \end{equation}

    Let~$\mathcal{A} \colon \displaystyle{\bigcup_{m \in \{-\mathfrak{K}, \dots, \mathfrak{K}\}}} \operatorname{int}(V_m) \to \R^{n\times n}$ be defined as
    \begin{equation}\label{EQ::ATensorDefinition}
        \mathcal{A}_{ij}(x) \coloneq \mathcal{K}_{ik}(x) \mathcal{K}_{jk}(x).
    \end{equation}

    Then,
    \begin{equation}\label{EQ::AijContinuous}
        \text{there exists a continuous extension of }\mathcal{A}_{ij} \text{ on }\overline{\Omega}_\tau.
    \end{equation}
    
    Moreover, there exist two constants~$\const{CONST::AijEllipticUpper} > \const{CONST::AijEllipticLower} > 0$, both depending only on~$n$,~$\mathcal{S}$ and~$L$, such that
    \begin{equation}\label{EQ::AijElliptic}
        \const{CONST::AijEllipticLower}\abs{\xi}^2 \leq \mathcal{A}_{ij}(x)\xi_i\xi_j \leq \const{CONST::AijEllipticUpper}\abs{\xi}^2
    \end{equation}
    for every~$x \in \overline{\Omega}_\tau$ and~$\xi \in \R^n$.
    
    Also, there exists a constant~$\const{CONST::AijHolder} > 0$, which depends only on~$n$,~$\mathcal{S}$,~$L$, and~$\alpha$, such that
    \begin{equation}\label{EQ::AijHolder}
        \abs{\mathcal{A}}_{1; \Omega_\tau} \leq \const{CONST::AijHolder}.
    \end{equation}
\end{lemma}

\begin{proof}
    We recall~\eqref{EQ::rhoC2aDiffeomorphism}, due to which~$\partial_i \rho_j(x)$ is a non-degenerate~$C^{1,\alpha}$ matrix field in the interior of~$V_m$ for every~$m \in \{-\mathfrak{K},\dots,\mathfrak{K}\}$. This ensures that~\eqref{EQ::KTensorDefinition} is a well posed definition for~$\mathcal{K}_{ij}$ in its domain.
    
    To prove~\eqref{EQ::AijContinuous}, we introduce the matrix field
    \begin{equation}\label{EQ::ZijDefinition}
        \mathcal{Z}_{ij}(x) \coloneq \partial_i \rho_k(x)\partial_j \rho_k(x),
    \end{equation}
    which, as a consequence of~\eqref{EQ::KTensorDefinition}, is the inverse of~$\mathcal{A}_{ij}$, in the sense that
    \begin{equation}
        \mathcal{A}_{ij}(x)\mathcal{Z}_{jk}(x) \equiv \mathcal{Z}_{ij}(x)\mathcal{A}_{jk}(x) \equiv \delta_{ik}.
    \end{equation}

    We use~\eqref{EQ::ZijDefinition} to compute
    \begin{equation}
        \mathcal{Z}_{ij}(x) \xi_i \xi_j = \partial_i \rho_k(x)\partial_j \rho_k(x) \xi_i \xi_j = {(\nabla \rho(x) \xi)}_k {(\nabla \rho(x) \xi)}_k = \abs{\nabla \rho(x) \xi}^2,
    \end{equation}
    hence, thanks to~\eqref{EQ::rhoC2aDiffeomorphism},~$\mathcal{Z}_{ij}$ satisfies
    \begin{equation}
            \frac{1}{C_\rho^2}\abs{\xi}^2 \leq \mathcal{Z}_{ij}(x) \xi_i \xi_j\leq C_\rho^2\abs{\xi}^2
    \end{equation}
    in the interior of each~$V_m$ for every~$\xi \in \R^n$. Therefore,~\eqref{EQ::AijElliptic} is true in the interior of each~$V_m$ with~$\const{CONST::AijEllipticLower} \coloneq \frac{1}{C_\rho^2}$ and~$\const{CONST::AijEllipticUpper} \coloneq C_\rho^2$. In particular, the existence of a continuous extension of~$\mathcal{Z}_{ij}$ on~$\overline{\Omega}_\tau$ is enough to prove both~\eqref{EQ::AijContinuous} and~\eqref{EQ::AijElliptic}. We shall now focus on proving the existence of such extension.

    We just need to prove that~$\mathcal{Z}_{ij}$ is continuous across the reflection boundaries. Let~$m \in \{-\mathfrak{K},\dots,\mathfrak{K}-1\}$ and~$x \in V_m \cap V_{m+1}$. We construct a set of Fermi coordinates~$(p_1,\dots,p_n)$ in a neighborhood~$U$ of~$x$. We denote with~$\rho_F$ the representation of~$\rho$ in such coordinates, and we let
    \begin{equation}
        \tilde{\mathcal{Z}}_{ij}(p_1,\dots,p_n) \coloneq \partial_i {(\rho_{F})}_k(p_1,\dots,p_n)\partial_j {(\rho_{F})}_k(p_1,\dots,p_n).
    \end{equation}
    That is,~$\tilde{\mathcal{Z}}$ is the representation of~$\mathcal{Z}$ in Fermi coordinates. This entails that
    \begin{equation}
        \tilde{\mathcal{Z}} \text{ is continuous } \iff \mathcal{Z} \text{ is continuous.}
    \end{equation}

    We consider a point $y \in U$ and let~$(y_0, s) \coloneq \Phi^{-1}(y)$. Combining~\eqref{EQ::FermiCoordinatesDefinition},~\eqref{EQ::rhoTildeDefinition}, and~\eqref{EQ::rhoDefinition}, we find
    \begin{equation}
        \begin{split}
            \rho_F(p_1(y),\dots,p_{n-1}(y), p_n(y)) &= (p_1(\rho(y)),\dots,p_n(\rho(y))) \\
                &= (p_1(\rho(y_0+s\nu(y_0))),\dots,p_n(\rho(y_0+s\nu(y_0)))) \\
                &= (p_1(y_0 + \omega(s)\nu(y_0)),\dots,p_n(y_0 + \omega(s)\nu(y_0)))\\
                &= (p_1(\Phi(y_0, \omega(s))),\dots,p_n(\Phi(y_0, \omega(s))))\\
                &= (p_1(\Phi(y_0, \omega(s))),\dots,p_n(\Phi(y_0, \omega(s))))\\
                &= (p_1(y_0),\dots,p_{n-1}(y_0),\omega(s))\\
                &= (p_1(y),\dots,p_{n-1}(y),\omega(p_n(y))),
        \end{split}
    \end{equation}
    hence,~$\partial_i {(\rho_{F})}_k(p_1,\dots,p_n)$ is defined if
    \begin{equation}\label{EQ::RhoFermiDifferentiabilityCondition}
        \frac{p_n-\eps}{2\eps}\notin \N,
    \end{equation}
    and
    \begin{equation}
        \partial_i {(\rho_{F})}_k(p_1,\dots,p_n) =  \left\{\begin{aligned}
                                                        &\delta_{ik},&\text{if }i \in \{1,\dots,n-1\},\\
                                                        &{(-1)}^{\lfloor\frac{p_n-\eps}{2\eps}\rfloor}\delta_{ik},&\text{otherwise}.
                                                    \end{aligned}\right.
    \end{equation}

    This entails that~$\tilde{\mathcal{Z}}$ is also well defined under condition~\eqref{EQ::RhoFermiDifferentiabilityCondition} and
    \begin{equation}
        \tilde{\mathcal{Z}}_{ij}(p_1,\dots,p_n) \equiv \delta_{ij}.
    \end{equation}
    From this, we deduce~$\tilde{\mathcal{Z}}$ admits a continuous extension, hence~$\mathcal{Z}$ and~$\mathcal{A}$ also do.

    It remains to prove~\eqref{EQ::AijHolder}. With this purpose, we recall~\eqref{EQ::rhoC2aDiffeomorphism},~\eqref{EQ::KTensorDefinition}, and~\eqref{EQ::ATensorDefinition}, from which we infer
    \begin{equation}
        \sup_{m \in \{-\mathfrak{K},\dots,\mathfrak{K}\}} \abs{A}_{1; V_k} \leq \sup_{m \in \{-\mathfrak{K},\dots,\mathfrak{K}\}} \abs{{(\nabla \rho)}^{-1}}^2_{1; V_k} \leq C_\rho^2
    \end{equation}
    and
    \begin{equation}
        \nabla A \in C(\overline{V_m})\qquad\text{for every }m \in \{-\mathfrak{K},\dots,\mathfrak{K}\}.
    \end{equation}
    We apply Lemma~\ref{LEM::PiecewiseLipschitzToGlobalLipschitz} to $\mathcal{A}$ with~$a \coloneq \alpha$ and~$b \coloneq -1$, concluding that~\eqref{EQ::AijHolder} holds for~$\const{CONST::AijHolder} \coloneq C_\rho^2 \const{CONST::PiecewiseLipschitzToGlobalLipschitz}$ as desired.
\end{proof}

We complement Lemma~\ref{LEM::ReflectionPullback} by proving an analogous result that holds for rescalings.

\begin{lemma}\label{LEM::RescalingPullback}
    There exists a matrix field~$\mathcal{I} \colon \overline{\Omega}_L \to \R^{n\times n}$ such that
    \begin{equation}\label{EQ::ITensorDefinition}
        \mathcal{I}_{ji}(x) \partial_j \mathfrak{z}_k(x) \equiv \mathcal{I}_{ij}(x) \partial_k \mathfrak{z}_j(x) \equiv \delta_{ik}.
    \end{equation}

    Let~$\mathcal{A}^\ast \colon \overline{\Omega}_L \to \R^{n\times n}$ be defined as
    \begin{equation}\label{EQ::AastTensorDefinition}
        \mathcal{A}_{ij}^\ast(x) \coloneq \mathcal{A}_{kl}(\mathfrak{z}(x))\mathcal{I}_{ik}(x)\mathcal{I}_{jl}(x),
    \end{equation}
    where the matrix field~$\mathcal{A}$ is the one given by Lemma~\ref{LEM::ReflectionPullback}.
    
    Then, there exist two constants~$\const{CONST::AastijEllipticUpper} > \const{CONST::AastijEllipticLower} > 0$, which depend only on~$n$,~$\mathcal{S}$ and~$L$, such that
    \begin{equation}\label{EQ::AastijElliptic}
        \const{CONST::AastijEllipticLower}\abs{\xi}^2 \leq \mathcal{A}_{ij}^\ast(x)\xi_i\xi_j \leq \const{CONST::AastijEllipticUpper}\abs{\xi}^2
    \end{equation}
    for every~$x \in \overline{\Omega}_L$ and~$\xi \in \R^n$.
    
    Also, there exists a constant~$\const{CONST::AastijHolder}$, which depends only on~$n$,~$\mathcal{S}$,~$L$, and~$\alpha$, such that
    \begin{equation}\label{EQ::AastijHolder}
        \abs{\mathcal{A}^\ast}_{1+\alpha; \Omega_L} \leq \const{CONST::AastijHolder}.
    \end{equation}
\end{lemma}

\begin{proof}
    From~\eqref{EQ::zetaIsC2aDiffeo} we infer that~$\nabla \mathfrak{z}$ is a non-degenerate~$C^{1,\alpha}$ matrix field, therefore it is invertible in its domain and the matrix~$\mathcal{I}(x)$ is well defined by~\eqref{EQ::ITensorDefinition}.

    Thanks to~\eqref{EQ::ITensorDefinition} and~\eqref{EQ::zetaIsC2aDiffeo} we have that, for every~$\xi \in \R^n$,
    \begin{equation}
        \frac{1}{C_\mathfrak{z}}|\xi|\leq |\mathcal{I}(x)\xi| \leq C_{\mathfrak{z}}|\xi|.
    \end{equation}
    Moreover, using~\eqref{EQ::AastTensorDefinition} we compute
    \begin{equation}
        \mathcal{A}_{ij}^\ast(x) \xi_i \xi_j = \mathcal{A}_{kl}(\mathfrak{z}(x))\mathcal{I}_{ik}(x)\mathcal{I}_{jl}(x)\xi_i \xi_j = \mathcal{A}_{kl}(\mathfrak{z}(x)) {(I(x)\xi)}_k {(I(x)\xi)}_l,
    \end{equation}
    therefore, also recalling~\eqref{EQ::AijElliptic},
    \begin{equation}
        \frac{\const{CONST::AijEllipticLower}}{C_\mathfrak{z}^2} \leq \mathcal{A}_{ij}^\ast(x) \xi_i \xi_j \leq \const{CONST::AijEllipticUpper}C_\mathfrak{z}^2,
    \end{equation}
    which proves~\eqref{EQ::AastijElliptic} with
    \[\const{CONST::AastijEllipticLower} \coloneq \frac{\const{CONST::AijEllipticLower}}{C_{\mathfrak{z}}^2},\qquad \const{CONST::AastijEllipticUpper}\coloneq \const{CONST::AijEllipticUpper}C_{\mathfrak{z}}^2.\]

    To prove~\eqref{EQ::AastijHolder}, we start by observing that, thanks to the regularity of~$\mathfrak{z}$,
    \begin{equation}
        \abs{\mathcal{I}}_{1;\Omega_L} \leq C_\mathfrak{z}^2.
    \end{equation}
    Combining this fact with~\eqref{EQ::AijHolder} and Proposition~\ref{PROP::HolderProductRule} (applied here twice with~$a \coloneq 1$ and~$b \coloneq c_1 \coloneq c_2 \coloneq 0$), we have that
    \begin{equation}
        \begin{split}
            \abs{\mathcal{A}^\ast}_{1;\Omega_L} &\leq \constapp{CONST::HolderProductRule} \left(\abs{(\mathcal{A}\circ \mathfrak{z})\mathcal{I}}_{1;\Omega_L}\abs{\mathcal{I}}_{0;\Omega_L} + \abs{(\mathcal{A}\circ \mathfrak{z}) \mathcal{I}}_{0;\Omega_L}\abs{\mathcal{I}}_{1;\Omega_L}\right)\\
                &\leq 2\constapp{CONST::HolderProductRule}\abs{(\mathcal{A}\circ \mathfrak{z}) \mathcal{I}}_{1;\Omega_L}\abs{\mathcal{I}}_{1;\Omega_L}\\
                &\leq 2\constapp{CONST::HolderProductRule}^2 \left(\abs{\mathcal{A}\circ \mathfrak{z}}_{1;\Omega_L}\abs{\mathcal{I}}_{0;\Omega_L} + \abs{\mathcal{A}\circ \mathfrak{z}}_{0;\Omega_L}\abs{\mathcal{I}}_{1;\Omega_L}\right)\abs{\mathcal{I}}_{1;\Omega_L}\\
                &\leq 2\constapp{CONST::HolderProductRule}^2 \abs{\mathcal{A}\circ \mathfrak{z}}_{1;\Omega_L}\abs{\mathcal{I}}_{1;\Omega_L}^2\\
                &\leq 2\constapp{CONST::HolderProductRule}^2 C_\mathfrak{z}^2 \const{CONST::TTopologicalIsomorphismUpper} \const{CONST::AijHolder},
        \end{split}
    \end{equation}
    where we have also applied Proposition~\ref{PROP::TTopologicalIsomorphism} (with~$a \coloneq 1+\alpha$ and~$b \coloneq 0$) to~$\mathcal{A}(\mathfrak{z}(x)) = \mathcal{T}\mathcal{A}(x)$. Thus, setting
    \[\const{CONST::AastijHolder}\coloneq 2\constapp{CONST::HolderProductRule}^2 C_\mathfrak{z}^2 \const{CONST::TTopologicalIsomorphismUpper} \const{CONST::AijHolder}\]
    proves~\eqref{EQ::AastijHolder}.
\end{proof}

\subsection{The reflected equation}\label{SSEC::ReflectedEquation}

We collect all the results that we found so far to prove Theorem~\ref{THM::epsIndependentRegularity}. We start by combining Lemmata~\ref{LEM::ReflectionPullback} and~\ref{LEM::RescalingPullback} to find the equation that is satisfied by a reflected solution.

\begin{lemma}
    Let the operator~$\mathcal{L}$ be defined as
    \begin{equation}\label{EQ::LoperatorDefinition}
        \mathcal{L}v \coloneq \mathcal{A}_{ij}^\ast\partial_{ij}v,
    \end{equation}
    where the matrix field~$\mathcal{A}^\ast$ is the one given by Lemma~\ref{LEM::RescalingPullback}.

    Then, there exist coefficients~$B_j$ such that, for every~$a\in(2,3)$,~$b\in [-a,+\infty)$,~$T>0$ and~$u \in \mathcal{H}_a^{(b)}(\Omega_\eps^T)$ satisfying~\eqref{EQ::NeumannAssumptionRuProperties}, there holds 
    \begin{equation}\label{EQ::LTRExpression}
        \mathcal{L}\mathcal{T}\mathcal{R}u = \mathcal{T}\mathcal{R}\Delta u + B_j \mathcal{T}\mathcal{R}\partial_j u.
    \end{equation}
    
    Moreover, for every~$\eta\in (0,1)$ and~$b^\prime \geq \max\{a+b-1, -\eta\}$, there exists a constant~$\const{CONST::BjTRdjuHolder} > 0$, which depends only on~$n$,~$\mathcal{S}$,~$L$,~$\alpha$,~$T$,~$a$,~$b$,~$\eta$, and~$b^\prime$, such that 
    \begin{equation}\label{EQ::BjTRdjuHolder}
        \abs{B_j \mathcal{T}\mathcal{R}\partial_j u}_{\eta;\Omega_L^T}^{(b^\prime)} \leq \const{CONST::BjTRdjuHolder} \abs{u}_{1;\Omega_\eps^T}^{(a+b-1)}.
    \end{equation}
\end{lemma}
\begin{proof}
    For simplicity, throughout this proof we omit explicitly writing the dependence of~$u$ from the time variable~$t$.

    We start by using~\eqref{EQ::LoperatorDefinition} to find, for every~$x \in \Omega_L$,
    \begin{equation}
        \mathcal{L}\mathcal{T}\mathcal{R}u(x) = \mathcal{A}_{ij}^\ast(x) \partial_{ij}\mathcal{T}\mathcal{R}u(x).
    \end{equation}
    Hence, substituting~\eqref{EQ::TuPartialSecondDerivatives} and~\eqref{EQ::AastTensorDefinition},
    \begin{equation}
        \begin{split}   
        \mathcal{L}\mathcal{T}\mathcal{R}u(x) &= \mathcal{A}_{ij}^\ast(x)\partial_i \mathfrak{z}_k(x) \partial_j \mathfrak{z}_l(x)\partial_{kl}\mathcal{R}u(\mathfrak{z}(x)) + \mathcal{A}_{ij}^\ast(x)\partial_{ij}\mathfrak{z}_k(x)\partial_k\mathcal{R}u(\mathfrak{z}(x))\\
        &= \mathcal{A}_{pq}(\mathfrak{z}(x))\mathcal{I}_{ip}(x)\mathcal{I}_{jq}(x)\partial_i \mathfrak{z}_k(x) \partial_j \mathfrak{z}_l(x)\partial_{kl}\mathcal{R}u(\mathfrak{z}(x))+ \mathcal{A}_{ij}^\ast(x)\partial_{ij}\mathfrak{z}_k(x)\partial_k\mathcal{R}u(\mathfrak{z}(x)).
        \end{split}
    \end{equation}
    Thanks to~\eqref{EQ::ITensorDefinition}, we find that
    \begin{equation}
        \begin{split}
            \mathcal{L}\mathcal{T}\mathcal{R}u(x) &= \mathcal{A}_{pq}(\mathfrak{z}(x))\delta_{pk}\delta_{ql}\partial_{kl}\mathcal{R}u(\mathfrak{z}(x))+ \mathcal{A}_{ij}^\ast(x)\partial_{ij}\mathfrak{z}_k(x)\partial_k\mathcal{R}u(\mathfrak{z}(x))\\
                &= \mathcal{A}_{pq}(\mathfrak{z}(x))\partial_{pq}\mathcal{R}u(\mathfrak{z}(x))+ \mathcal{A}_{ij}^\ast(x)\partial_{ij}\mathfrak{z}_k(x)\partial_k\mathcal{R}u(\mathfrak{z}(x)),
        \end{split}
    \end{equation}
    thus, due to~\eqref{EQ::RuPartialDerivatives} and~\eqref{EQ::RuPartialSecondDerivatives}, we have that
    \begin{equation}\label{EQ::LTRuPullbackFirstHalf}
        \begin{split}   
            \mathcal{L}\mathcal{T}\mathcal{R}u(x) &= \mathcal{A}_{pq}(\mathfrak{z}(x))\partial_p\rho_k(\mathfrak{z}(x))\partial_q\rho_l(\mathfrak{z}(x))\partial_{kl}u(\rho(\mathfrak{z}(x)))\\
                &\qquad + \mathcal{A}_{ij}(\mathfrak{z}(x))\partial_{ij}\rho_l(\mathfrak{z}(x))\partial_l u(\rho(\mathfrak{z}(x)))\\
                &\qquad + \mathcal{A}_{ij}^\ast(x)\partial_{ij}\mathfrak{z}_k(x)\partial_k\rho_l(\mathfrak{z}(x))\partial_l u(\rho(\mathfrak{z}(x))).
        \end{split}
    \end{equation}
    Using~\eqref{EQ::KTensorDefinition} and~\eqref{EQ::ATensorDefinition}, it follows that
    \begin{equation}\label{EQ::ApqToLaplacianPullback}
        \begin{split}
           \mathcal{A}_{pq}(\mathfrak{z}(x))&\partial_p\rho_k(\mathfrak{z}(x))\partial_q\rho_l(\mathfrak{z}(x))\partial_{kl}u(\rho(\mathfrak{z}(x)))\\
                &= \mathcal{K}_{pi}(\mathfrak{z}(x))\mathcal{K}_{qi}(\mathfrak{z}(x))\partial_p\rho_k(\mathfrak{z}(x))\partial_q\rho_l(\mathfrak{z}(x))\partial_{kl}u(\rho(\mathfrak{z}(x)))\\
                &= \delta_{ik}\delta_{il}\partial_{kl}u(\rho(\mathfrak{z}(x))) = \partial_{ii}u(\rho(\mathfrak{z}(x))) = \Delta u(\rho(\mathfrak{z}(x))).
        \end{split}
    \end{equation}
    Recalling~\eqref{EQ::RuAsComposition} and~\eqref{EQ::TuAsComposition} (applied here to~$\Delta u$), we infer that
    \[\Delta u(\rho(\mathfrak{z}(x))) = \mathcal{R}\Delta u(\mathfrak{z}(x)) = \mathcal{T}\mathcal{R}\Delta u(x).\]
    Combining this new information with~\eqref{EQ::ApqToLaplacianPullback} and substituting into~\eqref{EQ::LTRuPullbackFirstHalf} yields
    \begin{align}
        \mathcal{L}\mathcal{T}\mathcal{R}u(x) &= \mathcal{T}\mathcal{R}\Delta u(x)\\
            &+\bigg(\mathcal{A}_{ij}(\mathfrak{z}(x))\partial_{ij}\rho_l(\mathfrak{z}(x)) + \mathcal{A}_{ij}^\ast(x)\partial_{ij}\mathfrak{z}_k(x)\partial_k\rho_l(\mathfrak{z}(x))\bigg)\mathcal{T}\mathcal{R}\partial_l u(x).
    \end{align}
    Finally, defining
    \begin{equation}
        B_l(x) \coloneq \bigg(\mathcal{A}_{ij}(\mathfrak{z}(x))\partial_{ij}\rho_l(\mathfrak{z}(x)) + \mathcal{A}_{ij}^\ast(x)\partial_{ij}\mathfrak{z}_k(x)\partial_k\rho_l(\mathfrak{z}(x))\bigg),
    \end{equation}
    we have proved~\eqref{EQ::LTRExpression}.

    Let
    \begin{equation}\label{EQ::fBjSRdjuDefinition}
        f \colon \Omega_\tau^T \cup \mathcal{N}_\tau^T \ni (x,t) \longmapsto B_j(\mathfrak{z}^{-1}(x)) \mathcal{R}\partial_j u(x,t) \in \R.
    \end{equation}
    We claim that, for every~$m \in \{-\mathfrak{K},\dots,\mathfrak{K}\}$,
    \[f \in H_{1}^{(a+b-1)}(V_m^T),\]
    with
    \begin{equation}\label{EQ::BjSRdjuPiecewiseHolder}
        \abs{f}_{1;V_m^T}^{(a+b-1)} \leq C \abs{\mathcal{R}u}_{1;V_m^T}^{(a+b-1)},
    \end{equation}
    where~$C>0$ depends only on~$n$,~$\mathcal{S}$,~$L$,~$\alpha$,~$T$,~$a$, and~$b$.
    
    The continuity of~$f$ is a consequence of~\eqref{EQ::zetaIsC2aDiffeo},~\eqref{EQ::LTRExpression} and~\eqref{EQ::fBjSRdjuDefinition}. Therefore, supposing~\eqref{EQ::BjSRdjuPiecewiseHolder} is true, owing to Lemma~\ref{LEM::PiecewiseLipschitzToGlobalLipschitz} (with~$a \coloneq \eta$ and~$b \coloneq b^\prime$), we would have that~\eqref{EQ::BjTRdjuHolder} holds true with
    \[\const{CONST::BjTRdjuHolder} \coloneq C\const{CONST::PiecewiseLipschitzToGlobalLipschitz}.\]

    Thus, it just remains to prove~\eqref{EQ::BjSRdjuPiecewiseHolder}. On this account, recalling~\eqref{EQ::rhoC2aDiffeomorphism},~\eqref{EQ::zetaIsC2aDiffeo},~\eqref{EQ::AijHolder} and~\eqref{EQ::AastijHolder}, we see that
    \begin{equation}
        \abs{B\circ \mathfrak{z}^{-1}}_{1;V_m^T} \leq \const{CONST::AijHolder}C_\rho + \const{CONST::AastijHolder}C_\mathfrak{z}C\rho.
    \end{equation}
    Then, using Propositions~\ref{PROP::HolderEmbeddingA} (applied to~$\mathcal{R}\partial_j u$ with~$a \coloneq 1$,~$b \coloneq a+b-1$, and~$a^\prime \coloneq 0$),~\ref{PROP::HolderEmbeddingB} (applied to~$B\circ \mathfrak{z}^{-1}$ with~$a \coloneq 1$,~$b \coloneq -1$, and~$b^\prime \coloneq 0$) and~\ref{PROP::HolderProductRule} (with~$a \coloneq 1$,~$b \coloneq a+b-1$,~$c_1 \coloneq 0$, and~$c_2 \coloneq a+b-1$),
    \begin{align}
        \abs{f}_{1;V_m^T}^{(a+b-1)} &= \abs{\left(B\circ \mathfrak{z}^{-1}\right) \mathcal{R}\partial_j u}_{1;V_m^T}^{(a+b-1)}\\
            &\leq \constapp{CONST::HolderProductRule}\left(\abs{B\circ \mathfrak{z}^{-1}}_{0;V_m^T} \abs{\mathcal{R}\partial_j u}_{1;V_m^T}^{(a+b-1)} + \abs{B\circ \mathfrak{z}^{-1}}_{1;V_m^T}^{(0)} \abs{\mathcal{R}\partial_j u}_{0;V_m^T}^{(a+b-1)}\right)\\
            &\leq \constapp{CONST::HolderProductRule}\left(\const{CONST::AijHolder}C_\rho + \const{CONST::AastijHolder}C_\mathfrak{z}C\rho\right)\left(1+\constapp{CONST::HolderEmbeddingA} \constapp{CONST::HolderEmbeddingB}\right) \abs{\mathcal{R}\partial_j u}_{1;V_m^T}^{(a+b-1)} \eqcolon C\abs{\mathcal{R}\partial_j u}_{1;V_m^T}^{(a+b-1)},
    \end{align}
    which concludes the proof.
\end{proof}

Finally, we complete the proof of Theorem~\ref{THM::epsIndependentRegularity}.

\begin{proof}[Proof of Theorem~\ref{THM::epsIndependentRegularity}]
    Let~$\eps \in (0,L)$ and~$T > 0$, and let~$u_\eps \in \mathcal{H}_{2+\alpha}^{(-\lambda)}(\Omega_\eps^T) \subset C^\ast\left(\Omega_\eps^T\right)$ be the unique classical solution\footnote{We remark that~$\lambda_2$, and thus~$\lambda$, is chosen independently of~$\eps$ because the smoothness of~$\mathcal{N}_\eps$, the interior/exterior cone condition of~$\mathcal{P}_\eps^{(T)}$, and the~$\Sigma$-wedge condition constants are bounded independently from~$\eps$. Therefore,~\cite{MR826642}*{Lemma 3} applies uniformly for~$\eps \in (0,L)$.} of~\eqref{EQ::MainProblemEpsilon}. The existence of~$u_\eps$ is given by Proposition~\ref{PROP::ExistenceRegularity}.

    We now let~$v \coloneq \mathcal{T}\mathcal{R}u$. Thanks to point~\ref{ITM::uCastImpliesRuCast} of Proposition~\ref{PROP::RTopologicalIsomorphism} and~\eqref{EQ::uCastIffTuCast}, applied respectively to~$u_\eps$ and to~$\mathcal{R}u_\eps$, we deduce that~$v \in C^\ast(\Omega_L^T)$. Also, for every~$(x,t)\in \Omega_L^T$,
    \begin{equation}\label{EQ::dTExtendedSolution}
        \partial_t v(x,t) = \partial_t u_\eps(\rho(\mathfrak{z}(x)),t) = \mathcal{T}\mathcal{R}\partial_t u_\eps(x,t).
    \end{equation}

    Owing to~\eqref{EQ::LTRExpression} and~\eqref{EQ::MainProblemEpsilon}, we also find, for every~$(x,t)\in \Omega_L^T$,
    \begin{equation}\label{EQ::LExtendedSolution}
        \mathcal{L}v(x,t) = \mathcal{T}\mathcal{R}\Delta u_\eps(x,t)+B^j\mathcal{T}\mathcal{R}\partial_j u_\eps(x,t) = \mathcal{T}\mathcal{R}f_{u_\eps}^{(\eps)}(x,t) + B^j\mathcal{T}\mathcal{R}\partial_j u_\eps(x,t).
    \end{equation}

    Besides, for every~$(x,t)\in \mathcal{B}_N^T$,
    \begin{equation}\label{EQ::dNuExtendedSolution}
        \partial_\nu v(x,t) = 0.
    \end{equation}

    It is also immediate that, for every~$(x,t)\in \mathcal{P}_L^{(T)}$,
    \begin{equation}\label{EQ::DirichletExtendedSolution}
        v(x,t) = g_\eps(\rho(\mathfrak{z}(x)),t) \eqcolon \tilde{g}(x,t).
    \end{equation}
    Collecting~\eqref{EQ::dTExtendedSolution},~\eqref{EQ::LExtendedSolution},~\eqref{EQ::dNuExtendedSolution} and~\eqref{EQ::DirichletExtendedSolution},~$v$ solves
    \begin{equation}
        \left\{\begin{aligned}
            &\partial_t v - \mathcal{L}v = \mathcal{T}\mathcal{R}f_{u_\eps}^{(\eps)}+ B^j\mathcal{T}\mathcal{R}\partial_j u_\eps,&\text{in }\Omega_L^T,\\
            &\partial_\nu v = 0,&\text{on }\mathcal{N}_L^T,\\
            &v = \tilde{g},&\text{on }\mathcal{P}_L^{(T)}.
        \end{aligned}\right.
    \end{equation}

    Thus,~\eqref{EQ::AastijElliptic} and~\eqref{EQ::AastijHolder} allow us to apply~\cite{MR826642}*{Theorem 4, point (b)} and obtain\footnote{Although~\cite{MR826642}*{Theorem 4, point (b)}, as it is stated, makes the constant~$C$ depend on the elliptic operator at play, their proof only uses the ellipticity and Hölder bounds of the operator's coefficients. Namely, because our ellipticity and Hölder bounds on the coefficients of~$\mathcal{L}$ are uniform, the constant given by the cited result is uniform.}
    \begin{equation}\label{EQ::LinearSchauderEstForExtension}
        \abs{v}_{2+\alpha; \Omega_L^T}^{(-\lambda)} \leq C \left(\abs{\mathcal{T}\mathcal{R}f_{u_\eps}^{(\eps)}}_{\alpha; \Omega_L^T}^{(2-\lambda)} + \abs{B^j\mathcal{T}\mathcal{R}\partial_j u_\eps}_{\alpha; \Omega_L^T}^{(2-\lambda)} + \abs{\tilde{g}}_{\lambda; \mathcal{P}_L^{(T)}}\right),
    \end{equation}
    for a constant~$C>0$ which depends only on~$n$,~$\mathcal{S}$,~$L$,~$\alpha$,~$T$, and~$\lambda$.

    Using point~\ref{ITM::uHolderImpliesRuHolder} of Proposition~\ref{PROP::RTopologicalIsomorphism} (applied to~$f_{u_\eps}^{(\eps)}$ with~$a \coloneq \alpha$ and~$b \coloneq 2-\lambda$), Proposition~\ref{PROP::TTopologicalIsomorphism}, and Lemma~\ref{LEM::ForcingTermCalphaBound}, we compute
    \begin{equation}
        \abs{\mathcal{T}\mathcal{R}f_{u_\eps}^{(\eps)}}_{\alpha; \Omega_L^T}^{(2-\lambda)} \leq \const{CONST::TTopologicalIsomorphismUpper}\abs{\mathcal{R}f_{u_\eps}^{(\eps)}}_{\alpha; \Omega_\tau^T}^{(2-\lambda)} \leq \const{CONST::RuHolderEstimate}\const{CONST::TTopologicalIsomorphismUpper} \abs{f_{u_\eps}^{(\eps)}}_{\alpha; \Omega_\eps^T}^{(2-\lambda)} \leq \const{CONST::RuHolderEstimate}\const{CONST::TTopologicalIsomorphismUpper}\const{CONST::ForcingTermCalphaBound} \left(1 + \abs{u}_{0;\Omega_\eps^T} + \abs{u}_{1+\alpha; \Omega_\eps^T}^{(1-\lambda)}\right),
    \end{equation}
    thus, applying Corollary~\ref{COR::NonlinearSolutionSupBound} to~$u_\eps$,
    \begin{equation}\label{EQ::TRfuBound}
        \abs{\mathcal{T}\mathcal{R}f_{u_\eps}^{(\eps)}}_{\alpha; \Omega_L^T}^{(2-\lambda)} \leq \const{CONST::RuHolderEstimate}\const{CONST::TTopologicalIsomorphismUpper}\const{CONST::ForcingTermCalphaBound}\max\{1,e^{C_L T}\} \left(1 + \abs{g}_{0; \mathcal{P}_\eps^{(T)}} + \abs{u_\eps}_{1+\alpha; \Omega_\eps^T}^{(1-\lambda)}\right).
    \end{equation}

    Recalling~\eqref{EQ::BjTRdjuHolder}, the choice~$a \coloneq 2+\alpha$,~$b \coloneq -\lambda$,~$\eta \coloneq \alpha$, and~$b^\prime \coloneq 2-\lambda$ yields
    \begin{equation}\label{EQ::BjSRdjuBound}
        \abs{B^j\mathcal{T}\mathcal{R}\partial_j u_\eps}_{\alpha; \Omega_L^T}^{(2-\lambda)} \leq \const{CONST::BjTRdjuHolder} \abs{u_\eps}_{1; \Omega_\eps^T}^{(1-\lambda)} \leq \const{CONST::BjTRdjuHolder} \abs{u_\eps}_{1+\alpha; \Omega_\eps^T}^{(1-\lambda)}.
    \end{equation}

    We now claim that
    \begin{equation}\label{EQ::gTildeBound}
        \abs{\tilde{g}}_{\lambda; \mathcal{P}_L^{(T)}} \leq \const{CONST::PiecewiseRegularityToGlobalWithTime}\abs{g}_{\lambda; \mathcal{P}_\eps^{(T)}}.
    \end{equation}
    Indeed, also employing point~\ref{ITM::uHolderImpliesRuHolder} to any~$w \in \mathcal{H}_\lambda(\Omega_\eps^T)$ such that~$w=g$ on~$\mathcal{P}_L^{(T)}$ (with~$a \coloneq \lambda$ and~$b \coloneq -\lambda$), we find
    \begin{equation}
        \abs{\tilde{g}}_{\lambda; \mathcal{P}_L^{(T)}} = \inf_{\substack{w_0 \in \mathcal{H}_\lambda(\Omega_L^T)\\w_0 = \tilde{g}\text{ on }\mathcal{P}_L^{(T)}}}\abs{w_0}_{\lambda; \Omega_L^T} \leq \const{CONST::PiecewiseRegularityToGlobalWithTime}\inf_{\substack{w \in H_\lambda(\Omega_\eps^T)\\w = g\text{ on } \mathcal{P}_\eps^{(T)}}}\abs{\mathcal{R}w}_{\lambda;\Omega_L^T} \leq \const{CONST::PiecewiseRegularityToGlobalWithTime}\inf_{\substack{w \in H_\lambda(\Omega_\eps^T)\\w = g\text{ on } \mathcal{P}_\eps^{(T)}}}\abs{w}_{\lambda;\Omega_\eps^T} = \const{CONST::PiecewiseRegularityToGlobalWithTime} \abs{g}_{\lambda; \mathcal{P}_\eps^{(T)}}.
    \end{equation}

    Substituting~\eqref{EQ::TRfuBound},~\eqref{EQ::BjSRdjuBound} and~\eqref{EQ::gTildeBound} into~\eqref{EQ::LinearSchauderEstForExtension}, we find
    \begin{equation}
        \abs{v}_{2+\alpha; \Omega_L^T}^{(-\lambda)} \leq C\left(\const{CONST::RuHolderEstimate}\const{CONST::TTopologicalIsomorphismUpper}\const{CONST::ForcingTermCalphaBound}\max\{1,e^{C_L T}\} + \const{CONST::BjTRdjuHolder} + \const{CONST::PiecewiseRegularityToGlobalWithTime}\right) \left(1 + \abs{g}_{\lambda; \mathcal{P}_L^{(T)}} + \abs{u_\eps}_{1+\alpha; \Omega_\eps^T}^{(1-\lambda)}\right),
    \end{equation}
    hence, thanks to point~\ref{ITM::RuHolderImpliesuHolder} of Proposition~\ref{PROP::RTopologicalIsomorphism} in tandem with Proposition~\ref{PROP::TTopologicalIsomorphism}, applied respectively to~$u_\eps$ and~$\mathcal{R}u_\eps$, we have that
    \begin{equation}\label{EQ::uUniformSchauderBeforeInterpolation}
        \abs{u_\eps}_{2+\alpha; \Omega_\eps^T}^{(-\lambda)} \leq C\left(\const{CONST::RuHolderEstimate}\const{CONST::TTopologicalIsomorphismUpper}\const{CONST::ForcingTermCalphaBound}\max\{1,e^{C_L T}\} + \const{CONST::BjTRdjuHolder} + \const{CONST::PiecewiseRegularityToGlobalWithTime}\right) \left(1 + \abs{g}_{\lambda; \mathcal{P}_L^{(T)}} + \abs{u_\eps}_{1+\alpha; \Omega_\eps^T}^{(1-\lambda)}\right).
    \end{equation}

    Furthermore, by applying Proposition~\ref{PROP::HolderInterpolation} (with~$a \coloneq 2+\alpha$,~$a^\prime \coloneq 0$,~$a^\ast \coloneq 1+\alpha$,~$b \coloneq b^\prime \coloneq b^\ast \coloneq 1-\lambda$, and~$\vartheta \coloneq \frac{1+\alpha}{2+\alpha}$), and Proposition~\ref{PROP::HolderEmbeddingB} (applied twice to~$u_\eps$ with~$a \coloneq 2+\alpha$,~$b \coloneq -\lambda$, and~$b^\prime \coloneq 1-\lambda$, and with~$a \coloneq 0$,~$b \coloneq 0$, and~$b^\prime \coloneq 1-\lambda$), we find
    \begin{equation}
        \abs{u_\eps}_{1+\alpha; \Omega_\eps^T}^{(1-\lambda)} \leq \constapp{CONST::HolderInterpolation} {\left(\abs{u_\eps}_{2+\alpha; \Omega_\eps^T}^{(1-\lambda)}\right)}^\vartheta {\left(\abs{u_\eps}_{0; \Omega_\eps^T}^{(1-\lambda)}\right)}^{1-\vartheta}\leq \constapp{CONST::HolderEmbeddingB} \constapp{CONST::HolderInterpolation} {\left(\abs{u_\eps}_{2+\alpha; \Omega_\eps^T}^{(-\lambda)}\right)}^\vartheta {\left(\abs{u_\eps}_{0; \Omega_\eps^T}\right)}^{1-\vartheta}.
    \end{equation}
    Thanks to this, Corollary~\ref{COR::NonlinearSolutionSupBound} and the generalized Young inequality, for every~$\delta > 0$ we have that
    \begin{equation}
        \abs{u_\eps}_{1+\alpha; \Omega_\eps^T}^{(1-\lambda)} \leq \constapp{CONST::HolderEmbeddingB} \constapp{CONST::HolderInterpolation} \left(C_\delta \abs{u_\eps}_{0; \Omega_\eps^T} + \delta \abs{u_\eps}_{2+\alpha; \Omega_\eps^T}^{(-\lambda)}\right) \leq \constapp{CONST::HolderEmbeddingB} \constapp{CONST::HolderInterpolation} \left(C_\delta e^{C_L T}\abs{g}_{\lambda; \mathcal{P}_\eps^{(T)}} + \delta \abs{u_\eps}_{2+\alpha; \Omega_\eps^T}^{(-\lambda)}\right).
    \end{equation}
    We choose 
    \[\delta \coloneq \frac{1}{2\constapp{CONST::HolderEmbeddingB} \constapp{CONST::HolderInterpolation} C\left(\const{CONST::RuHolderEstimate}\const{CONST::TTopologicalIsomorphismUpper}\const{CONST::ForcingTermCalphaBound}\max\{1,e^{C_L T}\} + \const{CONST::BjTRdjuHolder} + \const{CONST::PiecewiseRegularityToGlobalWithTime}\right)},\]
    so that~\eqref{EQ::uUniformSchauderBeforeInterpolation} becomes
    \begin{equation}
        \abs{u_\eps}_{2+\alpha; \Omega_\eps^T}^{(-\lambda)} \leq 2C\left(\const{CONST::RuHolderEstimate}\const{CONST::TTopologicalIsomorphismUpper}\const{CONST::ForcingTermCalphaBound}\max\{1,e^{C_L T}\} + \const{CONST::BjTRdjuHolder} + \const{CONST::PiecewiseRegularityToGlobalWithTime}\right)\left(1 + (1+ \constapp{CONST::HolderEmbeddingB}\constapp{CONST::HolderInterpolation} e^{C_L T} C_\delta)\abs{g}_{\lambda; \mathcal{P}_L^{(T)}}\right).
    \end{equation}

    The result follows by setting
    \[\const{CONST::UniformRegularity} \coloneq 2C\left(\const{CONST::RuHolderEstimate}\const{CONST::TTopologicalIsomorphismUpper}\const{CONST::ForcingTermCalphaBound}\max\{1,e^{C_L T}\} + \const{CONST::BjTRdjuHolder} + \const{CONST::PiecewiseRegularityToGlobalWithTime}\right) (1+ \constapp{CONST::HolderEmbeddingB}\constapp{CONST::HolderInterpolation} e^{C_L T} C_\delta).\qedhere\]
\end{proof}

\section{Uniform convergence, dimensional reduction and proof of Theorem~\ref{THM::thm1}}\label{SEC::DimensionalReduction}
This section is devoted to the proof of Theorem~\ref{THM::thm1}.
\subsection{Asymptotic estimates}
We prove some quantitative asymptotic estimates as~$\eps \to 0$, which will be used in Section~\ref{SSEC::LimitEquation} to prove Theorem~\ref{THM::thm1}. We start by showing that, for small values of~$\eps$, the pointwise value of a Hölder function and its transverse average are comparable up to a small error.

\begin{lemma}\label{LEM::AvgApproximation}
    Let~$a \in (0,2)\smallsetminus \{1\}$,~$\eps \in (0,L)$, and~$u \in \mathcal{H}_a\left(\Omega_\eps\right)$. Let
    \begin{equation}
        U \colon \mathcal{S} \ni x \longmapsto \frac{1}{2\eps}\int_{-\eps}^{\eps} u(x+s\nu(x))\,ds \in \R.
    \end{equation}
    
    Furthermore, if~$a \in (1,2)$ also assume that, for every~$x \in \mathcal{N}_\eps$,
    \begin{equation}\label{EQ::AvgApproximationNeumann}
        \partial_\nu u(x) = 0.
    \end{equation}

    Then, for every~$x \in \mathcal{S}$ and~$s \in (-\eps, \eps)$, we have that
    \begin{equation}\label{EQ::AvgApproximationC0alphaBound}
        \abs{u^+(x + s \nu(x)) - U^+(x)} \leq \abs{u(x + s \nu(x)) - U(x)} \leq 2\eps^{a} \abs{\partial_{\nu}^{\lfloor a \rfloor} u}_{a-\lfloor a \rfloor; \Omega_\eps}.
    \end{equation}
\end{lemma}
\begin{proof}
        
    Let~$x \in \mathcal{S}$. By the Mean Value Theorem for integrals we infer that there exists~$s^\ast \in (-\eps, \eps)$ such that~$u(x+s^\ast \nu(x)) = U(x)$.

    If~$a \in (0,1)$, by the Hölder continuity of~$u$, we get
    \begin{equation}\label{EQ::AvgApproximationC0alphaBoundPartial}
        \abs{u(x + s \nu(x)) - U(x)} \leq \abs{s - s^\ast}^a {\left[ u \right]}_{a; \Omega_\eps} \leq 2^{a} \eps^a {\left[ u \right]}_{a; \Omega_\eps}.
    \end{equation}

    If~$a\in(1,2)$, we use the Neumann condition~\eqref{EQ::AvgApproximationNeumann} to deduce that
    \begin{equation}\label{EQ::AvgApproximationGradientC0alphaBound}
        \abs{\partial_{\nu(x)}u(x + s \nu(x))} \leq {(\eps - \abs{s})}^{a-1} {\left[ \partial_\nu u \right]}_{a-1; \Omega_\eps} \leq \eps^{a-1}{\left[ \partial_\nu u \right]}_{a-1; \Omega_\eps},
    \end{equation}
    hence,
    \begin{equation}\label{EQ::AvgApproximationC1alphaBound}
        \abs{u(x + s \nu(x)) - U(x)} \leq \int_{\min\{s, s^\ast\}}^{\max\{s, s^\ast\}} \abs{\partial_{\nu(x)}u(x + \sigma \nu(x))}\,d\sigma \leq 2\eps^{a}{\left[ \partial_\nu u \right]}_{a-1; \Omega_\eps}.
    \end{equation}

    The result follows from~\eqref{EQ::AvgApproximationC0alphaBoundPartial} and~\eqref{EQ::AvgApproximationC1alphaBound} upon noticing that the positive part map is Lipschitz continuous, and its Lipschitz constant is~$1$.
\end{proof}

We also show how the curvature related terms that appear in the expression of the Laplace-Beltrami operator on~$\mathcal{S}(s)$ and on~$\mathcal{S}$ are quantitatively close when~$s$ is small.

\begin{lemma}\label{LEM::LaplaceBeltramiApproximation}
    Let~$\mathcal{S}$ be a compact, connected, and orientable~$C^{3}$ hypersurface with boundary embedded into~$\R^n$. Let~$L_0$ be as in~\eqref{EQ::L0Definition} and~$L \in (0,L_0)$.
    
    Let~$\eps \in (0,L)$,~$A \subset \Omega_\eps$ be an open set, and~$u \in C^2(A)$. Define~$\tilde{u} \coloneq u \circ \restr{\Phi}{\Phi^{-1}(A)}$, where~$\Phi$ is as in~\eqref{EQ::PhiDefinition}.
    
    Then, there exists a constant~$\const{CONST::LaplaceBeltramiApproximation}>0$, which depends only on~$\mathcal{S}$ and~$L$, such that for every~$X \in A$, with~$(x,s) = \Phi^{-1}(X)$,
    \begin{equation}\label{EQ::LaplaceBeltramiApproximation}
        \abs{\Delta_{\mathcal{S}(s)}u(X) - \Delta_\mathcal{S} \tilde{u}(x,s)} \leq \const{CONST::LaplaceBeltramiApproximation} s \sup_{y \in A} \left(\abs{\nabla u(y)} + \abs{D^2 u(y)}\right).
    \end{equation}
\end{lemma}
\begin{proof}
    Let~$X \in A$. Let~$(p_1,\dots,p_n)$ be a set of Fermi coordinates in a neighborhood of~$X$. Then, using the notation of Proposition~\ref{PROP::FermiCoordinatesChristoffelSymbols}, we have
    \begin{equation}\label{EQ::LaplaceBeltramiSigmaS}
        \Delta_{\mathcal{S}(s)}u(X) = \tilde{g}^{ij}(X) \left( \partial_i \partial_j u(X) - \Gamma_{ij}^k(x,s) \partial_k u(X) \right),
    \end{equation}
    and
    \begin{equation}\label{EQ::LaplaceBeltramiSigma0}
        \Delta_{\mathcal{S}}\tilde{u}(x,s) = g^{ij}(x) \left( \partial_i \partial_j \tilde{u}(x,s) - \Gamma_{ij}^k(x,0) \partial_k \tilde{u}(x,s)\right).
    \end{equation}

    From the definition of~$\tilde{u}$, (see also~\eqref{EQ::FermiCoordinatesDefinition}),
    \begin{equation}\label{EQ::PartialDerivativeCorrespondence}
        \partial_j \tilde{u}(x,s) = \partial_j u(X).
    \end{equation}
    
    Moreover, due to~\eqref{EQ::FermiCoordinatesMetric},
    \begin{equation}
        g_{ij}(x) - \tilde{g}_{ij}(X) = s \left( h_{ik}(x)g_{kj}(x) + h_{ki}(x)g_{jk}(x) - s h_{ik}(x)h_{jl}(x)g_{kl}(x)\right).
    \end{equation}
    We remark the fact that the quantities~$g_{ij}$,~$\tilde{g}_{ij}$,~$g^{ij}$,~$\tilde{g}^{ij}$, and~$h_{ij}$ are uniformly bounded. Hence, since~$\abs{s} < L$, we have
    \begin{equation}\label{EQ::gijInverseEstimate}
        \abs{g^{ij}(x) - \tilde{g}^{ij}(X)} = \abs{g^{ik}(x) \left(\tilde{g}_{kl}(X) - g_{kl}(x)\right) \tilde{g}^{lj}(X)} \leq C s,
    \end{equation}
    for~$C > 0$ that depends only on~$\mathcal{S}$ and~$L$.

    Also,
    \begin{align}
        \partial_i \tilde{g}_{jk}(X) =  &\partial_i \left((\delta_{jp} - s h_{jp}(x))(\delta_{kq} - s h_{kq}(x)) g_{pq}(x)\right)\\
                                     =  &\partial_i g_{pq}(x)(\delta_{jp} - s h_{jp}(x))(\delta_{kq} - s h_{kq}(x))\\
                                        &- s\partial_i h_{jp}(x)(\delta_{kq} - s h_{kq}(x)) g_{pq}(x)\\
                                        &- s\partial_i h_{kq}(x)(\delta_{jp} - s h_{jp}(x)) g_{pq}(x).
    \end{align}

    Since~$\mathcal{S}$ is~$C^3$, it follows that~$h_{ij}$ is~$C^1$. As a consequence,
    \begin{equation}\label{EQ::gijDerivativeEstimate}
        \abs{\partial_i \tilde{g}_{jk}(X) - \partial_i g_{jk}(x)} \leq K s,
    \end{equation}
    for some~$K > 0$ which depends only on~$\mathcal{S}$ and~$L$.

    Combining~\eqref{EQ::gijInverseEstimate} and~\eqref{EQ::gijDerivativeEstimate} yields
    \begin{equation}
        \abs{\Gamma_{ij}^k(x,s) - \Gamma_{ij}^k(x,0)} \leq CK s^2 \leq CKL s.
    \end{equation}
    Recalling~\eqref{EQ::PartialDerivativeCorrespondence} and~\eqref{EQ::gijInverseEstimate}, we plug this new information into~\eqref{EQ::LaplaceBeltramiSigmaS} and~\eqref{EQ::LaplaceBeltramiSigma0}, obtaining
    \begin{align}
        \abs{\Delta_{\mathcal{S}(s)}u(X) - \Delta_\mathcal{S} \tilde{u}(x,s)} &= \big|\tilde{g}^{ij}(X) \left( \partial_i \partial_j u(X) - \Gamma_{ij}^k(x,s) \partial_k u(X) \right)\\
            &\qquad- g^{ij}(x) \left( \partial_i \partial_j \tilde{u}(x,s) - \Gamma_{ij}^k(x,0) \partial_k \tilde{u}(x,s)\right)\big|\\
            &\leq \abs{\left(\tilde{g}^{ij}(X) - g^{ij}(x)\right)\left( \partial_i \partial_j u(X) - \Gamma_{ij}^k(x,s) \partial_k u(X) \right)}\\
            &\qquad + \abs{g^{ij}(x)\left(\Gamma_{ij}^k(x,s) \partial_k u(X) - \Gamma_{ij}^k(x,0) \partial_k \tilde{u}(x,s) \right)}\\
            &\leq Cs\left(\sup_{y \in A}\abs{D^2 u(y)} + \abs{\Gamma_{ij}^k}_{0; \mathcal{S} \times (-L,L)} \sup_{y \in A}\abs{\nabla u(y)}\right)\\
            &\qquad + CKLs\abs{g^{ij}}_{0; \Omega_\eps}\sup_{y \in A}\abs{\nabla u(y)}\\
            &\leq C\left(1+\abs{\Gamma_{ij}^k}_{0; \mathcal{S} \times (-L,L)}+KL\abs{g^{ij}}_{0; \Omega_\eps}\right) s \sup_{y \in A} \left(\abs{\nabla u(y)} + \abs{D^2 u(y)}\right).
    \end{align}
    
    This implies~\eqref{EQ::LaplaceBeltramiApproximation} with
    \begin{equation}
        \const{CONST::LaplaceBeltramiApproximation} \coloneq C\left(1+\abs{\Gamma_{ij}^k}_{0; \mathcal{S} \times (-L,L)}+KL\abs{g^{ij}}_{0; \Omega_\eps}\right).\qedhere
    \end{equation}
\end{proof}

\subsection{The limit equation}\label{SSEC::LimitEquation}
In this section we present our proof of Theorem~\ref{THM::thm1}. To this end, we first discuss how the uniform estimate given by Theorem~\ref{THM::epsIndependentRegularity} allows us to find that~$U_\eps$ satisfies the equation in~\eqref{EQ::MainProblemS} up to an asymptotically infinitesimal residual.

\begin{lemma}\label{LEM::UepsLimitEquation}
    Let~$\mathcal{S}$ be a compact, connected, and orientable~$C^{3,\alpha}$ hypersurface with boundary embedded into~$\R^n$. Let~$L_0$ be as in~\eqref{EQ::L0Definition} and~$L \in (0,L_0)$.
    
    Suppose that the family~${\{K_\eps\}}_{\eps \in (0,L)}$ satisfies~\eqref{EQ::KernelAssumptionL1} and~\eqref{EQ::KernelAssumptionHolder} for a constant~$C_L$ independent of~$\eps$, and that~$\psi$ satisfies~\eqref{EQ::psiAssumptionZero} and~\eqref{EQ::psiAssumptionLipschitz} for a constant~$C_\psi$.
    
    Furthermore, let~$\lambda \in (0,1)$ and assume that for every~$T>0$ there exists a constant~$G_T > 0$ such that, for every~$\eps \in (0,L)$, we have that~$\abs{g_\eps}_{\lambda; \mathcal{P}_\eps^T} \leq G_T$.
    
    Then, for every~$A \Subset \mathcal{S}$,~$T> t_0 >0$, and~$(x,t) \in A \times (t_0,T]$, we have that
    \begin{equation}
        \lim_{\eps \to 0} \left(\partial_t U_\eps(x,t) - \Delta_{\mathcal{S}}U_\eps(x,t) - f^\ast_{U_\eps}(x,t)\right) = 0.
    \end{equation}
\end{lemma}
\begin{proof}
    Since~$A \Subset \mathcal{S}$, there exists~$\delta > 0$ such that
    \begin{equation}
        A_\eps \coloneq \Phi(A, [-\eps,\eps]) \subset I_\delta(\Omega_\eps),
    \end{equation}
    where~$\Phi$ is as in~\eqref{EQ::PhiDefinition}. Without loss of generality we choose~$\delta < t_0^2$. Therefore, using~\eqref{EQ::NonlinearSchauderEpsilon},
    \begin{equation}\label{EQ::uepsUniformBound}
        \abs{u_\eps}_{2+\alpha; A \times (t_0,T]} \leq \const{CONST::UniformRegularity} \delta^{\lambda-2-\alpha} (1+G_T),
    \end{equation}
    namely, all the Euclidean partial derivatives of~$u_\eps$ up to second order and its time derivative are continuous in~$\Phi(A, [-\eps,\eps]) \times (t_0,T]$. 

    Thus, we see that, for any~$(x,t) \in A \times (t_0,T]$, 
    \begin{equation}\label{EQ::partialTuULimit}
        \partial_t U_\eps(x,t) = \partial_t \left(\frac{1}{2\eps}\int_{-\eps}^\eps u_\eps(x+s\nu(x),t)\,ds\right) = \frac{1}{2\eps}\int_{-\eps}^\eps \partial_t u_\eps(x+s\nu(x),t)\,ds.
    \end{equation}

    Also, letting~$\tilde{u}_\eps(x,s,t) \coloneq u_\eps(x+s\nu(x),t)$, there holds
    \[\partial_s \tilde{u}_\eps(x,s,t) = \nabla u_\eps(x+s\nu(x),t) \cdot \nu(x) = \partial_{\nu(x)}u_\eps(x+s\nu(x),t),\]
    and, along the same lines,
    \[\partial_s^2 \tilde{u}_\eps(x,s,t) = \left(D^2 u_\eps(x+s\nu(x),t) \nu(x)\right) \cdot \nu(x) = \partial_{\nu(x)}^2 u_\eps(x+s\nu(x),t).\]
    Hence, Proposition~\ref{PROP::PacardLaplacianDecomposition} gives that
    \begin{equation}\label{EQ::PacardDecompositionuTilde}
        \Delta u_\eps(x+s\nu(x),t) = \partial_s^2 \tilde{u}_\eps(x,s,t) + \Delta_{\mathcal{S}(s)} u_\eps(x+s\nu(x),t) - H_s(x+s\nu(x)) \partial_s \tilde{u}_\eps(x,s,t).
    \end{equation}
    
    Besides,
    \begin{equation}
        \Delta_\mathcal{S}U_\eps(x,t)= \Delta_\mathcal{S} \left(\frac{1}{2\eps}\int_{-\eps}^\eps u_\eps(x+s\nu(x),t)\,ds\right) = \frac{1}{2\eps}\int_{-\eps}^\eps \Delta_\mathcal{S}u_\eps(x+s\nu(x),t)\,ds,
    \end{equation}
    so that, also applying Lemma~\ref{LEM::LaplaceBeltramiApproximation}, Proposition~\ref{PROP::HolderEmbeddingB} (with~$a \coloneq 2+\alpha$, $b \coloneq -2-\alpha$, and~$b^\prime \coloneq -2$), Proposition~\ref{PROP::HolderEmbeddingA} (with~$a \coloneq 2+\alpha$,~$a^\prime \coloneq 2$, and~$b \coloneq -2$), and~\eqref{EQ::uepsUniformBound},
    \begin{equation}\label{EQ::uULaplaceBeltramiEstimate}
        \begin{split}
            \bigg|\Delta_\mathcal{S}U_\eps(x,t) &- \frac{1}{2\eps}\int_{-\eps}^\eps \Delta_{\mathcal{S}(s)}u_\eps(x+s\nu(x),t)\,ds\bigg|\\
                &= \frac{1}{2\eps}\abs{\int_{-\eps}^\eps (\Delta_\mathcal{S} - \Delta_{\mathcal{S}(s)})u_\eps(x+s\nu(x),t) \,ds}\\
                &\leq \frac{1}{2\eps}\int_{-\eps}^\eps \abs{(\Delta_\mathcal{S} - \Delta_{\mathcal{S}(s)})u_\eps(x+s\nu(x),t)}\,ds\\
                &\leq \frac{1}{2\eps}\int_{-\eps}^\eps \const{CONST::LaplaceBeltramiApproximation} \abs{s} \constapp{CONST::HolderEmbeddingA}\constapp{CONST::HolderEmbeddingB} \const{CONST::UniformRegularity} \delta^{\lambda-2-\alpha} (1+G_T)\,ds\\
                &\leq \frac{1}{2} \const{CONST::LaplaceBeltramiApproximation} \constapp{CONST::HolderEmbeddingA}\constapp{CONST::HolderEmbeddingB} \const{CONST::UniformRegularity} \delta^{\lambda-2-\alpha} (1+G_T)\eps.
        \end{split}
    \end{equation}

    Also, thanks to the Neumann condition,
    \begin{equation}\label{EQ::uUdnu2Estimate}
        \frac{1}{2\eps}\int_{-\eps}^\eps \partial_s^2 \tilde{u}_\eps(x,s,t)\,ds = {\left[\frac{\partial_s \tilde{u}_\eps(x,s,t)}{2\eps}\right]} _{-\eps}^{\eps} = 0.
    \end{equation}
    
    Using~\eqref{EQ::AvgApproximationGradientC0alphaBound} with~$a \coloneq 1+\lambda$, we find\footnote{Although~\eqref{EQ::AvgApproximationGradientC0alphaBound} is proved in Lemma~\ref{LEM::AvgApproximation} by assuming that~$u \in \mathcal{H}_a(\Omega_\eps)$ for~$a > 1$ (which does not hold in our case), one can see that in order to obtain~\eqref{EQ::AvgApproximationGradientC0alphaBound} at a given point~$x + s\nu(x) \in \Omega_\eps$, the continuity of~$\partial_\nu u$ is only required in the set~$\{x + \sigma \nu(x) \colon \sigma \in (-\eps, \eps)\}$. This weaker condition holds because we assumed that~$(x,t) \in A \times (t_0,T]$.} 
    \begin{equation}
        \abs{\frac{1}{2\eps}\int_{-\eps}^\eps H_s(x+s\nu(x)) \partial_s \tilde{u}_\eps(x,s,t)\,ds} \leq \abs{H_s}_{0; \Omega_L}\eps^\lambda\abs{\partial_\nu u}_{\lambda;A_\eps},
    \end{equation}
    hence by Proposition~\ref{PROP::HolderEmbeddingB} (with~$a \coloneq 2+\alpha$, $b \coloneq -2-\alpha$, and~$b^\prime \coloneq -\lambda$), Proposition~\ref{PROP::HolderEmbeddingA} (with~$a \coloneq 2+\alpha$, $a^\prime \coloneq \lambda$, and~$b \coloneq -\lambda$), and~\eqref{EQ::uepsUniformBound},
    \begin{equation}\label{EQ::uUdnuEstimate}
        \abs{\frac{1}{2\eps}\int_{-\eps}^\eps H_s(x+s\nu(x)) \partial_s \tilde{u}_\eps(x,s,t)\,ds} \leq \constapp{CONST::HolderEmbeddingA}\constapp{CONST::HolderEmbeddingB}\abs{H_s}_{0; \Omega_L}\const{CONST::UniformRegularity} \delta^{\lambda-2-\alpha} (1+G_T)\eps^\lambda.
    \end{equation}

    Thus, combining~\eqref{EQ::PacardDecompositionuTilde},~\eqref{EQ::uULaplaceBeltramiEstimate},~\eqref{EQ::uUdnu2Estimate}, and~\eqref{EQ::uUdnuEstimate}, we find that
    \begin{equation}\label{EQ::LaplacianuULimit}
        \begin{split}
            \lim_{\eps \to 0} \bigg|\Delta_\mathcal{S}&U_\eps(x,t) - \frac{1}{2\eps}\int_{-\eps}^\eps \Delta u_\eps(x + s\nu(x),t)\,ds\bigg|\\
                &= \lim_{\eps \to 0} \bigg|\Delta_\mathcal{S}U_\eps(x,t) - \frac{1}{2\eps}\int_{-\eps}^\eps \Delta_{\mathcal{S}(s)} u_\eps(x + s\nu(x),t)\,ds\\
                &\qquad - \frac{1}{2\eps}\int_{-\eps}^\eps \partial_s^2 \tilde{u}_\eps(x,s,t)\,ds + \frac{1}{2\eps}\int_{-\eps}^\eps H_s(x+s\nu(x)) \partial_s \tilde{u}_\eps(x,s,t)\,ds\bigg|\\
                &\leq \lim_{\eps \to 0}\left( \bigg|\Delta_\mathcal{S}U_\eps(x,t) - \frac{1}{2\eps}\int_{-\eps}^\eps \Delta_{\mathcal{S}(s)} u_\eps(x + s\nu(x),t)\,ds\bigg|\right.\\
                &\qquad \left.+ \abs{\frac{1}{2\eps}\int_{-\eps}^\eps H_s(x+s\nu(x)) \partial_s \tilde{u}_\eps(x,s,t)\,ds}\right)\\
                &\leq \lim_{\eps \to 0} \left(\left(\abs{H_s}_{0; \Omega_L} + \frac{1}{4} \const{CONST::LaplaceBeltramiApproximation}  L^{1-\lambda}\right)\const{CONST::UniformRegularity}\constapp{CONST::HolderEmbeddingB}\constapp{CONST::HolderEmbeddingA}\delta^{\lambda-2-\alpha}(1+G_T)\eps^\lambda\right) = 0.
        \end{split}
    \end{equation}

    We now focus on the nonlocal term and observe that, by Proposition~\ref{PROP::FermiCoordinatesJacobian},
    \begin{align}
        \frac{1}{2\eps}&\int_{-\eps}^\eps \int_{\Omega_\eps} K_\eps(x+s\nu(x),z)u_\eps^+(z,t)\,dzds\\
            &= \frac{1}{2\eps}\int_{-\eps}^\eps \int_{-\eps}^\eps\int_{\mathcal{S}} K_\eps(x+s\nu(x),y+\sigma\nu(y)) u_\eps^+(y+\sigma\nu(y),t)\prod_{j=1}^{n-1}(1-\sigma\kappa_j(y))\,d\mathcal{H}^{n-1}(y)d\sigma ds,
    \end{align}
    hence,
    \begin{equation}\label{EQ::KepsEstimateCurvature}
        \begin{split}
            \bigg|&\frac{1}{2\eps}\int_{-\eps}^\eps \int_{\Omega_\eps} K_\eps(x+s\nu(x),z)u_\eps^+(z,t)\,dzds\\
                &\qquad - \frac{1}{2\eps}\int_{-\eps}^\eps \int_{-\eps}^\eps\int_{\mathcal{S}} K_\eps(x+s\nu(x),y+\sigma\nu(y)) u_\eps^+(y+\sigma\nu(y),t)\,d\mathcal{H}^{n-1}(y)d\sigma ds\bigg|\\
                &= \bigg|\frac{1}{2\eps}\int_{-\eps}^\eps \int_{-\eps}^\eps\int_{\mathcal{S}} K_\eps(x+s\nu(x),y+\sigma\nu(y))u_\eps^+(y+\sigma\nu(y),t)\\
                &\qquad \cdot \bigg(1-\prod_{j=1}^{n-1}(1-\sigma\kappa_j(y))\bigg)d\mathcal{H}^{n-1}(y)d\sigma ds\bigg|\\ 
                &\leq \frac{1}{2\eps}\abs{u_\eps}_{0; \Omega_\eps} \left(1-{\left(1-\frac{\eps}{L_0}\right)}^{n-1}\right) \int_{-\eps}^\eps \int_{-\eps}^\eps\int_{\mathcal{S}} K_\eps(x+s\nu(x),y+\sigma\nu(y))\,d\mathcal{H}^{n-1}(y)d\sigma ds\\
                &\leq e^{C_L T}G_T C_{n,L_0}\eps 2\int_{\mathcal{S}}K^\ast(x,y)\,d\mathcal{H}^{n-1}(y),
        \end{split}
    \end{equation}
    where we have also used Corollary~\ref{COR::NonlinearSolutionSupBound} (applied to~$u_\eps$ with~$C_\# \coloneq C_L$ and~$h(x,t) \coloneq \psi(x,t,u_\eps,\nabla u_\eps)$) and~\eqref{EQ::KastDefinition}. In particular, the last inequality holds only when~$\eps$ is small enough, which does not pose a limitation to our argument.
    
    We also find, due to Lemma~\ref{LEM::AvgApproximation} (with~$a\coloneq\lambda$), that
    \begin{equation}\label{EQ::KepsEstimateUu}
        \begin{split}
            \bigg|\frac{1}{2\eps}&\int_{-\eps}^\eps \int_{-\eps}^\eps\int_{\mathcal{S}} K_\eps(x+s\nu(x),y+\sigma\nu(y))u_\eps^+(y+\sigma\nu(y),t)\,d\mathcal{H}^{n-1}(y)d\sigma ds\\
                &\qquad - \frac{1}{2\eps}\int_{-\eps}^\eps \int_{-\eps}^\eps\int_{\mathcal{S}} K_\eps(x+s\nu(x),y+\sigma\nu(y))U_\eps^+(y,t)\,d\mathcal{H}^{n-1}(y)d\sigma ds\bigg|\\
                &\qquad \leq 2\eps^{\lambda}\abs{u_\eps}_{\lambda;\Omega_\eps}\frac{1}{2\eps}\int_{-\eps}^\eps \int_{-\eps}^\eps\int_{\mathcal{S}} K_\eps(x+s\nu(x),y+\sigma\nu(y))\,d\mathcal{H}^{n-1}(y)d\sigma ds\\
                &\leq 4\eps^\lambda \constapp{CONST::HolderEmbeddingA}\const{CONST::UniformRegularity}(1+G_T)\int_{\mathcal{S}}K^\ast(x,y)\,d\mathcal{H}^{n-1}(y).
        \end{split}
    \end{equation}

    Applying~\eqref{EQ::KastDefinition}, the Dominated Convergence Theorem, Proposition~\ref{PROP::FermiCoordinatesJacobian}, and~\eqref{EQ::KernelAssumptionL1} we have that
    \begin{equation}\label{EQ::KastIsL1}
        \begin{split}
            \int_{\mathcal{S}}K^\ast&(x,y)\,d\mathcal{H}^{n-1}(y) = \int_{\mathcal{S}}\left(\lim_{\eps \to 0}\frac{1}{2\eps}\int_{-\eps}^\eps \int_{-\eps}^\eps K_\eps(x+s\nu(x),y+\sigma\nu(y))\,dsd\sigma\right) d\mathcal{H}^{n-1}(y)\\
                & = \lim_{\eps \to 0} \frac{1}{2\eps} \int_{\mathcal{S}}\int_{-\eps}^\eps \int_{-\eps}^\eps K_\eps(x+s\nu(x),y+\sigma\nu(y))\,dsd\sigma d\mathcal{H}^{n-1}(y)\\
                & = \lim_{\eps \to 0} \frac{1}{2\eps} \int_{\mathcal{S}}\int_{-\eps}^\eps \int_{-\eps}^\eps K_\eps(x+s\nu(x),y+\sigma\nu(y)) \frac{\prod_{j=1}^{n-1}(1-\sigma \kappa_j(y))}{\prod_{j=1}^{n-1}(1-\sigma \kappa_j(y))}\,dsd\sigma d\mathcal{H}^{n-1}(y)\\
                & \leq \lim_{\eps \to 0} \frac{1}{2\eps} \int_{\mathcal{S}}\int_{-\eps}^\eps \int_{-\eps}^\eps K_\eps(x+s\nu(x),y+\sigma\nu(y)) \frac{\prod_{j=1}^{n-1}(1-\sigma \kappa_j(y))}{{(1-\frac{\eps}{L_0})}^{n-1}}\,dsd\sigma d\mathcal{H}^{n-1}(y)\\
                & = \lim_{\eps \to 0} \frac{1}{2\eps}{\left(1-\frac{\eps}{L_0}\right)}^{1-n}\int_{-\eps}^\eps \int_{\Omega_\eps} K_\eps(x+s\nu(x),z)  \, dz ds\\
                & \leq \lim_{\eps \to 0} C_L {\left(1-\frac{\eps}{L_0}\right)}^{1-n} = C_L
        \end{split}
    \end{equation}

    Combining~\eqref{EQ::KepsEstimateCurvature},~\eqref{EQ::KepsEstimateUu}, and~\eqref{EQ::KastIsL1} we get
    \begin{align}\label{EQ::KepsEstimateLimit}
        \lim_{\eps \to 0} \bigg|\frac{1}{2\eps}&\int_{-\eps}^\eps \int_{\Omega_\eps} K_\eps(x+s\nu(x),z)u_\eps^+(z,t)\,dzds - \frac{1}{2\eps}\int_{\mathcal{S}} K^\ast(x,y)U_\eps^+(y,t)\,d\mathcal{H}^{n-1}(y)\bigg| = 0.
    \end{align}

    Moreover, the Lipschitz property of~$\psi$ given by~\eqref{EQ::psiAssumptionLipschitz}, together with~\eqref{EQ::AvgApproximationC0alphaBound},~\eqref{EQ::AvgApproximationC1alphaBound} and the Neumann condition, entail that
    \begin{align}
        \bigg|\psi&\big(x+s\nu(x),t,u_\eps(x+s\nu(x),t),\nabla u_\eps(x+s\nu(x),t)\big) - \psi\big(x,t,U_\eps(x,t),\nabla_T U_\eps(x,t)\big)\bigg|\\
            &\leq C_\psi \big(s^\alpha + \abs{u_\eps(x+s\nu(x),t)-U_\eps(x,t)} + \abs{\nabla u_\eps(x+s\nu(x),t) - \nabla_T U_\eps(x,t)}\big)\\
            &\leq C_\psi \big(s^\alpha + s^\alpha \abs{u_\eps}_{\alpha; A_\eps} + \abs{\partial_\nu u_\eps(x+s\nu(x), t)} + \abs{\nabla_T u_\eps(x+s\nu(x),t) - \nabla_T U_\eps(x,t)}\big)\\
            &\leq C_\psi \big(s^\alpha + s^\alpha \abs{u_\eps}_{\alpha; A_\eps} + {(\eps-\abs{s})}^\alpha\abs{u}_{1+\alpha; A_\eps} + s^\alpha \abs{u}_{1+\alpha; A_\eps}\big)\\
            &\leq C_\psi \big(1+3\constapp{CONST::HolderEmbeddingA}\const{CONST::UniformRegularity} \delta^{\lambda-2-\alpha} (1+G_T)\big) \eps^\alpha,
    \end{align}
    hence,
    \begin{equation}
        \begin{split}
            \lim_{\eps \to 0} \bigg|\frac{1}{2\eps}\int_{-\eps}^\eps &\psi\big(x+s\nu(x),t,u_\eps(x+s\nu(x),t),\nabla u_\eps(x+s\nu(x),t)\big)\,ds\\
                &- \psi\big(x,t,U_\eps(x,t),\nabla_T U_\eps(x,t)\big)\bigg| = 0.
        \end{split}
    \end{equation}
    This and~\eqref{EQ::KepsEstimateLimit} ensure that
    \begin{equation}\label{EQ::forceuULimit}
        \lim_{\eps \to 0} \bigg|\frac{1}{2\eps}\int_{-\eps}^\eps f_{u_\eps}^{(\eps)}(x+s\nu(x),t)\,ds - f_{U_\eps}^\ast(x,t)\bigg| = 0.
    \end{equation}

    Since~$u_\eps$ solves~\eqref{EQ::MainProblemEpsilon}, we have that
    \begin{equation}
        \partial_t u_\eps(x+s\nu(x),t) - \Delta u_\eps(x+s\nu(x),t) - f_{u_\eps}^{(\eps)}(x+s\nu(x),t) = 0,
    \end{equation}
    so that, with the contribution of~\eqref{EQ::partialTuULimit},~\eqref{EQ::LaplacianuULimit}, and~\eqref{EQ::forceuULimit}, we find
    \begin{align}
        \lim_{\eps \to 0} \bigg|&\partial_t U_\eps(x,t) - \Delta_{\mathcal{S}}U_\eps(x,t) - f^\ast_{U_\eps}(x,t)\bigg| \leq \lim_{\eps \to 0} \left(\bigg|\partial_t U_\eps(x,t) - \frac{1}{2\eps}\int_{-\eps}^\eps \partial_t u_\eps(x+s\nu(x),t)\,ds\bigg|\right.\\
            &\left.+ \bigg|\Delta_{\mathcal{S}}U_\eps(x,t) - \frac{1}{2\eps}\int_{-\eps}^\eps \Delta u_\eps(x + s\nu(x),t)\,ds\bigg|+ \bigg|f_{U_\eps}^\ast(x,t) - \frac{1}{2\eps}\int_{-\eps}^\eps f_{u_\eps}^{(\eps)}(x+s\nu(x),t)\,ds \bigg|\right) = 0
    \end{align}
    as desired.
\end{proof}

\begin{remark}\label{REM::NonlocalitySemilinearForLimit}
    We emphasize that the estimate in~\eqref{EQ::KepsEstimateUu} is obtained thanks to the fact that the nonlocal term in equation~\eqref{EQ::MainEquation} depends on the solution, but not on its derivatives. In fact, this is the only term for which we need a regularity estimate for the solution that holds up to the boundary.
    
    However,~$C^{2,\alpha}$ estimates up to the boundary are not available due to the mixed boundary condition. This is part of the reason why the regularity argument of Section~\ref{SEC::UniformRegularity} needs to be sharp, explaining why weighted Hölder spaces are a powerful technical tool when studying equations of the type that we consider in this paper.
\end{remark}

Finally, we end the section by completing the proof of our main result.

\begin{proof}[Proof of Theorem~\ref{THM::thm1}]
    Let~${\{\eps_k\}}_k \subset (0,L)$ be an infinitesimal sequence and~$T>0$. Then, thanks to Theorem~\ref{THM::epsIndependentRegularity}, the sequence~${\{u_{\eps_k}\}}_k$ is such that~$\abs{u}_{2+\alpha; \Omega_\eps}^{(-\lambda)}$ is bounded. Namely, from this and~\eqref{EQ::UaverageDefinition} we find that the sequence~${\{U_{\eps_k}\}}_k$ is uniformly Hölder continuous on~$\mathcal{S}^T$ and all of the sequences~${\{\nabla_T U_{\eps_k}\}}_k$, ${\{\nabla_T^2 U_{\eps_k}\}}_k$,~${\{\partial_t U_{\eps_k}\}}_k$ are uniformly Hölder continuous on every set of the form~$A \times (\delta,T]$, where~$A \Subset \mathcal{S}$ and~$\delta \in (0,T)$. 
    
    Hence, by the Ascoli-Arzelà Theorem there must exists a subsequence that converges to a limit~$U$. From Lemma~\eqref{LEM::UepsLimitEquation} it follows that~$U$ must satisfy the equation in~\eqref{EQ::MainProblemS}.
    
    The boundary condition is also satisfied because from Lemma~\ref{LEM::AvgApproximation} (applied with~$a \coloneq \lambda$) it follows that the uniform convergence of~$U_{\eps_k}$ and that of~$\restr{u_{\eps_k}}{\mathcal{S}}$ must coincide.
\end{proof}

\section{Conclusion}\label{SEC::Conclusion}

In this work, we have analyzed a nonlocal parabolic model describing bushfire propagation in a gully-shaped domain where combustible material is confined to a narrow region surrounded by insulating rocky hillsides. The ignition mechanism introduces a nonlocal interaction term, creating significant analytical challenges, particularly due to the presence of mixed boundary conditions and geometric degeneracy.

Our main contribution is the rigorous derivation and justification of a dimensional reduction in the asymptotic regime of a narrow gully. We show that the original two-dimensional problem (or, more generally,~$n$-dimensional problem) converges to a geometric evolution equation posed along the (possibly curved) axis of the gully (or, more generally, a hypersurface). The reduced model retains the essential ignition dynamics while encoding the geometry of the domain through its curvature and boundary structure, leading to the reduced equation
\[\partial_t U(x,t) = \Delta_{\mathcal{S}}U(x,t) + \int_{\mathcal{S}} K^\ast(x,y)U^+(y,t)\,d\mathcal{H}^{n-1}(y) + \psi\big(x,t,U(x,t),\nabla_T U(x,t)\big).\]

{F}rom a methodological standpoint, the analysis combines:
\begin{itemize}
\item The use of Fermi coordinates to handle the curvilinear geometry,
\item Parabolic estimates specifically adapted to the nonlocal ignition structure,
\item A bespoke reflection argument to obtain uniform bounds despite the degenerating domain and varying boundary conditions.\end{itemize}

The reflection technique plays a crucial role in overcoming the degeneracies of the parabolic estimates induced by the shrinking cross-section and ensures stability of the limit process.

We remark that the ignition term requires a special treatment due to the absence of classical Hölder estimates up to the boundary. Nevertheless, sharp intermediate estimates in distance-weighted spaces are sufficient to treat such term. See also Remarks~\ref{REM::NonlocalitySemilinearForLimit} and~\ref{REM::NonlocalitySemilinearForRegularity}.

Beyond the specific bushfire application, the approach developed here provides a framework for studying dimensional reduction in nonlocal parabolic problems posed in thin, geometrically complex domains. As such, it provides a useful, cross-disciplinary framework that can be exploited both in pure mathematics and in concrete applications.

Overall, this study illustrates how geometric analysis and tailored parabolic techniques can be combined to rigorously bridge multi-dimensional ignition models with effective lower-dimensional descriptions.

\begin{appendix}

\section{Weighted Hölder spaces}\label{APP::WeightedHolder}

In this appendix we state some well known results about weighted Hölder spaces, which we extensively use throughout the paper. Intuitively speaking, these spaces offer a quantitative way to measure how interior regularity of certain functions blows up near the boundary.

Our definitions and notation are mostly based on~\cite{MR826642} and~\cite{MR241822}*{Section~I.1}, however we also refer the interested reader to~\cite{MR244627}*{Section~3.2},~\cite{MR1814364}*{Section~4.3},~\cite{MR1465184}*{Section~4.1},~\cite{MR588031} and the references therein. All the results that we present here are an adaptation to our framework of results contained in these references.

We consider a domain, that is, a bounded, connected, open set~$X \subset \R^n$ and introduce the Hölder seminorm for~$a \in (0,1)$
\begin{equation}\label{EQ::EllipticSeminormsDefinition}
    {\left[u\right]}_{a; X} \coloneq \sup_{\substack{x,x^\prime \in X\\x\neq x^\prime}} \frac{\abs{u(x) - u(x^\prime)}}{\abs{x-x^\prime}^a}
\end{equation}

We define the ``elliptic'' Hölder norms as
\begin{equation}\label{EQ::EllipticNormsDefinition}
    \begin{split}
        &\abs{u}_{0; X} \coloneq \sup_{x \in X} \abs{u(x)},\\
        &\abs{u}_{a; X} \coloneq \abs{u}_{0; X} + {\left[u\right]}_{a; X},\\
        &\abs{u}_{1+a; X} \coloneq \abs{u}_{0; X} + \abs{\nabla u}_{a; X},\\
        &\abs{u}_{2+a; X} \coloneq \abs{u}_{0; X} + \abs{\nabla u}_{1+a; X}.
    \end{split}
\end{equation}

To define the ``parabolic'' counterpart, we define the parabolic norm as
\begin{equation}\label{EQ::ParabolicDistanceDefinition}
    \abs{(x,t)}_P \coloneq {\left(\abs{x}^2 + \abs{t}\right)}^{\frac{1}{2}},
\end{equation}
and, unless otherwise specified, we always consider subsets of~$\R^{n+1}$ to be endowed with the associated norm-induced metric.

Hence, we consider a time cylinder~$X^T \subset \R^{n+1}$ for some~$T>0$ and define the seminorms for~$a \in (0,1)$ as
\begin{equation}\label{EQ::ParabolicSeminormsDefinition}
    \begin{split}
        &{\left[u\right]}_{a; X^T} \coloneq \sup_{\substack{(x,t),(x^\prime, t^\prime) \in X^T\\(x,t)\neq(x^\prime, t^\prime)}} \frac{\abs{u(x,t) - u(x^\prime,t^\prime)}}{\abs{(x-x^\prime, t-t^\prime)}_P^a},\\
        &\left\langle u\right\rangle_{a; X^T} \coloneq \sup_{\substack{x\in X\\t,t^\prime \in (0,T)\\t\neq t^\prime}} \frac{\abs{u(x,t) - u(x,t^\prime)}}{\abs{t-t^\prime}^{\frac{1+a}{2}}}.
    \end{split}
\end{equation}

The norms are then given by
\begin{equation}\label{EQ::ParabolicNormsDefinition}
    \begin{split}
        &\abs{u}_{0; X^T} \coloneq \sup_{(x,t) \in X^T} \abs{u(x,t)},\\
        &\abs{u}_{a; X^T} \coloneq \abs{u}_{0; X^T} + {\left[u\right]}_{a; X^T},\\
        &\abs{u}_{1+a; X^T} \coloneq \abs{u}_{0; X^T} + \left\langle u\right\rangle_{a; X^T} + \abs{\nabla u}_{a; X^T},\\
        &\abs{u}_{2+a; X^T} \coloneq \abs{u}_{0; X^T} + \abs{\partial_t u}_{a; X^T} + \abs{\nabla u}_{1+a; X^T}.
    \end{split}
\end{equation}

We always implicitly assume without causing confusion that the norms and seminorms are to be intended as in~\eqref{EQ::EllipticSeminormsDefinition} and~\eqref{EQ::EllipticNormsDefinition} whenever dealing with a domain that is a subset of~$\R^n$, and in the sense of~\eqref{EQ::ParabolicSeminormsDefinition} and~\eqref{EQ::ParabolicNormsDefinition} for domains contained in~$\R^{n+1}$.

For a given~$\delta > 0$, we let
\begin{align}
    &I_\delta(X) \coloneq \left\{x \in X \colon \operatorname{dist}(x,\mathcal{B}) > \delta\right\},\\
    &I_\delta(X^T) \coloneq \left\{ (x,t) \in X^T \colon \operatorname{dist}(x,\mathcal{B}) > \delta,\, t > \delta^{\frac{1}{2}}\right\},
\end{align}
where~$\mathcal{B}$ is a (possibly improper) subset of~$\partial X$. Again, the convention in use is understood in dependence of the domain being purely spatial or spatio-temporal. Whenever~$X$ corresponds to~$\Omega_L$ as defined in~\eqref{EQ::OmegaLDefinition}, we always consider~$\mathcal{B}$ to correspond to~$\mathcal{D}_L$ unless otherwise specified. Notice that, with this definition,~$\abs{(x-x^\prime,t-t^\prime)}_P > \delta$ for every~$(x,t) \in I_\delta(\Omega_L^T)$ and~$(x^\prime,t^\prime) \in \mathcal{P}_L^{(T)}$.

From now on in this appendix,~$X$ will denote either a domain contained in~$\R^n$ or a time cylinder contained in~$\R^{n+1}$, i.e., we will not explicitly write the superindex~$T$. Then, the weighted norms are defined for~$a \in [0,3)$ and~$b \in [-a,+\infty)$ as 
\begin{equation}
    \abs{u}_{a; X}^{(b)} = \sup_{\delta > 0} \delta^{a+b} \abs{u}_{a; I_\delta(X)}.
\end{equation}

Finally, we define the spaces
\begin{equation}
    \mathcal{H}_a(X) \coloneq \left\{u \in C^{\lfloor a\rfloor}(X) \colon \abs{u}_{a;X} < +\infty\right\}
\end{equation}
and
\begin{equation}
    \mathcal{H}_a^{(b)}(X) \coloneq \left\{u \in C^{\lfloor a\rfloor}(X) \colon \abs{u}_{a;X}^{(b)} < +\infty\right\}.
\end{equation}

It is well known that~$\mathcal{H}_a(X)$ and~$\mathcal{H}_a^{(b)}(X)$, each endowed with the respective norm, are Banach spaces. We remark that the space~$\mathcal{H}_a^{(-a)}(X)$ coincides with the space~$\mathcal{H}_a(X)$, and that, for~$a \in [0,1)$, the space~$\mathcal{H}_a(X)$ is exactly the space~$C^{0,a}(\overline{X})$ in the corresponding metric (with the convention that~$C^{0,0}(\overline{X}) = C(\overline{X})$). They also enjoy some useful inclusion/monotonicity properties, as showcased in the next two results (see also~\cite{MR588031}*{Lemma~2.1}). 

\begin{proposition}\label{PROP::HolderEmbeddingA}
    Let~$X$ be either a domain or time cylinder. Let~$a \in (0,3)$,~$a^\prime \in [0,a)$, and~$b \in [-a^{\prime},+\infty)$. 
    
    Then, there exists a constant~$\constapp{CONST::HolderEmbeddingA}>0$, which depends only on~$X$,~$a$ and~$a^\prime$, such that, for every~$u \in \mathcal{H}_a^{(b)}(X)$,
    \begin{equation}\label{EQ::HolderEmbeddingA}
        \abs{u}_{a^\prime; X}^{(b)} \leq \constapp{CONST::HolderEmbeddingA} \abs{u}_{a; X}^{(b)}.
    \end{equation}
\end{proposition}

\begin{proposition}\label{PROP::HolderEmbeddingB}
    Let~$X$ be a domain or time cylinder. Let~$a \in (0,3)$ and~$b^\prime \geq b\in [-a,+\infty)$. 
    
    Then, there exists a constant~$\constapp{CONST::HolderEmbeddingB}>0$, which depends only on~$b$,~$b^\prime$, and~$\operatorname{diam}(X)$, such that, for every~$u \in \mathcal{H}_a^{(b)}(X)$,
    \begin{equation}\label{EQ::HolderEmbeddingB}
        \abs{u}_{a; X}^{(b^\prime)} \leq \constapp{CONST::HolderEmbeddingB} \abs{u}_{a; X}^{(b)}.
    \end{equation}
\end{proposition}

\begin{remark}
    Although~$\constapp{CONST::HolderEmbeddingA}$ and~$\constapp{CONST::HolderEmbeddingB}$ depend on~$X$, these constants are controlled by geometric quantities that remain uniformly bounded for families of tubular neighborhoods. That is, for every~$0<\eps<L<L_0$, setting either~$X \coloneq \Omega_\eps$ or~$X \coloneq \Omega_\eps^T$, both~\eqref{EQ::HolderEmbeddingA} and~\eqref{EQ::HolderEmbeddingB} hold with uniform constants~$\constapp{CONST::HolderEmbeddingA}$ and~$\constapp{CONST::HolderEmbeddingB}$ that may depend on~$L$ but not on~$\eps$.
\end{remark}

The spaces of type~$\mathcal{H}_a^{(b)}$ also enjoy an interpolation inequality, in analogy with classical Hölder continuous spaces.

\begin{proposition}\label{PROP::HolderInterpolation}
    Let~$X$ be either a domain or time cylinder. Let~$a,a^\prime \in [0,3)$,~$b \in [-a,+\infty)$,~$b^\prime \in [-a^\prime,+\infty)$, and~$\vartheta \in (0,1)$. Let~$a^\ast \coloneq \vartheta a + (1-\vartheta)a^\prime$ and~$b^\ast \coloneq \vartheta b + (1-\vartheta)b^\prime$. 
    
    Then, there exists a constant~$\constapp{CONST::HolderInterpolation} > 0$, which depends only on~$n$,~$a$,~$a^\prime$, and~$\vartheta$, such that, for every~$u \in \mathcal{H}_a^{(b)}(X) \cap \mathcal{H}_{a^\prime}^{(b^\prime)}(X)$,
    \begin{equation}
        \abs{u}_{a^\ast; X}^{(b^\ast)} \leq \constapp{CONST::HolderInterpolation} {\left(\abs{u}_{a; X}^{(b)}\right)}^\vartheta {\left(\abs{u}_{a^\prime; X}^{(b^\prime)}\right)}^{1-\vartheta}.
    \end{equation}
\end{proposition}

The product of classical Hölder continuous functions is Hölder continuous. The same holds for weighted Hölder continuity (see also~\cite{MR588031}*{Lemma~2.2}).

\begin{proposition}\label{PROP::HolderProductRule}
    Let~$X$ be either a domain or time cylinder. Let~$a \in [0,3)$,~$b \in [-a,+\infty)$, and~$c_1,c_2 \in  [0,a+b]$. 
    
    Then, there exists a constant~$\constapp{CONST::HolderProductRule} > 0$, which depends only on~$a$, such that, for every~$u \in \mathcal{H}_a^{(b-c_1)}(X)\cap\mathcal{H}_0^{(c_2)}(X)$ and~$v \in \mathcal{H}_a^{(b-c_2)}(X)\cap\mathcal{H}_0^{(c_1)}(X)$,
    \begin{equation}
        \abs{uv}_{a; X}^{(b)} \leq \constapp{CONST::HolderProductRule} \left(\abs{u}_{a; X}^{(b-c_1)}\abs{v}_{0; X}^{(c_1)} + \abs{u}_{0; X}^{(c_2)}\abs{v}_{a; X}^{(b-c_2)}\right) .
    \end{equation}
\end{proposition}

We also state an important compactness property (see also~\cite{MR588031}*{Lemma~4.2}).
\begin{proposition}\label{PROP::HolderCompactEmbedding}
    Let~$X$ be either a domain or time cylinder. Let~$a \in (0,3)$,~$b \in [-a,0)$,~$a^\prime \in [0,a)$ and~$b^\prime \in [-a,0)$. Furthermore, assume that~$b^\prime > b$.

    Then, every bounded sequence~${\{u_k\}}_k \subset \mathcal{H}_a^{(b)}(X)$ admits a subsequence that is strongly convergent in~$\mathcal{H}_{a^\prime}^{(b^\prime)}(X)$.
\end{proposition}

\section{Fermi coordinates for hypersurfaces}\label{APP::FermiCoordinates}
Here we collect some notions from differential geometry about tubular neighborhoods of Riemannian manifolds and recall some results about Fermi coordinates, with the purpose of giving a brief review and fixing notation and conventions.

In his work~\cite{1922RendL..31...21F}, Fermi introduced a particular coordinate choice to express points ``near'' a geodetic line in a Riemannian manifold. Later, Gray~\cite{MR827265} proposed a more general framework where the same idea can be used to parametrize a neighborhood of any Riemannian submanifold of a given ambient Riemannian manifold. Our discussion focuses on a particular case, dealing with hypersurfaces in~$\R^n$, for a more general treatment of the subject see~\cite{MR2024928}.

Throughout this section, we let~$T_x\mathcal{S}$ denote the space of vectors tangent to~$\mathcal{S}$ at the point~$x \in \mathcal{S}$. We shall always assume~$\mathcal{S}$ is~$C^{k,\alpha}$ for some~$k \in \{2,3,\dots\}$ and~$\alpha \in (0,1)$ unless specified otherwise, and denote with~${\{\kappa_j(x)\}}_{j=1}^{n-1}$ the principal curvatures at each point~$x \in \mathcal{S}$. Also let~$L_0$ be as in~\eqref{EQ::L0Definition} and assume~$L \in (0,L_0)$.

In differential geometry, Fermi coordinates generalize normal coordinates in a neighborhood of a point. Specifically, they offer us a way to parametrize the set~$\Omega_L$ defined in~\eqref{EQ::OmegaLDefinition}. In order to do so, we introduce the map~$\Phi \colon \mathcal{S} \times (-L, L) \to \Omega_L$, defined as
\begin{equation}\label{EQ::PhiDefinition}
    \Phi(x, s) = x + s \nu(x).
\end{equation}
If~$L$ is chosen in the interval~$(0,L_0)$,~$\Phi$ is a~$C^{2,\alpha}$ diffeomorphism. That is,~$\Phi$ is an invertible map and both~$\Phi$ and~$\Phi^{-1}$ have~$C^{2,\alpha}$ regularity. Thus,~$\Omega_L$ is a uniform tubular neighborhood of~$\mathcal{S}$, which allows us to use Fermi coordinates. First, let us state the assumptions with the following result (see also~\cite{MR3887684}*{Theorem~5.25} for a more abstract version).

\begin{proposition}\label{PROP::TubularNeighborhood}
    The map~$\Phi$ defined in~\eqref{EQ::PhiDefinition} is a~$C^{k-1,\alpha}$ diffeomorphism.
\end{proposition}

\begin{proof}
    Since~$\mathcal{S}$ is at least~$C^2$,~$\Phi$ is differentiable everywhere with differential
    \begin{equation}\label{EQ::PhiDifferential}
        d\Phi(x,s)[\xi, \sigma] = \xi + s d\nu(x)[\xi] + \sigma\nu(x)
    \end{equation}
    for every~$x \in \mathcal{S}$,~$s \in (-L, L)$,~$\xi \in T_x\mathcal{S}$, and~$\sigma \in \R$.

    We let~${\{e_j\}}_{j=1}^{n-1}$ be the orthonormal frame for~$T_x\mathcal{S}$ consisting of the principal directions of curvature, with each~$e_j$ being associated to the corresponding~$\kappa_j$. Thanks to the Weingarten equation, we have
    \begin{equation}\label{EQ::ShapeOperator}
        d\nu(x)[\xi] = -A_x(\xi) = -\sum_{j=1}^{n-1} \kappa_j \xi_j e_j,
    \end{equation}
    where~$A_x \colon T_x\mathcal{S} \to T_x\mathcal{S}$ is the shape operator of~$\mathcal{S}$ at the point~$x$.

    Plugging this information into~\eqref{EQ::PhiDifferential} yields
    \begin{equation}\label{EQ::PhiDifferentialCoordinates}
        d\Phi(x,s)[\xi, \sigma] = \sum_{j=1}^{n-1} \bigg((1-s\kappa_j(x))\xi_j e_j\bigg) + \sigma \nu(x).
    \end{equation}
    Thanks to~\eqref{EQ::L0Definition}, we have
    \begin{equation}\label{EQ::skappajSmall}
        \abs{s\kappa_j(x)} \leq \abs{s}\abs{\kappa_j(x)} \leq \frac{L}{L_0} < 1,
    \end{equation}
    ensuring that the factors~$1-s\kappa_j(x)$ in~\eqref{EQ::PhiDifferentialCoordinates} never vanish. This entails that~$d\Phi(x,s)$ is always invertible, allowing us to apply the inverse function theorem to conclude that~$\Phi$ is a local diffeomorphism onto its image.

    We now claim that~$\Phi$ is injective, which would make it a global diffeomorphism. Suppose by contradiction it is not the case. Then, there exist~$(x,s)$ and~$(x^\prime,s^\prime)$ in the set~$\mathcal{S} \times (-L,L)$ such that~$X \coloneq x+s\nu(x) = x^\prime + s^\prime\nu(x^\prime)$. Without loss of generality, we suppose that~$s > s \prime$. Then, the set~$\overline{B}_s(X)$ intersects~$\mathcal{S}$ at least in two points, namely~$x$ and~$x^\prime$. However, since
    \[\abs{s} < L < L_0 \leq \inf_{1\leq j\leq n-1}\frac{1}{\abs{\kappa_j(x)}},\]
    it follows that~$\overline{B}_s(X)$ is completely contained in the closure of the osculating sphere of~$\mathcal{S}$ at the point~$x$, yielding the desired contradiction.

    It remains to prove that~$\Phi$ and~$\Phi^{-1}$ are both~$C^{k-1,\alpha}$. The regularity of~$\Phi$ simply follows from the~$C^{k-1,\alpha}$ regularity of~$\nu$, which holds because~$\mathcal{S}$ is~$C^{k,\alpha}$.
    
    On the other hand, if we consider~$d$ to be the signed distance function from~$\mathcal{S}$, we have that~$d$ is~$C^{k,\alpha}$ in~$\Omega_L$ (see e.g.~\cite{MR1814364}*{Lemma~14.16}). Furthermore, for every~$(x,s)\in\mathcal{S} \times (-L,L)$,
    \begin{equation}
        d(x+s\nu(x)) = s\qquad\text{and}\qquad\nabla d(x+s\nu(x)) = \nu(x),
    \end{equation}
    therefore, for every~$X \in \Omega_L$,
    \begin{equation}
        \Phi^{-1}(X) = (X-d(X)\nabla d(X), d(X)).
    \end{equation}
    From this it follows that~$\Phi^{-1}$ has the same regularity as~$\nabla d$, that is,~$C^{k-1,\alpha}$, concluding the proof.
\end{proof}

We now give the formal definition of Fermi coordinates in our setting.

\begin{definition}
    Let~$(x_1,\dots,x_{n-1})$ be a coordinate system on an open subset~$U\subset\mathcal{S}$. We say that the Fermi coordinates relative to~$(x_1,\dots,x_{n-1})$ are~$(p_1,\dots,p_n)$, defined as
    \begin{align}\label{EQ::FermiCoordinatesDefinition}
        \begin{cases}
            p_j(\Phi(x,s)) = x_j,&1\leq j\leq n-1,\\
            p_n(\Phi(x,s)) = s,&
        \end{cases}
    \end{align}
    for~$(x,s) \in A \times (-L, L)$.
\end{definition}

In the following, we collect several identities that are useful to carry out computations with Fermi coordinates. First of all, we express the metric tensor in Fermi coordinates in terms of local coordinates on~$\mathcal{S}$ (see also~\cite{MR3887684}*{Proposition~5.26}).

\begin{proposition}
    Let~$g_{ij}$ be the components of the Riemannian metric of~$\mathcal{S}$ with respect to a coordinate system~$(x_1,\dots,x_{n-1})$ on an open set~$U \subset \mathcal{S}$. Moreover, let~$(p_1,\dots,p_n)$ be the Fermi coordinates relative to~$(x_1,\dots,x_{n-1})$.
    
    Then, for each~$x \in U$,~$s \in (-L,L)$, and~$X = \Phi(x,s)$, the components~$\tilde{g}_{ij}$ of the Riemannian metric relative to Fermi coordinates are given by:
    \begin{align}\label{EQ::FermiCoordinatesMetric}
        \begin{cases}
            \tilde{g}_{ij}(X) = \left(\delta_{ik} - s h_{ik}(x)\right) \left(\delta_{jl} - s h_{jl}(x)\right) g_{kl}(x),&1\leq i,j \leq n-1,\\
            \tilde{g}_{nj}(X) = \tilde{g}_{jn}(X) = 0,&1\leq j \leq n-1,\\
            \tilde{g}_{nn}(X) = 1,
        \end{cases}
    \end{align}
    where~$h_{ij}$ denotes the coefficients of the shape operator of~$\mathcal{S}$.
\end{proposition}

\begin{proof}
    We have:
    \begin{equation}\label{EQ::gTildeExpansionGeneral}
        \tilde{g}_{ij}(X) = \dotprod{\partial_i X}{\partial_j X}.
    \end{equation}
    We consider the local coordinate frame~${\{e_j\}}_{j=1}^{n-1}$ for~$\mathcal{S}$ given by~$e_j = \partial_j x$. Then, using~\eqref{EQ::PhiDifferential}
    \begin{equation}\label{EQ::PhiChainRuleX}
        \partial_j X = d\Phi(x,s)[e_j, 0] = e_j + s d\nu(x)[e_j],
    \end{equation}
    for every~$1\leq j\leq n-1$. Thanks to the Weingarten equation we also get:
    \begin{equation}\label{EQ::WeingartenEquation}
        d\nu(x)[e_j] = -A_x(e_j) = -h_{jl} e_l.
    \end{equation}
    Plugging this and~\eqref{EQ::PhiChainRuleX} into~\eqref{EQ::gTildeExpansionGeneral}, and also recalling that~$e_j = \delta_{jl}e_l$, we find:
    \begin{align}
        \tilde{g}_{ij}(X)   &= \dotprod{(\delta_{ik}-sh_{i}^{k}(x))e_k}{(\delta_{jl}-sh_{jl}(x))e_l}\\
                            &= (\delta_{ik}-sh_{ik}(x))(\delta_{jl}-sh_{jl}(x)) \dotprod{e_k}{e_l}\\
                            &= (\delta_{ik}-sh_{ik}(x))(\delta_{jl}-sh_{jl}(x))g_{kl}(x),
    \end{align}
    which proves the first line of~\eqref{EQ::FermiCoordinatesMetric}.

    Moreover,~\eqref{EQ::PhiChainRuleX} and~\eqref{EQ::WeingartenEquation} entail that~$\partial_j X \in T_x\mathcal{S}$ for~$1\leq j\leq n-1$. Hence, considering that
    \begin{equation}
        \partial_n X = d\Phi(x,s)[0, 1] = \nu(x),
    \end{equation}
    the last two lines of~\eqref{EQ::FermiCoordinatesMetric} immediately follow.
\end{proof}

We also recall some useful properties of the Christoffel symbols in Fermi coordinates. Once again, the best we can do without explicit assumption on~$\mathcal{S}$ is expressing them in terms of quantities related only to~$\mathcal{S}$ and~$s$.

\begin{proposition}\label{PROP::FermiCoordinatesChristoffelSymbols}
    Let~$(x_1,\dots,x_{n-1})$ be a coordinate system on an open set~$A \subset \mathcal{S}$. Let~$(p_1,\dots,p_n)$ be the Fermi coordinates relative to~$(x_1,\dots,x_{n-1})$. For each~$x \in U$,~$s \in (-L,L)$, and~$X = \Phi(x,s)$, denote with~$\Gamma_{ij}^k(x,s)$ the Christoffel symbols on~$\mathcal{S}(s)$ at the point~$X$.
    
    Then, the Christoffel symbols of~$\tilde{g}$ in~$\Omega_L$, denoted as~$\tilde{\Gamma}_{ij}^k(X)$ are given by:
    \begin{align}
        \begin{cases}\label{EQ::FermiCoordinatesChristoffelSymbols}
            \tilde{\Gamma}_{ij}^k(X) = \Gamma_{ij}^{k}(x,s),&1\leq i,j,k \leq n-1,\\
            \tilde{\Gamma}_{ij}^n(X) = h^{(s)}_{ij}(X),&1\leq i,j \leq n-1,\\
            \tilde{\Gamma}_{nj}^k(X) = \tilde{\Gamma}_{jn}^k(X) = 0,&1\leq j,k \leq n,
        \end{cases}
    \end{align}
    where~$h^{(s)}_{ij}$ denotes the components of the second fundamental form of~$\mathcal{S}(s)$.
\end{proposition}

\begin{proof}
    The case where~$1\leq i,j,k \leq n-1$ (i.e.~first line in~\eqref{EQ::FermiCoordinatesChristoffelSymbols}) is trivial due to~$\mathcal{S}(s)$ inheriting the metric~$\tilde{g}_{ij}$ from~$\Omega_L$.

    The case~$k = n$, that is, the second line in~\eqref{EQ::FermiCoordinatesChristoffelSymbols}, follows directly from the definition of the second fundamental form and Christoffel symbols, upon noticing that the~$n$-th basis vector in Fermi coordinates coincides with~$\nu$, namely
    \begin{equation}
        \tilde{\Gamma}_{ij}^n(X) = \dotprod{D_{e_i(X)}e_j(X)}{e_n(X)} = \dotprod{D^{(s)}_{e_i(X)}e_j(X)}{\nu(x)} = h^{(s)}_{ij}(x),
    \end{equation}
    where~$D$ is the Levi-Civita connection in~$\Omega_L$ and~$D^{(s)}$ is the Levi-Civita connection on~$\mathcal{S}(s)$.

    When either~$i$ or~$j$ are equal to~$n$,~\eqref{EQ::FermiCoordinatesMetric} shows that~$\tilde{g}_{ij}(X)$ is constant, therefore, denoting with~$\tilde{g}^{kl}(\cdot)$ the inverse of~$\tilde{g}_{kl}(\cdot)$,
    \begin{equation}
        \tilde{\Gamma}_{ij}^k(X) = \frac{1}{2} \tilde{g}^{kl}(X) \left( \partial_i \tilde{g}_{jl} + \partial_j \tilde{g}_{il} - \partial_l \tilde{g}_{ij} \right)(X) = 0,
    \end{equation}
    proving the last line in~\eqref{EQ::FermiCoordinatesChristoffelSymbols}.
\end{proof}

We also recall the Jacobian determinant for changing variables in integrals.

\begin{proposition}\label{PROP::FermiCoordinatesJacobian}
    Let~$E \subseteq \Omega_L$ be a measurable set, and~$f \in L^1(E)$. Then
    \begin{equation}\label{EQ::FermiCoordinatesJacobian}
        \int_{E} f(X) \,dX = \int_{\Phi^{-1}(E)} f(\Phi(x,s)) \prod_{j=1}^{n-1}(1 - s\kappa_j(x))\,d\mathcal{H}^{n-1}(x) ds.
    \end{equation}
\end{proposition}

\begin{proof}
    The Jacobian determinant of the variable change in~\eqref{EQ::FermiCoordinatesJacobian} is
    \begin{equation}\label{EQ::JacobianAsMetricTensor}
        \determinant J(x,s) = \sqrt{\determinant \tilde{g}(x + s\nu(x))}\text{,}
    \end{equation}
    with~$\tilde{g}$ as in~\eqref{EQ::FermiCoordinatesMetric}.

    By choosing an orthonormal frame for~$\mathcal{S}$ consisting of the principal directions of curvature, we diagonalize both the components of~$g$ and~$h$, obtaining that
    \begin{equation}
        g_{ij}(x) = \delta_{ij}\qquad \text{and}\qquad h_{ij}(x) = \left\{\begin{aligned}
            &\kappa_j(x),\qquad \text{if }i=j,\\
            &0,\qquad\text{otherwise.}
        \end{aligned}\right.
    \end{equation}
    
    Hence, also using~\eqref{EQ::FermiCoordinatesMetric}, for~$1\leq i,j \leq n-1$ we have
    \begin{align}
        \tilde{g}_{ij}(x + s \nu(x)) &= \left(\delta_{ik} - s h_{ik}(x)\right) \left(\delta_{jl} - s h_{jl}(x)\right) g_{kl}(x)\\
                                     &= (1 - s\kappa_i(x))\delta_{ik} (1 - s\kappa_j(x))\delta_{jl} \delta_{kl}\\
                                     &= {(1 - s\kappa_j(x))}^2 \delta_{ij}\text{.}
    \end{align}
    Thus,
    \begin{equation}
        \determinant \tilde{g}(x + s\nu(x)) = \prod_{j=1}^{n-1}{(1-s\kappa_j(x))}^2\text{,}
    \end{equation}
    which, together with~\eqref{EQ::skappajSmall} and~\eqref{EQ::JacobianAsMetricTensor}, concludes the proof.
\end{proof}

Next, we present an identity relating the Laplacian (in Euclidean coordinates) of a function in~$\Omega_L$ to the Laplace-Beltrami operator applied to such function on each~$\mathcal{S}(s)$. See also~\cite{MR2032110} for a proof of this fact in a more general setting.

\begin{proposition}\label{PROP::PacardLaplacianDecomposition}
    Assume that~$\mathcal{S}$ is~$C^3$. Denote with~$H_s$ the mean curvature of~$\mathcal{S}(s)$ (defined here as the sum of the principal curvatures). Let~$U \subset \Omega_L$ be an open set and~$f \in C^2(U)$. Then, for every~$X \in U$ and~$(x,s) = \Phi^{-1}(X)$:
    \begin{equation}
        \Delta f(X) = \partial^2_{\nu(x)}f(X) + \Delta_{\mathcal{S}(s)}f(X) - H_s(X) \partial_{\nu(x)}f(X)\text{,}
    \end{equation}
    where~$\Delta_{\mathcal{S}(s)}$ denotes the Laplace-Beltrami operator on~$\mathcal{S}(s)$.
\end{proposition}

\begin{proof}
    Let~$X \in \Omega_L$ and~$(p_1,\dots,p_n)$ be a set of Fermi coordinates in a neighborhood of~$X$. Recalling the connection characterization of the Laplace-Beltrami operator, we apply it to the classical Laplacian, which coincides to the Laplace-Beltrami operator relative to Euclidean coordinates, in~$\Omega_L$ to recover
    \begin{equation}\label{EQ::LaplaceBeltramiAllIndices}
        \Delta f(X) = \tilde{g}^{ij}(X) \left(\partial_i \partial_j f - \tilde{\Gamma}_{ij}^k \partial_k f\right)(X),
    \end{equation}
    where~$\tilde{g}^{ij}$ denotes the components of the inverse of~$\tilde{g}_{ij}$, with the latter as in~\eqref{EQ::FermiCoordinatesMetric}, and~$\tilde{\Gamma}{ij}^k$ are as in~\eqref{EQ::FermiCoordinatesChristoffelSymbols}.
    
    The expression of the Christoffel symbols with indices~$1\leq i,j,k \leq n-1$ entails that
    \begin{equation}\label{EQ::LaplaceBeltramiIndicesSigmaS}
        \sum_{i,j = 1}^{n-1}\left(\tilde{g}^{ij}\partial_i \partial_j f\right)(X) - \sum_{i,j,k = 1}^{n-1}\left(\tilde{g}^{ij}\tilde{\Gamma}_{ij}^k \partial_k f\right)(X) = \Delta_{\mathcal{S}(s)}f(X).
    \end{equation}

    Due to the well known fact that~$\tilde{g}^{ij} h_{ij}^{(s)} = H_s$, we obtain that
    \begin{equation}\label{EQ::LaplaceBeltramiIndicesCurvatureTerm}
        - \sum_{i,j = 1}^{n-1}\left(\tilde{g}^{ij}\tilde{\Gamma}_{ij}^n \partial_n f\right)(X) = - H_s(X) \partial_n f(X) = - H_s(X) \partial_{\nu(x)} f(X). 
    \end{equation}

    Moreover,~$\tilde{g}_{ij}$ is block diagonal. Hence,~$\tilde{g}^{ij}$ is also block diagonal and we have that~$\tilde{g}^{nj}(X) = \tilde{g}^{jn}(X) = 0$ for every~$1\leq j\leq n-1$. Then
    \begin{equation}\label{EQ::LaplaceBeltramiIndicesNullChristoffel}
        \sum_{j = 1}^{n-1}\left(\tilde{g}^{nj}\partial_n \partial_j f\right)(X) - \sum_{j,k = 1}^{n-1}\left(\tilde{g}^{nj}\tilde{\Gamma}_{nj}^k \partial_k f\right)(X) = 0,
    \end{equation}
    and the same holds for the symmetric case~$1\leq i\leq n-1$,~$j=n$.

    Finally,~$\tilde{\Gamma}_{nn}^n = 0$, and~$\tilde{g}^{nn} = 1$ (again, due to block diagonality), thus
    \begin{equation}\label{EQ::LaplaceBeltramiIndicesNNN}
        \tilde{g}^{nn}(X) \left(\partial_n \partial_n f - \tilde{\Gamma}_{nn}^n \partial_n f\right)(X) = \partial^2_{n}f(X) = \partial^2_{\nu(x)}f(X).
    \end{equation}

    The right hand side of~\eqref{EQ::LaplaceBeltramiAllIndices} is exactly the sum of~\eqref{EQ::LaplaceBeltramiIndicesSigmaS},~\eqref{EQ::LaplaceBeltramiIndicesCurvatureTerm},~\eqref{EQ::LaplaceBeltramiIndicesNullChristoffel}, and~\eqref{EQ::LaplaceBeltramiIndicesNNN}, which ends the proof.
\end{proof}
\begin{remark}
    The additional assumption that~$\mathcal{S}$ is~$C^3$ is fundamental in order to define the operator~$\Delta_{\mathcal{S}(s)}$ for every~$s$ and at every point. In fact, each hypersurface~$\mathcal{S}(s)$ loses a degree of regularity due to its metric containing a factor depending on the shape operator coefficients~$h_{ij}$.
    
    In principle, it is possible to ask for~$\mathcal{S}$ to be only~$C^{2,1}$, defining the operator in a weak sense, with a pointwise meaning almost everywhere. Nevertheless, we ask for~$C^3$ regularity to find pointwise estimates everywhere. 
\end{remark}

We now state a special case of the Generalized Gauss Lemma. The original version of this result appears in~\cite{MR182537}*{paragraph~15}, where Gauss dealt with circular neighborhoods of a point in a Riemannian manifold. The most general version (see~\cite{MR2024928}*{Lemma~2.11}) deals with generic tubular neighborhoods of Riemannian submanifolds embedded in Riemannian manifolds. We remark that the version we give is a weaker result and a direct consequence of the one found in~\cite{MR2024928}. Nevertheless, we also provide two different proofs of the fact, highlighting some intuitive geometrical aspects in our specific case.

\begin{lemma}\label{LEM::GaussLemma}
    Let~$s \in (-L, L)$ and~$x^\prime \in \mathcal{S}(s)$ with~$x^\prime = x + s \nu(x)$ for some~$x \in \mathcal{S}$. Then, denoting with~$\nu^\ast$ the unit normal to~$\mathcal{S}(s)$ at the point~$x^\prime$,~$\nu^\ast$ is parallel to~$\nu(x)$.
\end{lemma}

\begin{proof}
    We prove that~$T_x\mathcal{S} = T_{x^\prime}\mathcal{S}(s)$, from which the result will plainly follow. Recalling the definition of the map~$\Phi$ in~\eqref{EQ::PhiDefinition}, it is straightforward that~$\Psi \coloneq \Phi(\cdot, s)$ acts as a diffeomorphism between~$\mathcal{S}$ and~$\mathcal{S}(s)$. Therefore,~$T_{x^\prime}\mathcal{S}(s)$ is spanned by~${\{d\Psi(x)[e_j]\}}_{j=1}^{n-1}$ for any basis~${\{e_j\}}_{j=1}^{n-1}$ of~$T_x\mathcal{S}$. By choosing~$\sigma = 0$ in~\eqref{EQ::PhiDifferentialCoordinates} we immediately see that~$d\Psi(x)[\xi] \in T_x\mathcal{S}$ for every~$\xi \in T_x\mathcal{S}$, which completes the proof.
\end{proof}

\begin{proof}[\protect{Alternative proof}]
    Let~$y \coloneq x^\prime-x$ and consider the set~$B \coloneq B_{\frac{s}{2}}(y)$. Since~$x^\prime \in \bar{B} \cap \mathcal{S}(s)$, if we show it is the unique point in this set we also conclude that~$B$ is tangent to~$\mathcal{S}(s)$, hence proving that~$x^\prime-y = \frac{s}{2} \nu(x)$ is orthogonal to~$T_{x^\prime}\mathcal{S}(s)$, which would conclude the proof.

    Let us then suppose by contradiction that there exists a point~$z \in \bar{B} \cap \mathcal{S}(s)$, with~$z \neq x^\prime$. Thus, by the triangular inequality,~$\abs{z-x} < s$, where strict inequality holds due to~$z$ being different from~$x^\prime$. But then, since~$\mathcal{S}$ is~$C^2$,~$z$ has a unique orthogonal projection~$z^\prime$ on~$\mathcal{S}$, with
    \begin{equation}\label{EQ::zzPrime}
        z = z^\prime + s \nu(z^\prime).
    \end{equation}

    However, we also know that~$z$ is the minimizer in~$\mathcal{S}$ of the Euclidean distance from~$z$, hence
    \begin{equation}\label{EQ::zDistanceBound}
        \abs{z - z^\prime} = \dist(z, \mathcal{S}) \leq \abs{z-x} < s\text{.}
    \end{equation}

    Since~\eqref{EQ::zzPrime} and~\eqref{EQ::zDistanceBound} are incompatible, we have found the desired contradiction.
\end{proof}

We conclude the section proving an ancillary geometric estimate, which can be interpreted as a weaker analogue of convexity for~$\Omega_L$.
\begin{lemma}\label{LEM::Quasiconvexity}
    There exists a constant~$\constapp{CONST::Quasiconvexity} >0$, which depends only on~$n$,~$\mathcal{S}$ and~$L$, such that, for every~$x,x^\prime \in \Omega_L$, there exists a~$C^1$ curve~$\gamma \colon [0,1] \to \Omega_L$ such that
    \begin{equation}\label{EQ::QuasiconvexityInequality}
        \gamma(0)=x,\qquad\gamma(1)=x^\prime,\qquad\abs{x-x^\prime}\leq\operatorname{length}(\gamma)\leq \constapp{CONST::Quasiconvexity} \abs{x-x^\prime}.
    \end{equation}
\end{lemma}
\begin{proof}
    We cover~$\mathcal{S}$ by finitely many coordinate charts~${\{(U_j, \varphi_j)\}}_j$. We assume without loss of generality that for every~$j$, either~$U_j= B_1(0) \subset \R^{n-1}$ if~$(U_j, \varphi_j)$ is an interior chart, or~$U_j = B_1^+(0) \subset \R^{n-1}$ if~$(U_j, \varphi_j)$ is a boundary chart. Therefore, via Fermi coordinates we can cover~$\Omega_L$ with sets of type~$W_j \coloneq \Phi(\varphi_j(U_j),(-L,L))$. Importantly, each~$W_j$ is diffeomorphic to~$U_j \times (-L,L)$ with a uniform constant~$C > 0$, which depends only on~$n$,~$\mathcal{S}$ and~$L$. We denote such diffeomorphism by~$f_j$.

    Thanks to Lebesgue's Number Lemma (see e.g.~\cite{MR3728284}*{Lemma~27.5}), we find~$\mu > 0$ such that, for every~$x,x^\prime \in \Omega_L$, either
    \begin{equation}\label{EQ::xxprimeSameWj}
        x,x^\prime\in W_j\text{ for some }j,
    \end{equation}
    or
    \begin{equation}\label{EQ::xxprimeFar}
        \abs{x-x^\prime} > \mu.
    \end{equation}

    If~\eqref{EQ::xxprimeSameWj} is the case, we find~$y,y^\prime \in U_j$ and~$s,s^\prime \in (-L,L)$ such that
    \[x=\Phi(\varphi_j(y),s)\qquad \text{and}\qquad x^\prime = \Phi(\varphi_j(y^\prime),s^\prime).\]
    We let
    \begin{equation}
        \gamma \colon [0,1] \ni \sigma \longmapsto \Phi(\varphi_j(\sigma y + (1-\sigma)y^\prime), \sigma s + (1-\sigma)s^\prime) \in W_j.
    \end{equation}
    Because~$U_j$ is strictly convex,~$\gamma([0,1])$ is contained in the interior of~$W_j$. Also, the diffeomorphism~$f_j$ entails
    \begin{equation}\label{EQ::gammaLengthNear}
        \operatorname{length}(\gamma) \leq C{\left(\abs{y-y^\prime}^2 + \abs{s-s^\prime}^2\right)}^{\frac{1}{2}} \leq C^2 \abs{x-x^\prime}.
    \end{equation}

    If~\eqref{EQ::xxprimeFar} holds true, thanks to the compact Riemannian metric structure of~$\Omega_L$ we can find a constant~$D>0$ such that every couple of points contained in~$\Omega_L$ are joined by a continuous curve of length less than or equal to~$D$. We find such a curve~$\gamma$ connecting~$x$ and~$x^\prime$, and notice that
    \begin{equation}\label{EQ::gammaLengthFar}
        \operatorname{length}(\gamma) \leq D \leq \frac{D}{\mu}\abs{x-x^\prime}.
    \end{equation}

    All in all,~\eqref{EQ::gammaLengthNear} and~\eqref{EQ::gammaLengthFar} prove that~\eqref{EQ::QuasiconvexityInequality} holds with
    \[\constapp{CONST::Quasiconvexity} \coloneq \max\left\{C^2, \frac{D}{\mu}\right\}.\qedhere\]
\end{proof}

\section{Some technical proofs}\label{APP::TechnicalProofs}
In this appendix we collect the proofs of some of our more technical results.

\begin{proof}[Proof of Lemma~\ref{LEM::PiecewiseRegularityToGlobal}]\prooflabel{PRF::PiecewiseRegularityToGlobal}{Proof of Lemma~\ref{LEM::PiecewiseRegularityToGlobal}}
    Let~$\delta > 0$, and~$x,x^\prime \in I_\delta(\Omega_\tau)$, with~$x \neq x^\prime$. Consider the following two cases: either~$\abs{x-x^\prime} < \frac{\delta}{2}$, or~$\abs{x-x^\prime} \geq \frac{\delta}{2}$.

    If~$\abs{x-x^\prime} < \frac{\delta}{2}$, then
    \begin{equation}\label{EQ::SegmentIsInOmegaDelta}
        E \coloneq \bigg\{(1-\mu)x + \mu x^\prime \colon \mu \in (0,1)\bigg\} \subset \overline{B}_{\frac{\delta}{2}}(x) \subset I_{\frac{\delta}{2}}(\Omega_\tau),
    \end{equation}
    therefore, the~$n$-th Fermi coordinate (i.e., the second component of~$\Phi^{-1}$) is well defined and continuous on~$E$.

    From~\eqref{EQ::VkDefinition}, we see that
    \begin{equation}\label{EQ::OmegaDeltaVkPartition}
        I_{\frac{\delta}{2}}(\Omega_\tau) = \bigcup_{m \in \{-\mathfrak{K}, \dots, \mathfrak{K}\}}I_{\frac{\delta}{2}}(V_m),
    \end{equation}
    therefore, there exist~$m_x$ and~$m_{x^\prime}$ such that~$x \in I_{\frac{\delta}{2}}(V_{m_x})$ and~$x^\prime \in I_{\frac{\delta}{2}}(V_{m_{x^\prime}})$. Without loss of generality, we assume that~$m_x \leq m_{x^\prime}$.

    If~$m_x = m_{x^\prime}$, then~$x,x^\prime \in I_{\frac{\delta}{2}}(V_{m_x})$, thus
    \begin{equation}\label{EQ::xxPrimeNearSameVk}
        \abs{u(x)-u(x^\prime)} \leq \abs{u}_{a; I_{\frac{\delta}{2}}(V_{m_x})} \abs{x-x^\prime}^a \leq \max_{m \in \{-\mathfrak{K},\dots,\mathfrak{K}\}} \abs{u}_{a; I_{\frac{\delta}{2}}(V_m)}\abs{x-x^\prime}^a.
    \end{equation}

    If, instead,~$m_x < m_{x^\prime}$, we let
    \begin{equation}\label{EQ::KdefinitionAndBound}
        K \coloneq m_{x^\prime} - m_x \leq 2\mathfrak{K} \leq \frac{L}{\eps}
    \end{equation}
    and~${\{\mu_j\}}_{j=1}^K$ be such that
    \begin{equation}\label{EQ::mujSubsegmentDefinition}
        \mu_j \coloneq \inf\left\{\mu \in [0,1] \colon (1-\mu)x + \mu x^\prime \in I_{\frac{\delta}{2}}(V_{m_x + j})\right\}.
    \end{equation}
    Each~$\mu_j$ is well defined because of the continuity of the~$n$-th Fermi coordinate and because, thanks to~\eqref{EQ::SegmentIsInOmegaDelta}, the segment~$E$ is fully contained in~$I_{\frac{\delta}{2}}(\Omega_\tau)$. Hence, letting~$x_j \coloneq (1-\mu_j)x + \mu_j x^\prime$ for each~$j \in \{1,\dots,m_{x^\prime}-m_x\}$, we have
    \begin{equation}
        x_j \in I_{\frac{\delta}{2}}(V_{m_x + j - 1})\cap I_{\frac{\delta}{2}}(V_{m_x + j}),\qquad \text{for every }j \in \{1,\dots,m_{x^\prime}-m_x\}.
    \end{equation}
    Thus,
    \begin{equation}
        \begin{split}
            |u&(x)-u(x^\prime)| \leq \abs{u(x)-u(x_1)} + \sum_{j=1}^{K-1} \abs{u(x_j)-u(x_{j+1})} + \abs{u(x_K)-u(x^\prime)}\\
                &\leq \abs{u}_{a; I_{\frac{\delta}{2}}(V_{m_x})} \abs{x-x_1}^a + \sum_{j=1}^{K-1} \left( \abs{u}_{a; I_{\frac{\delta}{2}}(V_{m_x + j})} \abs{x_j-x_{j+1}}^a\right)+\abs{u}_{a; I_{\frac{\delta}{2}}(V_{m_{x^\prime}})} \abs{x_K - x^\prime}^a\\
                &\leq \max_{m \in \{-\mathfrak{K},\dots,\mathfrak{K}\}} \abs{u}_{a; I_{\frac{\delta}{2}}(V_m)} \left( \mu_1^a + \sum_{j=1}^{K-1} {(\mu_{j+1} - \mu_j)}^a + {(1-\mu_K)}^a\right)\abs{x-x^\prime}^a.
        \end{split}
    \end{equation}
    Hence, using the Jensen inequality and~\eqref{EQ::KdefinitionAndBound} we have that
    \begin{equation}
        \begin{split}
            |u&(x)-u(x^\prime)| \leq {(K+1)}^{1-a} \max_{m \in \{-\mathfrak{K},\dots,\mathfrak{K}\}} \abs{u}_{a; I_{\frac{\delta}{2}}(V_m)}\abs{x-x^\prime}^a\\
                &\leq {\left(\frac{2L}{\eps}\right)}^{1-a} \max_{m \in \{-\mathfrak{K},\dots,\mathfrak{K}\}} \abs{u}_{a; I_{\frac{\delta}{2}}(V_m)}\abs{x-x^\prime}^a.
        \end{split}
    \end{equation}

    Thanks to this and~\eqref{EQ::xxPrimeNearSameVk}, we know that, for~$\abs{x-x^\prime} < \frac{\delta}{2}$,
    \begin{equation}\label{EQ::xxPrimeNearGeneral}
        \abs{u(x)-u(x^\prime)} \leq {\left(\frac{2L}{\eps}\right)}^{1-a}\max_{m \in \{-\mathfrak{K},\dots,\mathfrak{K}\}} \abs{u}_{a; I_{\frac{\delta}{2}}(V_m)}\abs{x-x^\prime}^a.
    \end{equation}

    We now focus on the case~$\abs{x-x^\prime} \geq \frac{\delta}{2}$. We first observe that, due to~\eqref{EQ::OmegaDeltaVkPartition},
    \begin{equation}\label{EQ::PiecewiseSupToGlobal}
        \abs{u}_{0;I_\delta(\Omega_\tau)} = \max_{m \in \{-\mathfrak{K},\dots,\mathfrak{K}\}}\abs{u}_{0;I_\delta(V_m)}.
    \end{equation}
    Then, if~$b \geq 0$, we compute
    \begin{align}\label{EQ::xxprimeFarPositiveB}
        \abs{u(x)-u(x^\prime)} &\leq \abs{u(x)}+\abs{u(x^\prime)} \leq 2\abs{u}_{0;I_\delta(\Omega_\tau)}\\
            &\leq 2\max_{m \in \{-\mathfrak{K},\dots,\mathfrak{K}\}}\abs{u}_{0;I_\delta(V_m)}{\left(\frac{2\abs{x-x^\prime}}{\delta}\right)}^a\\
            &\leq 2^{a+1}{\delta^{-a}}\max_{m \in \{-\mathfrak{K},\dots,\mathfrak{K}\}}\abs{u}_{0;I_\delta(V_m)}\abs{x-x^\prime}^a.
    \end{align}

    Combining this,~\eqref{EQ::xxPrimeNearGeneral}, and~\eqref{EQ::PiecewiseSupToGlobal}, we find
    \begin{equation}\label{EQ::PiecewiseToGlobalDelta}
        \abs{u}_{a; I_\delta(\Omega_\tau)} \leq {\left(\frac{2L}{\eps}\right)}^{1-a}\max_{m \in \{-\mathfrak{K},\dots,\mathfrak{K}\}} \abs{u}_{a; I_{\frac{\delta}{2}}(V_m)} + (1+2^{a+1}{\delta^{-a}})\max_{m \in \{-\mathfrak{K},\dots,\mathfrak{K}\}}\abs{u}_{0;I_\delta(V_m)}.
    \end{equation}

    We use this information to discover that
    \begin{equation}
        \begin{split}
            \abs{u}_{a; \Omega_\tau}^{(b)} &= \sup_{\delta > 0}\delta^{a+b} \abs{u}_{a; I_\delta(\Omega_\tau)}\\
                &\leq \sup_{\delta > 0}\delta^{a+b}\left({\left(\frac{2L}{\eps}\right)}^{1-a}\max_{m \in \{-\mathfrak{K},\dots,\mathfrak{K}\}} \abs{u}_{a; I_{\frac{\delta}{2}}(V_m)} + (1+2^{a+1}{\delta^{-a}})\max_{m \in \{-\mathfrak{K},\dots,\mathfrak{K}\}}\abs{u}_{0;I_\delta(V_m)}\right)\\
                &= {\left(\frac{2L}{\eps}\right)}^{1-a}\max_{m \in \{-\mathfrak{K},\dots,\mathfrak{K}\}}\sup_{\delta > 0}\delta^{a+b}\abs{u}_{a; I_{\frac{\delta}{2}}(V_m)} +\max_{m \in \{-\mathfrak{K},\dots,\mathfrak{K}\}} \sup_{\delta > 0}\delta^{a+b}\abs{u}_{0;I_\delta(V_m)}\\
                &\qquad + 2^{a+1}\max_{m \in \{-\mathfrak{K},\dots,\mathfrak{K}\}} \sup_{\delta > 0}\delta^{b}\abs{u}_{0;I_\delta(V_m)}.
        \end{split}
    \end{equation}
    Here, we explicitly use the fact that~$b \geq 0$, so that the space~$\mathcal{H}_0^{(b)}$ is nontrivial, and its norm is well defined, obtaining
    \begin{equation}\label{EQ::PiecewiseRegularityToGlobalLongComputation}
        \begin{split}
                \abs{u}_{a; \Omega_\tau}^{(b)}&\leq 2^{b}{\left(\frac{L}{\eps}\right)}^{1-a}\max_{m \in \{-\mathfrak{K},\dots,\mathfrak{K}\}}\sup_{\delta > 0}\delta^{a+b}\abs{u}_{a; I_\delta(V_m)}\\
                &\qquad + {\left(\operatorname{diam}(\Omega_L)\right)}^a\max_{m \in \{-\mathfrak{K},\dots,\mathfrak{K}\}} \sup_{\delta \in (0,\operatorname{diam}(\Omega_L))}\delta^{b}\abs{u}_{0;I_\delta(V_m)}+2^{a+1}\max_{m \in \{-\mathfrak{K},\dots,\mathfrak{K}\}} \abs{u}_{0;V_m}^{(b)}\\
                &\leq \left(2^{b}{\left(\frac{L}{\eps}\right)}^{1-a} + \constapp{CONST::HolderEmbeddingA}\left({\left(\operatorname{diam}(\Omega_L)\right)}^a + 2^{a+1}\right)\right)\max_{m \in \{-\mathfrak{K},\dots,\mathfrak{K}\}} \abs{u}_{a;V_m}^{(b)},
        \end{split}
    \end{equation}
    where in the last inequality we also used Proposition~\ref{PROP::HolderEmbeddingA} with the choice~$a \coloneq a$,~$a^\prime \coloneq 0$, and~$b \coloneq b$.

    If~$b<0$, instead, we let~$\beta \coloneq -b \in (0,a]$, so that~$u \in \mathcal{H}_a^{(-\beta)}(V_m)$ for every $m$. Thanks to Proposition~\ref{PROP::HolderEmbeddingA} (applied here with~$a \coloneq a$,~$b \coloneq -\beta$, and~$a^\prime \coloneq \beta$), we know that~$u \in \mathcal{H}_{\beta}^{(-\beta)}(V_m) = \mathcal{H}_{\beta}(V_m)$ for every~$m \in \{-\mathfrak{K},\dots,\mathfrak{K}\}$. Notably, this means that~$u$ is~$\beta$-Hölder continuous up to the Dirichlet boundary in each~$V_m$. 

    We let~$\gamma$ be the curve joining~$x$ and~$x^\prime$ given by Lemma~\ref{LEM::Quasiconvexity} (namely, we have that~$\gamma(0) = x$ and~$\gamma(1) = x^\prime$). We assume that~$m_x < m_x^\prime$, otherwise~\eqref{EQ::xxPrimeNearSameVk} holds. Then, we let
    \begin{equation}
        \sigma_j \coloneq \inf\left\{ \sigma \in [0,1] \colon \gamma(\sigma) \in V_{m_x+j} \right\},
    \end{equation}
    for every~$j \in \{1,\dots,m_{x^\prime}-m_x\}$. As with~\eqref{EQ::mujSubsegmentDefinition}, the values~$\sigma_j$ are well defined and such that
    \begin{equation}
        y_j \coloneq \gamma(\sigma_j)\in I_{\frac{\delta}{2}}(V_{m_x + j - 1})\cap I_{\frac{\delta}{2}}(V_{m_x + j}),\qquad \text{for every }j \in \{1,\dots,m_{x^\prime}-m_x\}.
    \end{equation}
    Thus,
    \begin{equation}
        \begin{split}
        |u(&x)-u(x^\prime)|\leq \abs{u(x)-u(y_1)} + \sum_{j=1}^{K-1} \abs{u(y_j)-u(y_{j+1})} + \abs{u(y_K)-u(x^\prime)}\\
            &\leq \abs{u}_{\beta; V_{m_x}} \abs{x-y_1}^{\beta} + \sum_{j=1}^{K-1} \left( \abs{u}_{\beta; V_{m_x+j}} \abs{y_j-y_{j+1}}^{\beta}\right) +\abs{u}_{\beta; V_{m_{x^\prime}}} \abs{y_K - x^\prime}^{\beta}\\
            &\leq \max_{m \in \{-\mathfrak{K},\dots,\mathfrak{K}\}} \abs{u}_{\beta; V_{m}} \left( {\left(\int_0^{\sigma_1}\abs{\dot{\gamma}(t)}\,dt\right)}^{\beta}+ \sum_{j=1}^{K-1} {\left(\int_{\sigma_j}^{\sigma_{j+1}}\abs{\dot{\gamma}(t)}\,dt\right)}^{\beta} + {\left(\int_{\sigma_K}^{1}\abs{\dot{\gamma}(t)}\,dt\right)}^{\beta}\right).
        \end{split}
    \end{equation}
    We now apply the Jensen inequality,~\eqref{EQ::KdefinitionAndBound}, and Proposition~\ref{PROP::HolderEmbeddingA} (with~$a \coloneq a$,~$b \coloneq b$, and~$a^\prime \coloneq \beta$), finding
    \begin{equation}
        \begin{split}
            |u(x)-u(x^\prime)|&\leq \max_{m \in \{-\mathfrak{K},\dots,\mathfrak{K}\}} \abs{u}_{\beta; V_{m}} {(K+1)}^{1-\beta} {(\operatorname{length}(\gamma))}^{\beta}\\
                &\leq {\left(\frac{2L}{\eps}\right)}^{1-\beta}\constapp{CONST::HolderEmbeddingA}\max_{m \in \{-\mathfrak{K},\dots,\mathfrak{K}\}} \abs{u}_{a; V_{m}}^{(-\beta)}{(\operatorname{length}(\gamma))}^{\beta},
        \end{split}
    \end{equation}
    which, thanks to~\eqref{EQ::QuasiconvexityInequality} and~$\abs{x-x^\prime} > \frac{\delta}{2}$, gives that
    \begin{equation}\label{EQ::xxprimeFarNegativeB}
        \begin{split}
            |u(x)-u(x^\prime)| &\leq {\left(\frac{2L}{\eps}\right)}^{1-\beta}\constapp{CONST::HolderEmbeddingA}\constapp{CONST::Quasiconvexity}^\beta \max_{m \in \{-\mathfrak{K},\dots,\mathfrak{K}\}}\abs{u}_{a; V_{m}}^{(-\beta)}\abs{x-x^\prime}^\beta\\
                &\leq 2^{1+a-2\beta}{\left(\frac{L}{\eps}\right)}^{1-\beta}\constapp{CONST::HolderEmbeddingA}\constapp{CONST::Quasiconvexity}^\beta \max_{m \in \{-\mathfrak{K},\dots,\mathfrak{K}\}}\abs{u}_{a; V_{m}}^{(-\beta)}\abs{x-x^\prime}^a \delta^{-a+\beta}.
        \end{split}
    \end{equation}
    This, together with~\eqref{EQ::xxPrimeNearGeneral} and~\eqref{EQ::PiecewiseSupToGlobal}, shows that
    \begin{equation}\begin{split}
            \abs{u}_{a; I_\delta(\Omega_\tau)} \leq &{\left(\frac{2L}{\eps}\right)}^{1-a}\max_{m \in \{-\mathfrak{K},\dots,\mathfrak{K}\}} \abs{u}_{a; I_{\frac{\delta}{2}}(V_m)} + \max_{m \in \{-\mathfrak{K},\dots,\mathfrak{K}\}}\abs{u}_{0;I_\delta(V_m)}\\
                &+ 2^{1+a-2\beta}{\left(\frac{L}{\eps}\right)}^{1-\beta}\constapp{CONST::HolderEmbeddingA}\constapp{CONST::Quasiconvexity}^\beta \max_{m \in \{-\mathfrak{K},\dots,\mathfrak{K}\}}\abs{u}_{a; V_{m}}^{(-\beta)}\delta^{-a+\beta}.
        \end{split}
    \end{equation}
    Finally,
    \begin{equation}\label{EQ::HolderStitchingBeta}
        \begin{split}
            \abs{u}_{a; \Omega_\tau}^{(-\beta)} &= \sup_{\delta > 0}\delta^{a-\beta} \abs{u}_{a; I_\delta(\Omega_\tau)}\\
            &\leq \sup_{\delta > 0}\delta^{a-\beta}\left({\left(\frac{2L}{\eps}\right)}^{1-a}\max_{m \in \{-\mathfrak{K},\dots,\mathfrak{K}\}} \abs{u}_{a; I_{\frac{\delta}{2}}(V_m)}+ \max_{m \in \{-\mathfrak{K},\dots,\mathfrak{K}\}}\abs{u}_{0;I_\delta(V_m)}\right.\\
            &\qquad \left. +2^{1+a-2\beta}{\left(\frac{L}{\eps}\right)}^{1-\beta}\constapp{CONST::HolderEmbeddingA}\constapp{CONST::Quasiconvexity}^\beta\max_{m \in \{-\mathfrak{K},\dots,\mathfrak{K}\}}\abs{u}_{a; V_{m}}^{(-\beta)}\delta^{-a+\beta} \right)\\
            &\leq 2^{1-\beta} {\left(\frac{L}{\eps}\right)}^{1-a} \max_{m \in \{-\mathfrak{K},\dots,\mathfrak{K}\}}\abs{u}_{a;V_m}^{(-\beta)} + {(\operatorname{diam}(\Omega_L))}^{a-\beta}\max_{m \in \{-\mathfrak{K},\dots,\mathfrak{K}\}}\abs{u}_{0;V_m}\\
            &\qquad + 2^{1+a-2\beta}{\left(\frac{L}{\eps}\right)}^{1-\beta}\constapp{CONST::HolderEmbeddingA}\constapp{CONST::Quasiconvexity}^\beta {(\operatorname{diam}(\Omega_L))}^{a-\beta}\max_{m \in \{-\mathfrak{K},\dots,\mathfrak{K}\}}\abs{u}_{a; V_{m}}^{(-\beta)}\\
            &\leq \left(2^{1-\beta} + {(\operatorname{diam}(\Omega_L))}^{a-\beta} \left(C + 2^{1+a-2\beta}\constapp{CONST::HolderEmbeddingA}\constapp{CONST::Quasiconvexity}^\beta\right)\right){\left(\frac{L}{\eps}\right)}^{1-\beta}\max_{m \in \{-\mathfrak{K},\dots,\mathfrak{K}\}}\abs{u}_{a;V_m}^{(-\beta)}.
        \end{split}
    \end{equation}

    Recalling that~$\beta = -b$ and~$\vartheta = 1+b = 1-\beta$, we substitute into~\eqref{EQ::HolderStitchingBeta} and find
    \begin{equation}
        \abs{u}_{a; \Omega_\tau}^{(b)}\leq \left(2^{1+b} + {(\operatorname{diam}(\Omega_L))}^{a+b} \left(C + 2^{1+a+2b}\constapp{CONST::HolderEmbeddingA}\constapp{CONST::Quasiconvexity}^{-b}\right)\right){\left(\frac{L}{\eps}\right)}^{\vartheta}\max_{m \in \{-\mathfrak{K},\dots,\mathfrak{K}\}}\abs{u}_{a;V_m}^{(b)}.
    \end{equation}
    From this and~\eqref{EQ::PiecewiseRegularityToGlobalLongComputation}, we conclude that~\eqref{EQ::PiecewiseRegularityToGlobal} holds with
    \[\const{CONST::PiecewiseRegularityToGlobal} \coloneq \left(2^{b}L^{1-a} + \constapp{CONST::HolderEmbeddingA}\left({\left(\operatorname{diam}(\Omega_L)\right)}^a + 2^{a+1}\right)\right)\]
    if~$b \geq 0$, and
    \[\const{CONST::PiecewiseRegularityToGlobal} \coloneq \left(2^{1+b} + {(\operatorname{diam}(\Omega_L))}^{a+b} \left(C + 2^{1+a+2b}\constapp{CONST::HolderEmbeddingA}\constapp{CONST::Quasiconvexity}^{-b}\right)\right) L^{\vartheta}\]
    if~$b<0$.
\end{proof}

\begin{proof}[Proof of Lemma~\ref{LEM::PiecewiseLipschitzToGlobalLipschitz}]\prooflabel{PRF::PiecewiseLipschitzToGlobalLipschitz}{Proof of Lemma~\ref{LEM::PiecewiseLipschitzToGlobalLipschitz}}
    Let~$\delta > 0$, and~$x,x^\prime \in I_\delta(\Omega_\tau)$, with~$x \neq x^\prime$. Consider the following two cases: either~$\abs{x-x^\prime} < \frac{\delta}{2}$, or~$\abs{x-x^\prime} \geq \frac{\delta}{2}$.

    If~$\abs{x-x^\prime} < \frac{\delta}{2}$, then
    \begin{equation}\label{EQ::SegmentIsInOmegaDeltaLipschitz}
        E \coloneq \bigg\{(1-\mu)x + \mu x^\prime \colon \mu \in (0,1)\bigg\} \subset \overline{B}_{\frac{\delta}{2}}(x) \subset I_{\frac{\delta}{2}}(\Omega_\tau),
    \end{equation}
    therefore, the~$n$-th Fermi coordinate (i.e., the second component of~$\Phi^{-1}$) is well defined and continuous on~$E$.

    From~\eqref{EQ::VkDefinition}, we see that
    \begin{equation}\label{EQ::OmegaDeltaVkPartitionLipschitz}
        I_{\frac{\delta}{2}}(\Omega_\tau) = \bigcup_{m \in \{-\mathfrak{K}, \dots, \mathfrak{K}\}}I_{\frac{\delta}{2}}(V_m),
    \end{equation}
    therefore, there exist~$m_x$ and~$m_{x^\prime}$ such that~$x \in I_{\frac{\delta}{2}}(V_{m_x})$ and~$x^\prime \in I_{\frac{\delta}{2}}(V_{m_{x^\prime}})$. Without loss of generality, we assume that~$m_x \leq m_{x^\prime}$.

    If~$m_x = m_{x^\prime}$, then~$x,x^\prime \in I_{\frac{\delta}{2}}(V_{m_x})$, thus
    \begin{equation}\label{EQ::xxPrimeNearSameVkLipschitz}
        \abs{u(x)-u(x^\prime)} \leq \abs{u}_{1; I_{\frac{\delta}{2}}(V_{m_x})} \abs{x-x^\prime} \leq \max_{m \in \{-\mathfrak{K},\dots,\mathfrak{K}\}} \abs{u}_{1; I_{\frac{\delta}{2}}(V_m)}\abs{x-x^\prime}.
    \end{equation}

    If, instead,~$m_x < m_{x^\prime}$, we let
    \begin{equation}\label{OLD::EQ::KdefinitionAndBound}
        K \coloneq m_{x^\prime} - m_x \leq 2\mathfrak{K} \leq \frac{L}{\eps}
    \end{equation}
    and~${\{\mu_j\}}_{j=1}^K$ be such that
    \begin{equation}\label{EQ::mujSubsegmentDefinitionLipschitz}
        \mu_j \coloneq \inf\left\{\mu \in [0,1] \colon (1-\mu)x + \mu x^\prime \in I_{\frac{\delta}{2}}(V_{m_x + j})\right\}.
    \end{equation}
    Each~$\mu_j$ is well defined because of the continuity of the~$n$-th Fermi coordinate and because, thanks to~\eqref{EQ::SegmentIsInOmegaDeltaLipschitz}, the segment~$E$ is fully contained in~$I_{\frac{\delta}{2}}(\Omega_\tau)$. Hence, letting~$x_j \coloneq (1-\mu_j)x + \mu_j x^\prime$ for each~$j \in \{1,\dots,m_{x^\prime}-m_x\}$, we have
    \begin{equation}
        x_j \in I_{\frac{\delta}{2}}(V_{m_x + j - 1})\cap I_{\frac{\delta}{2}}(V_{m_x + j}),\qquad \text{for every }j \in \{1,\dots,m_{x^\prime}-m_x\}.
    \end{equation}
    Thus,
    \begin{equation}
        \begin{split}
            |u&(x)-u(x^\prime)| \leq \abs{u(x)-u(x_1)} + \sum_{j=1}^{K-1} \abs{u(x_j)-u(x_{j+1})} + \abs{u(x_K)-u(x^\prime)}\\
                &\leq \abs{u}_{1; I_{\frac{\delta}{2}}(V_{m_x})} \abs{x-x_1} + \sum_{j=1}^{K-1} \left( \abs{u}_{1; I_{\frac{\delta}{2}}(V_{m_x + j})} \abs{x_j-x_{j+1}}\right)+\abs{u}_{1; I_{\frac{\delta}{2}}(V_{m_{x^\prime}})} \abs{x_K - x^\prime}\\
                &\leq \max_{m \in \{-\mathfrak{K},\dots,\mathfrak{K}\}} \abs{u}_{1; I_{\frac{\delta}{2}}(V_m)} \left( \mu_1 + \sum_{j=1}^{K-1} {(\mu_{j+1}-\mu_j)} + {(1-\mu_K)}\right)\abs{x-x^\prime}.
        \end{split}
    \end{equation}
    Hence, using~\eqref{EQ::mujSubsegmentDefinitionLipschitz} we have that
    \begin{equation}\label{EQ::xxPrimeNearGeneralLipschitz}
            |u(x)-u(x^\prime)| \leq \max_{m \in \{-\mathfrak{K},\dots,\mathfrak{K}\}} \abs{u}_{1; I_{\frac{\delta}{2}}(V_m)}\abs{x-x^\prime}.
    \end{equation}
    Thanks to~\eqref{EQ::xxPrimeNearSameVkLipschitz}, we know that~\eqref{EQ::xxPrimeNearGeneralLipschitz} holds true whenever~$\abs{x-x^\prime} < \frac{\delta}{2}$.

    We now focus on the case~$\abs{x-x^\prime} \geq \frac{\delta}{2}$. We first observe that, due to~\eqref{EQ::OmegaDeltaVkPartitionLipschitz},
    \begin{equation}\label{EQ::PiecewiseSupToGlobalLipschitz}
        \abs{u}_{0;I_\delta(\Omega_\tau)} = \max_{m \in \{-\mathfrak{K},\dots,\mathfrak{K}\}}\abs{u}_{0;I_\delta(V_m)}.
    \end{equation}
    Then, if~$b \geq 0$, we compute
    \begin{align}\label{OLD::EQ::xxprimeFarPositiveB}
        \abs{u(x)-u(x^\prime)} &\leq \abs{u(x)}+\abs{u(x^\prime)} \leq 2\abs{u}_{0;I_\delta(\Omega_\tau)}\\
            &\leq 2\max_{m \in \{-\mathfrak{K},\dots,\mathfrak{K}\}}\abs{u}_{0;I_\delta(V_m)}{\left(\frac{2\abs{x-x^\prime}}{\delta}\right)}\\
            &\leq 4{\delta^{-1}}\max_{m \in \{-\mathfrak{K},\dots,\mathfrak{K}\}}\abs{u}_{0;I_\delta(V_m)}\abs{x-x^\prime}.
    \end{align}

    Combining this and~\eqref{EQ::xxPrimeNearGeneralLipschitz} we find
    \begin{equation}\label{EQ::PiecewiseToGlobalDeltaLipschitz}
        \abs{u(x)-u(x^\prime)} \leq \left(\max_{m \in \{-\mathfrak{K},\dots,\mathfrak{K}\}} \abs{u}_{1; I_{\frac{\delta}{2}}(V_m)} + 4{\delta^{-1}}\max_{m \in \{-\mathfrak{K},\dots,\mathfrak{K}\}}\abs{u}_{0;I_\delta(V_m)}\right) \abs{x-x^\prime}.
    \end{equation}
    We now let~$a \in (0,1)$ and observe that
    \begin{equation}\label{EQ::LipschitzToHolderDistance}
        \abs{x-x^\prime} \leq {\operatorname{diam}(\Omega_L)}^{1-a} \abs{x-x^\prime}^a \leq \max\{1,\operatorname{diam}(\Omega_L)\} \abs{x-x^\prime}^a,
    \end{equation}
    hence, using also~\eqref{EQ::PiecewiseSupToGlobalLipschitz} and~\eqref{EQ::PiecewiseToGlobalDeltaLipschitz} we have that
    \begin{equation}
        \abs{u}_{a; I_{\frac{\delta}{2}}(\Omega_\tau)}^{(b)} \leq \max\{1,\operatorname{diam}(\Omega_L)\} \left(\max_{m \in \{-\mathfrak{K},\dots,\mathfrak{K}\}} \abs{u}_{1; I_{\frac{\delta}{2}}(V_m)} + (1+4{\delta^{-1}})\max_{m \in \{-\mathfrak{K},\dots,\mathfrak{K}\}}\abs{u}_{0;I_\delta(V_m)}\right).
    \end{equation}

    We use this information to discover that
    \begin{equation}\label{EQ::LongComputationBis}
        \begin{split}
            \abs{u}_{a; \Omega_\tau}^{(b)} &= \sup_{\delta > 0}\delta^{1+b} \abs{u}_{1; I_\delta(\Omega_\tau)}\\
                &\leq \max\{1,\operatorname{diam}(\Omega_L)\} \sup_{\delta > 0}\delta^{1+b}\left(\max_{m \in \{-\mathfrak{K},\dots,\mathfrak{K}\}} \abs{u}_{1; I_{\frac{\delta}{2}}(V_m)} + (1+4{\delta^{-1}})\max_{m \in \{-\mathfrak{K},\dots,\mathfrak{K}\}}\abs{u}_{0;I_\delta(V_m)}\right)\\
                &= \max\{1,\operatorname{diam}(\Omega_L)\} \left(\max_{m \in \{-\mathfrak{K},\dots,\mathfrak{K}\}}\sup_{\delta > 0}\delta^{1+b}\abs{u}_{1; I_{\frac{\delta}{2}}(V_m)} +\max_{m \in \{-\mathfrak{K},\dots,\mathfrak{K}\}} \sup_{\delta > 0}\delta^{1+b}\abs{u}_{0;I_\delta(V_m)}\right.\\
                &\qquad \left.+ 4\max_{m \in \{-\mathfrak{K},\dots,\mathfrak{K}\}} \sup_{\delta > 0}\delta^{b}\abs{u}_{0;I_\delta(V_m)}\right).
        \end{split}
    \end{equation}
    Here, we explicitly use the fact that~$b \geq 0$, so that the space~$\mathcal{H}_0^{(b)}$ is nontrivial, and its norm is well defined, obtaining
    \begin{equation}\label{EQ::PiecewiseRegularityToGlobalLongComputationLipschitz}
        \begin{split}
                \abs{u}_{a; \Omega_\tau}^{(b)}&\leq \max\{1,\operatorname{diam}(\Omega_L)\} \left(2^{1+b}\max_{m \in \{-\mathfrak{K},\dots,\mathfrak{K}\}}\sup_{\delta > 0}\delta^{1+b}\abs{u}_{1; I_\delta(V_m)}\right.\\
                &\qquad \left.+ \operatorname{diam}(\Omega_L)\max_{m \in \{-\mathfrak{K},\dots,\mathfrak{K}\}} \sup_{\delta \in (0,\operatorname{diam}(\Omega_L))}\delta^{b}\abs{u}_{0;I_\delta(V_m)}+4\max_{m \in \{-\mathfrak{K},\dots,\mathfrak{K}\}} \abs{u}_{0;V_m}^{(b)}\right)\\
                &\leq \max\{1,\operatorname{diam}(\Omega_L)\} \left(2^{1+b} + \constapp{CONST::HolderEmbeddingA}\left(\operatorname{diam}(\Omega_L) + 4\right)\right)\max_{m \in \{-\mathfrak{K},\dots,\mathfrak{K}\}} \abs{u}_{1;V_m}^{(b)},
        \end{split}
    \end{equation}
    where in the last inequality we also used Proposition~\ref{PROP::HolderEmbeddingA} with the choice~$a \coloneq 1$,~$a^\prime \coloneq 0$, and~$b \coloneq b$.

    If~$b=-1$, we have that~$u \in \mathcal{H}_{1}^{(-1)}(V_m) = \mathcal{H}_{1}(V_m)$ for every~$m \in \{-\mathfrak{K},\dots,\mathfrak{K}\}$. Notably, this means that~$u$ is~$C^1$ up to the Dirichlet boundary in each~$V_m$. 

    We let~$\gamma$ be the curve joining~$x$ and~$x^\prime$ given by Lemma~\ref{LEM::Quasiconvexity} (namely, we have that~$\gamma(0) = x$ and~$\gamma(1) = x^\prime$). We assume that~$m_x < m_x^\prime$, otherwise~\eqref{EQ::xxPrimeNearSameVkLipschitz} holds. Then, we let
    \begin{equation}
        \sigma_j \coloneq \inf\left\{ \sigma \in [0,1] \colon \gamma(\sigma) \in V_{m_x+j} \right\},
    \end{equation}
    for every~$j \in \{1,\dots,m_{x^\prime}-m_x\}$. As with~\eqref{EQ::mujSubsegmentDefinitionLipschitz}, the values~$\sigma_j$ are well defined and such that
    \begin{equation}
        y_j \coloneq \gamma(\sigma_j)\in I_{\frac{\delta}{2}}(V_{m_x + j - 1})\cap I_{\frac{\delta}{2}}(V_{m_x + j}),\qquad \text{for every }j \in \{1,\dots,m_{x^\prime}-m_x\}.
    \end{equation}
    Thus,
    \begin{equation}
        \begin{split}
        |u(&x)-u(x^\prime)|\leq \abs{u(x)-u(y_1)} + \sum_{j=1}^{K-1} \abs{u(y_j)-u(y_{j+1})} + \abs{u(y_K)-u(x^\prime)}\\
            &\leq \abs{u}_{1; V_{m_x}} \abs{x-y_1} + \sum_{j=1}^{K-1} \left( \abs{u}_{1; V_{m_x+j}} \abs{y_j-y_{j+1}}\right) +\abs{u}_{1; V_{m_{x^\prime}}} \abs{y_K - x^\prime}\\
            &\leq \max_{m \in \{-\mathfrak{K},\dots,\mathfrak{K}\}} \abs{u}_{1; V_{m}} \left( {\left(\int_0^{\sigma_1}\abs{\dot{\gamma}(t)}\,dt\right)}+ \sum_{j=1}^{K-1} {\left(\int_{\sigma_j}^{\sigma_{j+1}}\abs{\dot{\gamma}(t)}\,dt\right)} + {\left(\int_{\sigma_K}^{1}\abs{\dot{\gamma}(t)}\,dt\right)}\right)\\
            & = \operatorname{length}(\gamma) \max_{m \in \{-\mathfrak{K},\dots,\mathfrak{K}\}} \abs{u}_{1; V_{m}}.
        \end{split}
    \end{equation}
    Therefore, thanks to~\eqref{EQ::QuasiconvexityInequality} and~$\abs{x-x^\prime} > \frac{\delta}{2}$, we have that
    \begin{equation}\label{OLD::EQ::xxprimeFarNegativeB}
        |u(x)-u(x^\prime)| \leq \constapp{CONST::Quasiconvexity} \max_{m \in \{-\mathfrak{K},\dots,\mathfrak{K}\}}\abs{u}_{1; V_{m}}\abs{x-x^\prime},
    \end{equation}
    which, together with~\eqref{EQ::xxPrimeNearGeneralLipschitz}, gives that
    \begin{equation}\label{EQ::LipschitzEstimateNegativeB}
        |u(x)-u(x^\prime)| \leq \left(\max_{m \in \{-\mathfrak{K},\dots,\mathfrak{K}\}} \abs{u}_{1; I_{\frac{\delta}{2}}(V_m)} + \constapp{CONST::Quasiconvexity} \max_{m \in \{-\mathfrak{K},\dots,\mathfrak{K}\}}\abs{u}_{1; V_{m}}\right)\abs{x-x^\prime}.
    \end{equation}
    We combine this new information with~\eqref{EQ::LipschitzToHolderDistance} to obtain
    \begin{equation}
        |u(x)-u(x^\prime)| \leq \max\{1,\operatorname{diam}(\Omega_L)\}\left(\max_{m \in \{-\mathfrak{K},\dots,\mathfrak{K}\}} \abs{u}_{1; I_{\frac{\delta}{2}}(V_m)} + \constapp{CONST::Quasiconvexity} \max_{m \in \{-\mathfrak{K},\dots,\mathfrak{K}\}}\abs{u}_{1; V_{m}}\right)\abs{x-x^\prime}^a.
    \end{equation}
    This and~\eqref{EQ::PiecewiseSupToGlobalLipschitz} show that
    \begin{equation}\begin{split}
            \abs{u}_{a; I_\delta(\Omega_\tau)} \leq &\max\{1,\operatorname{diam}(\Omega_L)\}\max_{m \in \{-\mathfrak{K},\dots,\mathfrak{K}\}} \abs{u}_{1; I_{\frac{\delta}{2}}(V_m)} + \max_{m \in \{-\mathfrak{K},\dots,\mathfrak{K}\}}\abs{u}_{0;I_\delta(V_m)}\\
                &+ \constapp{CONST::Quasiconvexity} \max\{1,\operatorname{diam}(\Omega_L)\}\max_{m \in \{-\mathfrak{K},\dots,\mathfrak{K}\}}\abs{u}_{1; V_{m}}.
        \end{split}
    \end{equation}
    Finally,
    \begin{equation}\label{EQ::HolderStitchingBetaLipschitz}
        \begin{split}
            \abs{u}_{a; \Omega_\tau} &= \sup_{\delta > 0} \abs{u}_{a; I_\delta(\Omega_\tau)}\\
                &\leq \sup_{\delta > 0}\left(\max\{1,\operatorname{diam}(\Omega_L)\}\max_{m \in \{-\mathfrak{K},\dots,\mathfrak{K}\}} \abs{u}_{1; I_{\frac{\delta}{2}}(V_m)} + \max_{m \in \{-\mathfrak{K},\dots,\mathfrak{K}\}}\abs{u}_{0;I_\delta(V_m)}\right.\\
                &+ \left.\constapp{CONST::Quasiconvexity} \max\{1,\operatorname{diam}(\Omega_L)\}\max_{m \in \{-\mathfrak{K},\dots,\mathfrak{K}\}}\abs{u}_{1; V_{m}}\right)\\
                &= \left(1+(1+\constapp{CONST::Quasiconvexity}) \max\{1,\operatorname{diam}(\Omega_L)\}\right)\max_{m \in \{-\mathfrak{K},\dots,\mathfrak{K}\}}\abs{u}_{1; V_{m}}.
        \end{split}
    \end{equation}

    We observe that if~$b \geq 0$, then~$\max\{-a,b\} = b$, hence~\eqref{EQ::PiecewiseRegularityToGlobalLongComputationLipschitz} gives~\eqref{EQ::PiecewiseLipschitzToGlobalHolder} with
    \[\const{CONST::PiecewiseLipschitzToGlobalLipschitz} \coloneq \max\{1,\operatorname{diam}(\Omega_L)\} \left(2^{1+b} + \constapp{CONST::HolderEmbeddingA}\left(\operatorname{diam}(\Omega_L) + 4\right)\right).\]
    If~$b = -1$, instead, we have that~$\max\{-a,b\} = -a$, and~\eqref{EQ::PiecewiseLipschitzToGlobalHolder} follows from~\eqref{EQ::HolderStitchingBetaLipschitz} upon choosing
    \[\const{CONST::PiecewiseLipschitzToGlobalLipschitz} \coloneq \left(1+(1+\constapp{CONST::Quasiconvexity}) \max\{1,\operatorname{diam}(\Omega_L)\}\right).\]

    It remains to prove~\eqref{EQ::PiecewiseLipschitzToGlobalLipschitz}. If~$b \geq 0$, from~\eqref{EQ::PiecewiseToGlobalDeltaLipschitz} we have that~$u$ is Lipschitz continuous in~$I_\delta(\Omega_\tau)$, with Lipschitz constant
    \[C_{\delta} \coloneq \left(\max_{m \in \{-\mathfrak{K},\dots,\mathfrak{K}\}} \abs{u}_{1; I_{\frac{\delta}{2}}(V_m)} + 4{\delta^{-1}}\max_{m \in \{-\mathfrak{K},\dots,\mathfrak{K}\}}\abs{u}_{0;I_\delta(V_m)}\right).\]
    Instead, if~$b = -1$, the Lipschitz continuity of~$u$ still holds in~$I_\delta(\Omega_\tau)$ as a consequence of~\eqref{EQ::LipschitzEstimateNegativeB} with
    \[C_{\delta} \coloneq \left(\max_{m \in \{-\mathfrak{K},\dots,\mathfrak{K}\}} \abs{u}_{1; I_{\frac{\delta}{2}}(V_m)} + \constapp{CONST::Quasiconvexity} \max_{m \in \{-\mathfrak{K},\dots,\mathfrak{K}\}}\abs{u}_{1; V_{m}}\right).\]

    Thanks to the assumption that~$\nabla u$ is continuous, we have that
    \begin{equation}
        \abs{\nabla u}_{0; I_\delta(\Omega_\tau)} \leq C_{\delta},
    \end{equation}
    therefore,
    \begin{equation}
        \abs{u}_{1; I_\delta(\Omega_\tau)} \leq \abs{u}_{0; I_\delta(\Omega_\tau)} + C_\delta
    \end{equation}
    and
    \begin{equation}
        \abs{u}_{1; \Omega_\tau}^{(b)} \leq \sup_{\delta > 0} \delta^{1+b} (\abs{u}_{0; I_\delta(\Omega_\tau)}+C_\delta).
    \end{equation}

    The proof is concluded upon noticing that, up to replacing~$\max\{1,\operatorname{diam}(\Omega_L)\}$ with~$1$, the computation is the same as~\eqref{EQ::LongComputationBis} and~\eqref{EQ::PiecewiseRegularityToGlobalLongComputationLipschitz} if~$b \geq 0$, and~\eqref{EQ::HolderStitchingBetaLipschitz} if~$b = -1$.
\end{proof}
\end{appendix}
{
 \hbadness=10000
 \bibliography{biblio}
}
\end{document}